\documentclass[10pt,reqno]{amsart}
\usepackage[english]{babel}
\usepackage{amsmath}
\numberwithin{equation}{section}
\usepackage{amsfonts}
\usepackage{amssymb}
\usepackage{mathrsfs}
\usepackage{amsthm}
\usepackage{appendix}
\usepackage[all]{xy}
\usepackage{graphicx}
\usepackage[usenames]{color}
\usepackage[margin=1in]{geometry}

\newcommand{\R}{{{\mathbb R}}}
\newcommand{\s}{{{\mathcal S}}}

\newcommand{\Ftil}{{{\mathcal{F}(\tilde{\Gamma})}}}
\newcommand{\F}{{{\mathcal{F}(\Gamma)}}}
\newcommand{\N}{{{\mathbb N}}}

\providecommand{\dr}[1]{\partial_{r}^{#1}}
\providecommand{\dx}[1]{\partial_{\xi}#1}
\providecommand{\da}[1]{\partial_{\eta}^{#1}}
\providecommand{\dxx}[1]{\partial_{\xi}^{2}#1}

\providecommand{\abs}[1]{\lvert#1\rvert}
\providecommand{\e}[1]{^{#1}}

\providecommand{\norm}[1]{\lVert#1\rVert}

\providecommand{\est}[1]{C\exp(\norm{\Ftil}^{2}_{L^{\infty}(S\e{1})}+\norm{\tilde{\Gamma}}^{2}_{H^{#1}(S\e{1})})}
\theoremstyle{plain}
\providecommand{\om}[1]{\partial\Omega_2\e{#1}}
\newtheorem{lem}{Lemma}[section]
\newtheorem{prop}{Proposition}[section]
\newtheorem{prop*}{Proposition}
\newtheorem{thm}{Theorem}[section]
\newtheorem{nota}{Remark}[section]

\theoremstyle{definition}

\newtheorem{innercustomgeneric}{\customgenericname}
\providecommand{\customgenericname}{}
\newcommand{\newcustomtheorem}[2]{%
  \newenvironment{#1}[1]
  {%
   \renewcommand\customgenericname{#2}%
   \renewcommand\theinnercustomgeneric{##1}%
   \innercustomgeneric
  }
  {\endinnercustomgeneric}
}

\newcustomtheorem{customprop}{Proposition}

\title[On the dynamics of thin layers of viscous flows inside another viscous fluid]{On the dynamics of thin layers of viscous flows inside another viscous fluid}
\author{Tania Pernas-Casta\~{n}o and Juan J. L. Vel\'azquez}

\begin{document}

\maketitle

\begin{abstract}
In this work we will study the dynamics of a thin layer of a viscous fluid which is embedded in the interior of another viscous fluid. The resulting
flow can be approximated by means of the solutions of a free boundary problem for the Stokes equation in which one of the unknowns is the shape of a curve which approximates the geometry of the thin layer of fluid. We also derive the equation yielding the thickness of this fluid. This model, that will be termed as the \textit{Geometric Free Boundary Problem}, will be derived using matched asymptotic expansions.
We will prove that the Geometric Free Boundary Problem is well posed and the solutions of the thickness equation are well defined (in particular they do not yield breaking of fluid layers) as long as the solutions of the Geometric Free Boundary Problem exist.  
\end{abstract}
\section{introduction}

In this paper we study the dynamics of a thin layer of a viscous fluid which
is embedded in the interior of other viscous fluid. Specifically, we are
interested in situations in which a thin layer of fluid $2$ is completely surrounded by a fluid $1$. Notice that, the fluid $2$ does not intersect  the boundaries of the container. See for instance Figure \ref{figdom}. We will restrict our analysis to two dimensional problems.

\begin{figure}[h]\label{figdom}
\center
\includegraphics[width=90mm]{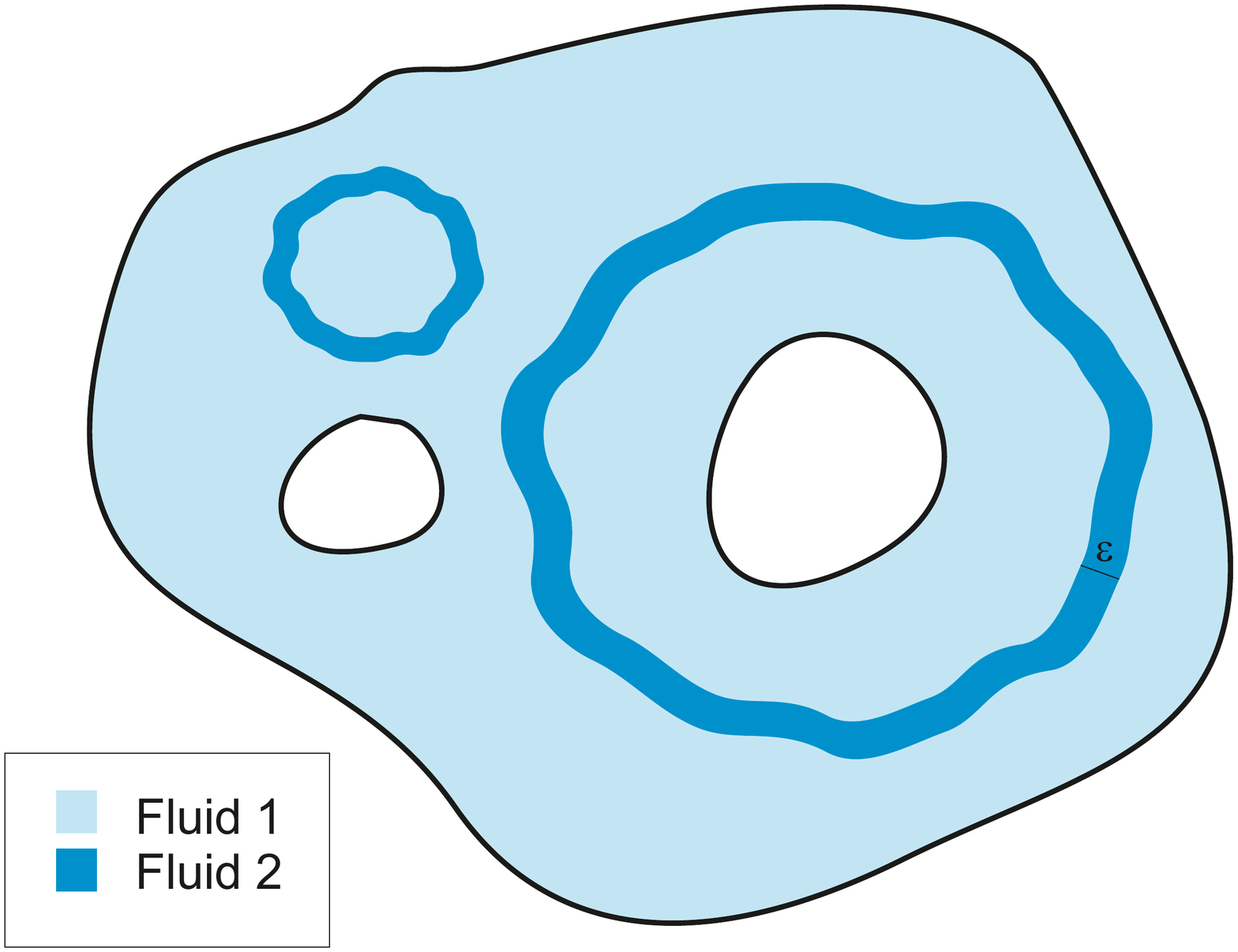}
\caption{Thin layer of a viscous fluid inside another viscous fluid}
\end{figure}

We have studied similar two fluid flows in the case in which the thin layer of fluid is in contact with one of the boundaries of the container in \cite{PV1}. In particular, we have studied the case in which the fluids are confined between two rotating coaxial cylinders, i.e. the so called Taylor-Couette geometry. It turns out that in the specific scaling limit and geometry considered in \cite{PV1}, the motion of the dominant fluid can be approximated by means of the classical one fluid Taylor-Couette flow and the dynamics of the interface which separates both fluids, can be approximated by means of a
suitable thin film like equation. Similar results for viscous flows in cylindrical geometries can be found in \cite{Taranets}.

On the contrary, we will show in this paper that, in situations like in Figure \ref{figdom}, the resulting flows can be approximated by means of the solutions of a free boundary problem for the Stokes equation in which one of the unknowns is the shape of a curve parametrized by $\Gamma(\alpha,t)$,  which approximates the geometry of the thin layer of fluid. The dynamics of the interface is dominated to the leading order by the capillary and viscous forces due to the fluid $1$. The resulting evolution equations for the interface are:
\begin{align*}
&\partial_t\Gamma=(u_0(\Gamma)\cdot n)n+\psi_0(\xi,t)\tau\\
&\psi_0(\xi,t)=\int_0\e{\xi}\kappa(r,t)u_0(\Gamma(r,t),t)\cdot n(r,t)dr
\end{align*}
where $u_0$ is the solution of the Stokes problem for some boundary conditions to the leading order when we tends the thickness of the fluid $2$ to zero. 

 We also derive an equation yielding the thickness of the layer of fluid $2$, $h(\alpha,t)$, with the form:
 
 \begin{align*}
\partial_t h(\xi,t)+(u_0(\Gamma(\xi,t))\cdot\tau(\xi,t)-\psi_0(\xi,t))\dx{h}(\xi,t)=(n(\xi,t)\nabla u_0(\Gamma(\xi,t))\cdot n(\xi,t))h(\xi,t)
\end{align*}

  This model that will be termed as the \textit{Geometric Free Boundary Problem} will be derived using matched asymptotic expansions. We will also prove in this paper local well posedness for the Geometric Free Boundary Problem. The well posedness of a free boundary problem for two viscous fluids where the velocity is given by the Navier-Stokes equation, has been extensively studied using tools of the Semi-group Theory (c.f. \cite{PrussNS}, \cite{PrussSimonet},\cite{Shibata} and \cite{Solonnik}). That problem has analogies with our Geometric Free Boundary Problem. The main difference is that we describe our problem by means of the Stokes equation instead of Navier-Stokes equation. The techniques used in this paper is based on the fact that the velocity field con be written in terms of the geometry of the curve by some operators that can be estimated using the Theory of Singular Integrals. This idea has been widely used in many problems, including water waves (c.f. \cite{interface_evol_water}), Muskat problem (c.f. \cite{granero_h2}, \cite{interface_evol_muskat}, \cite{pernas_inhomogeneous}), Surface Quasi-Geostrophic equation (c.f. \cite{gancedo_patch}), etc.

The plan of the paper is the following:

In Section \ref{S2} we explain in detail the specific two fluids flows that we consider and we collect some geometrical results that will be used later. In Section \ref{S3}, we derive the Geometric Free Boundary Problem using matched asymptotic expansions. We devote Section \ref{S4} to the study of the local well posedness for the Geometric Free Boundary Problem and the positivity of the function $h$ as long as the solutions exist. 

\section{Description and preliminaries of the problem}\label{S2}

Let be $\bar{D}\subset\R\e{2}$ a compact set with boundary $\partial D\in C\e{2}$. $D$ is completely filled with two fluids $1$ and $2$  which are immiscible and incompressible. We will assume that the fluid $2$ has viscosity $\mu_2$ and fills a thin open domain $\Omega_2(t)$. The meaning of thin domain will be made precise later. We will suppose also that the boundary of the domain $\Omega_2(t)$ does not intersects with $\partial D$.  The techniques used in this paper can be adapted to more general domains, as those in Figure \ref{figdom}, but in order to simplify the notation, we just consider the case in which the domain $D$ is connected. Moreover, we consider that the fluid $1$ has viscosity $\mu$ and it fills the disjoint open domains $\Omega_1(t)$ and $\Omega_3(t)$. Therefore (c.f. Figure \ref{figdom2}),
 $$D=Int(\overline{\Omega_1}(t)\cup\overline{\Omega_2}(t)\cup\overline{\Omega_3}(t))\quad\textit{with}\quad\Omega_j(t)\cap\Omega_l(t)=\emptyset\quad\textit{for}\quad j\neq l\quad\textit{and}\quad j,l=1,2,3$$
On the other hand, we will denote the non-connected open set  
\begin{equation}\label{omega}
\Omega(t)=\Omega_1(t)\cup\Omega_2(t)\cup\Omega_3(t)
\end{equation} 
  
\begin{figure}[h]\label{figdom2}
\center
\includegraphics[width=90mm]{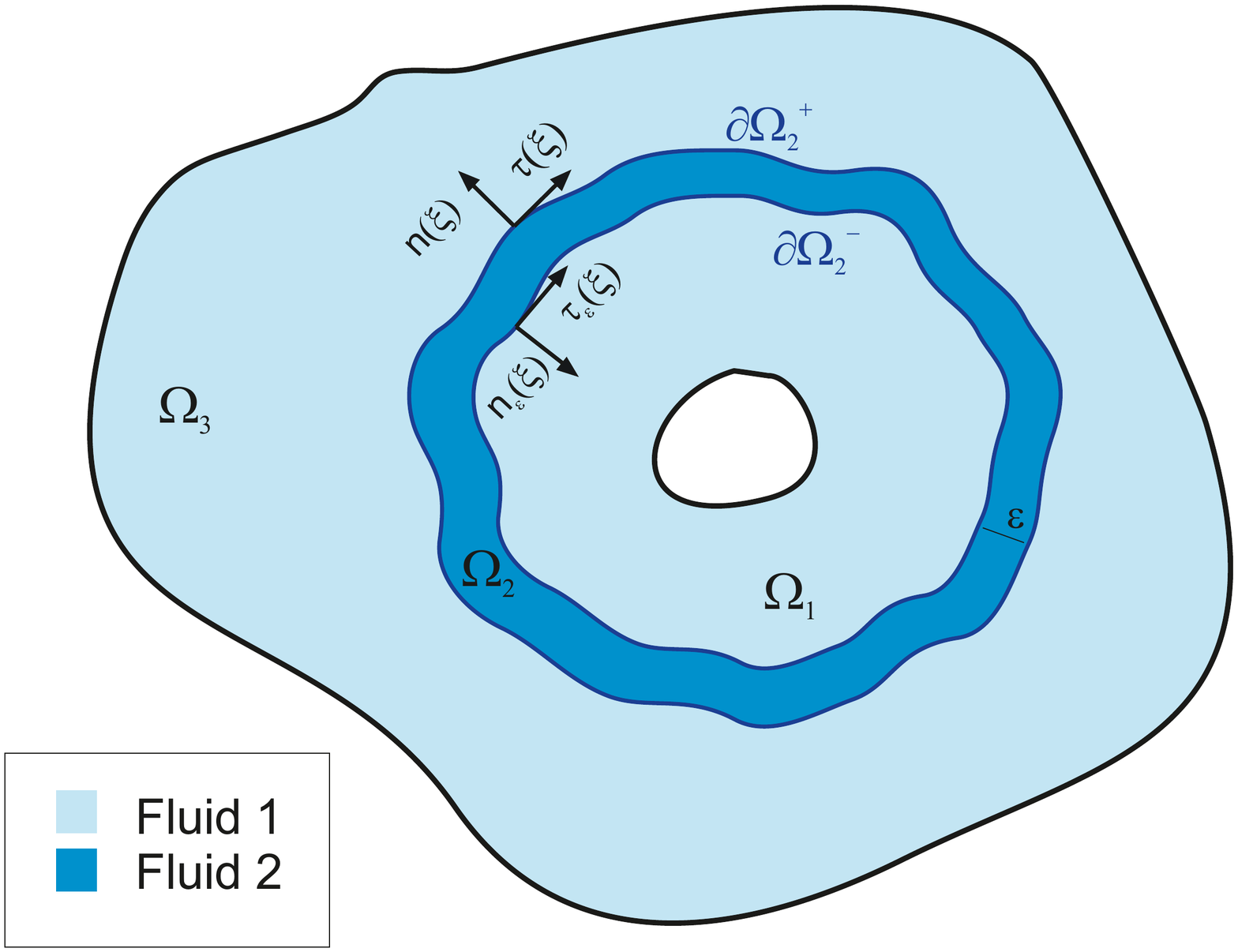}
\caption{Thin layer of a viscous fluid inside another viscous fluid}
\end{figure}
 
  The boundary of the domain $\Omega_2(t)$ consists in two curves $\partial\Omega_2\e{-}(t)$ and $\partial\Omega_2\e{+}(t)$. 
  The velocity of our problem is given by Stokes equations:
\begin{equation}\label{stokesepsilon}
\mu(x)\Delta u=\nabla p,\quad\nabla\cdot u=0\quad\textit{in}\quad \Omega_{k}(t)\quad\textit{for}\quad k=1,2,3
\end{equation}
where 
\begin{equation}\label{viscosidad}
\mu(x)=
\begin{cases}
  \mu & \mbox{if}\quad x\in\Omega_1(t)\cup\Omega_3(t) \\
  \mu_2 & \mbox{if}\quad x\in\Omega_2(t).
\end{cases}
\end{equation}

Given any function $f$ defined in $\Omega(t)$ as in \eqref{omega}, we will denote as $f\e{k}=f|_{\Omega_k(t)}$. We will impose the following boundary conditions:

\begin{align}\label{cond1}
&u=U\quad\textit{smooth in}\quad \partial D,\\\label{cond2}
&u\e{3}\cdot n=u\e{2}\cdot n\quad\textit{at}\quad\partial\Omega_2\e{+},\\\label{cond3}
&u\e{3}\cdot \tau=u\e{2}\cdot \tau\quad\textit{at}\quad\partial\Omega_2\e{+},\\\label{cond4}
&u\e{1}\cdot n=u\e{2}\cdot n\quad\textit{at}\quad\partial\Omega_2\e{-},\\\label{cond5}
&u\e{1}\cdot \tau=u\e{2}\cdot \tau\quad\textit{at}\quad\partial\Omega_2\e{-},\\\label{cond6}
&(\Sigma\e{3}-\Sigma\e{2})n=-\kappa n\quad\textit{in}\quad\partial\Omega_2\e{+},\\\label{cond7}
&(\Sigma\e{2}-\Sigma\e{1})n=-\kappa n\quad\textit{in}\quad\partial\Omega_2\e{-},
\end{align}
where $U\in H\e{\frac{1}{2}}(\partial D)$, $\int_{\partial D}U\cdot n=0$ on $\partial D$ and the normal vector $n$ is chosen as the outer normal to $\Omega_2(t)$ both in \eqref{cond2} and \eqref{cond4}. We denote in \eqref{cond3} and \eqref{cond5} as $\tau$ the tangent vector to the boundary and $\kappa$ is the curvature of the boundary. 
Here we considered the surface tension coefficient is equals to 1. $\Sigma^k$ represents the stress tensor corresponding to each domain $\Omega_k(t)$, namely:
\begin{equation}\label{sigmas}
\begin{cases}
\Sigma\e{1}=-p\e{1}I+\mu(\nabla u\e{1}+(\nabla u\e{1})\e{T})\\
\Sigma\e{2}=-p\e{2}I+\mu_2(\nabla u\e{2}+(\nabla u\e{2})\e{T})\\
\Sigma\e{3}=-p\e{3}I+\mu(\nabla u\e{3}+(\nabla u\e{3})\e{T})
\end{cases}
\end{equation}
  
  The evolution of the interfaces $\partial\Omega_2\e{+}(t)$ and $\partial\Omega_2\e{-}(t)$ is given by 
  \begin{align}\label{ecevol3}
  &V_n=u\e{3}\cdot n=u\e{2}\cdot n\quad\textit{at}\quad\partial\Omega_2\e{+}\\\label{ecevol1}
  &V_n=u\e{1}\cdot n=u\e{2}\cdot n\quad\textit{at}\quad\partial\Omega_2\e{-}
  \end{align}
  where $ V_n$ is the normal velocity of the interfaces in \eqref{cond2} and \eqref{cond4}.

In the rest of this Section, we will drop the dependence of time $t$ for simplicity.  We will assume that the boundaries $\partial\Omega_2\e{+}$ and $\partial\Omega_2\e{-}$ are $C\e{2}$ Jordan curves. We introduce a parameter $\varepsilon>0$ which gives the order of magnitude of the thickness of the domain $\Omega_2$. We represent the external boundary $\partial\Omega_2\e{+}$ using the arc length parametrization as follows:

\begin{align}\label{gama}
\Gamma: \R\to D\subset\R\e{2},\quad\xi\to\Gamma(\xi),\quad\Gamma\in C\e{2}(\R;D)
\end{align}
where $\abs{\dx{}\Gamma}=1$. We assume that the curve $\partial\Omega_2\e{+}$ is followed in the clockwise sense for increasing $\xi$. The function $\Gamma$ is periodic with period $L$, which is the total length of the curve $\partial\Omega_2\e{+}$. Notice that $L$ is a function of $t$ although this dependence will not be written in the following.  The tangent vector to $\partial\Omega_2\e{+}$ is $\tau(\xi)=\partial_\xi\Gamma(\xi)$ and the outer normal vector $n(\xi)$ is given by $n(\xi)=\mathbb{A}\tau(\xi)$ with
\begin{equation}\label{matrizA}
\mathbb{A} = \;
   \begin{pmatrix}
      0 & -1 \\
      1 & 0  \\
   \end{pmatrix}
\end{equation}

 We have chosen the sign of the curvature in \eqref{cond6} in such a way that  

\begin{equation}\label{derivnormal}
\partial_\xi n(\xi)=-\kappa(\xi)\tau(\xi),\quad \partial_\xi \tau(\xi)=\kappa(\xi)n(\xi)
\end{equation} 
 
being $\kappa(\xi)$ the curvature of $\Gamma$. 

On the other hand, we parametrize the internal boundary $\partial\Omega_2\e{-}$ by means of a function 
\begin{align*}
\Gamma_\varepsilon: \R\to D\subset\R\e{2},\quad\xi\to\Gamma_\varepsilon(\xi),\quad\Gamma_\varepsilon\in C\e{2}(\R;D)
\end{align*}

In order to obtain a domain $\Omega_2$ with thickness of order $\varepsilon$, we will assume that $\partial\Omega_2\e{-}$ is parametrized as:
\begin{equation}\label{gamaepsilon}
\Gamma_\varepsilon(\xi)=\Gamma(\xi)-\varepsilon h(\xi)n(\xi)
\end{equation}
where $h\in C\e{2}(\R)$ and $h(\xi)=h(\xi+L)$, $h>0$.

 The tangent vector to $\partial\Omega_2\e{-}$ is then given by $\tau_\varepsilon(\xi)=\frac{\partial_\xi\Gamma_\varepsilon(\xi)}{\norm{\dx{\Gamma}_\varepsilon(\xi)}}$ and the outer normal vector is $n_\varepsilon(\xi)=\mathbb{A}\e{T}\tau_\varepsilon(\xi)=-\mathbb{A}\tau_\varepsilon(\xi)$ with $\mathbb{A}$ in \eqref{matrizA}. Due to the change of the orientation of the base $\{\tau_\varepsilon(\xi), n_\varepsilon(\xi)\}$ with respect to $\{\tau(\xi), n(\xi)\}$  we have that (c.f. \eqref{derivnormal})
 
 \begin{equation}\label{derivnormaleps}
 \dx{n}_\varepsilon(\xi)=\kappa_\varepsilon(\xi)\tau_\varepsilon(\xi),\quad\dx{\tau}_\varepsilon(\xi)=-\kappa_\varepsilon(\xi)n_\varepsilon(\xi)
 \end{equation}
The vectors $\tau_\varepsilon$, $n_\varepsilon$ and the curvature $\kappa_\varepsilon$ can be approximated in terms of $\tau$, $n$ and $\kappa$ as follows: 
\begin{lem}\label{L1} Suppose that $\Gamma\in C\e{2}(\R;D)$ and $h\in C\e{2}(\R)$ with $\Gamma(\xi)=\Gamma(\xi+L)$, $h(\xi)=h(\xi+L)$ for some $L>0$. Then 
 \begin{align}
 &\tau_\varepsilon(\xi)=\tau(\xi)-\varepsilon \dx{h}(\xi)n(\xi)+\mathcal{O}(\varepsilon\e{2})\label{tangenteepsilon}\\
 &n_\varepsilon(\xi)=-n(\xi)-\varepsilon \dx{h}(\xi)\tau(\xi)+\mathcal{O}(\varepsilon\e{2})\label{normalepsilon}\\
 &\kappa_\varepsilon(\xi)=\kappa(\xi)-\varepsilon \dxx{h}(\xi)+\mathcal{O}(\varepsilon\e{2})\label{curvaturaepsilon}
 \end{align}
 where the error terms $\mathcal{O}(\varepsilon\e{2})$ depend on $\norm{\Gamma}_{C\e{2}}$ and $\norm{h}_{C\e{2}}$.
\end{lem}
\begin{proof}
Using \eqref{derivnormal} and \eqref{gamaepsilon} we obtain
\begin{equation}\label{taueps1}
\tau_\varepsilon(\xi)=\frac{(1+\varepsilon h(\xi)\kappa(\xi))\tau(\xi)-\varepsilon\dx{h}(\xi)n(\xi)}{((1+\varepsilon h(\xi)\kappa(\xi))\e{2}+\varepsilon\e{2}(\dx{h}(\xi))\e{2})\e{\frac{1}{2}}}.
\end{equation}

Therefore Taylor theorem implies \eqref{tangenteepsilon}. Using that $n(\xi)=\mathbb{A}\tau(\xi)$ and $n_\varepsilon(\xi)=-\mathbb{A}\tau_\varepsilon(\xi)$
\begin{displaymath}
n_\varepsilon(\xi)=-\frac{\varepsilon\dx{h}(\xi)\tau(\xi)+(1+\varepsilon h(\xi)\kappa(\xi))n(\xi)}{((1+\varepsilon h(\xi)\kappa(\xi))\e{2}+\varepsilon\e{2}(\dx{h}(\xi))\e{2})\e{\frac{1}{2}}}
\end{displaymath}
whence \eqref{normalepsilon} follows.

Differentiating \eqref{taueps1} and using the second equations in \eqref{derivnormal} and \eqref{derivnormaleps}, as well as \eqref{normalepsilon} and multiplying by $\tau(\xi)$, we obtain \eqref{curvaturaepsilon}.
\end{proof}

We will use repeatedly that if the integrals of some functions by some particular combinations of derivatives vanish, then some combinations of these functions vanish too. More precisely we have,

\begin{lem}\label{lemavarphi}
Let $\sigma\subset D$ a $C\e{1}$- Jordan curve and  $M\in C\e{1}(\sigma;M_{2\times 2}(\R))$ and $a,b\in C(\sigma;\R)$. If
\begin{align}\label{intabc1}
&\int_{\sigma}M:(\nabla\varphi+\nabla\varphi\e{T})ds+\int_{\sigma}a(\tau\cdot\varphi) ds+\int_{\sigma}b(n\cdot\varphi) ds=0
\end{align}
for all $\varphi\in C\e{\infty}_c(D;\R\e{2})$ then
\begin{align}\label{conclusionlema}
\tau\cdot(M+M\e{T})n=0,\quad n\cdot(M+M\e{T})n=0,\quad a-\dx{}(\tau\cdot(M+M\e{T})\tau)=0,\quad  \quad b-\kappa\tau\cdot(M+M\e{T})\tau=0.
\end{align} 
\end{lem}
\begin{proof} We parametrize the curve $\sigma$ using the arc length parametrization $\lambda:\R\to D$.
Then, for each $\xi\in\R$ we can defined a orthogonal base $\{\tau(\xi),n(\xi)\}$ where $\tau=\dx{}\lambda(\xi)$ and $n(\xi)=\mathbb{A}\tau(\xi)$ for $\mathbb{A}$ in \eqref{matrizA}. 

Therefore, for each $\xi\in\R$ we can represent $\nabla\varphi=\varphi_{11}(\tau\otimes\tau)+\varphi_{12}(\tau\otimes n)+\varphi_{21}(n\otimes\tau)+\varphi_{22}(n\otimes n)$ with
\begin{align}\label{phi11}
&\varphi_{11}=\tau\nabla\varphi\cdot\tau=\dx(\tau\cdot\varphi)-\kappa\varphi\cdot n\\\label{phi12}
&\varphi_{12}=\tau\nabla\varphi\cdot n=\tau\cdot\frac{\partial\varphi}{\partial n}\\\label{phi21}
&\varphi_{21}=n\nabla\varphi\cdot\tau=\dx(n\cdot\varphi)+\kappa\varphi\cdot\tau\\\label{phi22}
&\varphi_{22}=n\nabla\varphi\cdot n=n\cdot\frac{\partial\varphi}{\partial n}
\end{align}
where we have used that $\nabla\varphi_j\cdot\tau=\frac{\partial\varphi}{\partial \xi}$, $\nabla\varphi_j\cdot n=\frac{\partial\varphi}{\partial n}$ for $j=1,2$ and \eqref{derivnormal}.
Thus,
\begin{equation}\label{intabc2}
M:(\nabla\varphi+(\nabla\varphi)\e{T})=2m_{11}(\dx(\tau\cdot\varphi)-\kappa n\cdot\varphi)+(m_{12}+m_{21})(\tau\cdot\frac{\partial\varphi}{\partial n}+\dx(n\cdot\varphi)+\kappa\tau\cdot\varphi)+2m_{22}(n\cdot\frac{\partial\varphi}{\partial n})
\end{equation}
then using \eqref{intabc2} we can rewrite \eqref{intabc1} as
\begin{align}\label{intabc}
&\int_{\sigma}(m_{12}+m_{21})(\tau\cdot\frac{\partial\varphi}{\partial n})ds+2\int_{\sigma}m_{22}(n\cdot\frac{\partial\varphi}{\partial n})ds+\int_{\sigma}\tilde{a}(\tau\cdot\varphi)ds-\int_{\sigma}\tilde{b}(n\cdot\varphi)ds=0
\end{align}
for all $\varphi\in C_c\e{\infty}(\bar{D}:\R\e{2})$ where $\tilde{a}=\kappa(m_{12}+m_{21})-2\dx{}m_{11}+a$ and $\tilde{b}=2\kappa m_{11}+\dx{}(m_{12}+m_{21})-b$.

We can parametrize a neighbourhood of the curve $\sigma$ in the form $\textbf{x}=\lambda(\xi)+z n(\xi)$ with $z\in(-\delta,\delta)$ for $\delta$ sufficiently small. The mapping $(\xi,z)\to \textbf{x}$ is invertible in $(\xi_0-\delta,\xi_0+\delta)\times(-\delta,\delta)$, $\xi_0\in\R$ and $\delta$ sufficiently small, due to the Inverse Function Theorem.

We can then define a test function $\varphi\in C_c\e{\infty}(\bar{D};\R\e{2})$ by means of $\varphi(\textbf{x})=\chi(z)\psi(\xi)$ for $\chi\in C\e{\infty}_c((-\delta,\delta);\R)$ with $\chi(z)=z$ for $\abs{z}\leq \frac{\delta}{2}$ and $\psi\in C_c\e{\infty}((\xi_0-\delta,\xi_0+\delta);\R\e{2})$ arbitrary. Therefore $\varphi|_{\sigma}=0$ and $\frac{\partial\varphi}{\partial n}|_{\sigma}=\psi(\xi)$ whence \eqref{intabc} reduces to
\begin{align*}
&\int_{\sigma}(m_{12}+m_{21})(\tau\cdot\psi)ds+2\int_{\sigma}m_{22}(n\cdot\psi)ds=0
\end{align*}
which may be seen to be valid for any $\psi\in C_c((\xi_0-\delta,\xi_0+\delta);\R\e{2})$ using a density argument. Choosing $\psi$ in such a way that $\psi\cdot\tau=0$ or $\psi\cdot n=0$, we arrive at $m_{12}+m_{21}=0$ and $m_{22}=0$ respectively. Thus, the first two identities in \eqref{conclusionlema} follow and \eqref{intabc} becomes
\begin{displaymath}
\int_{\sigma}\tilde{a}(\tau\cdot\varphi)ds-\int_{\sigma}\tilde{b}(n\cdot\varphi)ds=0
\end{displaymath}
for all $\varphi\in C_c\e{\infty}(\bar{D}:\R\e{2})$ with $\tilde{a}=-2\dx{}m_{11}+a$ and $\tilde{b}=2\kappa m_{11}-b$.

Now taking $\varphi(\xi,z)=\frac{\chi(z)}{z}\psi(\xi)$ with $\chi$ and $\psi$ as above, we can conclude similarly that $\tilde{a}|_{\sigma}=\tilde{b}|_{\sigma}=0$, hence the result follows. 
\end{proof}
We now reformulate \eqref{stokesepsilon}-\eqref{cond7} using machinery of the theory of distributions. Given any open set $O\subset\R\e{2}$, we will denote as $\mathcal{D}'(O;\R\e{2})$  the space of distributions (c.f. \cite{distrib}). We will use respectively the following class of functions:
 Suppose that $\Omega$ is an open set that can be written as an union $\bigcup_{j=1}\e{N}O_j$, $N\in\N$. We will denote as $X(\Omega)$ the space of continuous functions in $\Omega$ having a limit at the boundaries of each domain $O_j$, more precisely for any $k=0,1,2,\cdots$ we define
\begin{displaymath}
X\e{k}(\Omega)=\{f\in C\e{k}(\Omega): f(x)=f_j(x), x\in O_j, f_j\in C\e{k}(\overline{O_j}), j=1,\cdots,N\}
\end{displaymath} 
$X\e{k}(\Omega;\R\e{2})$ is defined similarly for vector value functions. 
\begin{lem}\label{distristokeseps} Let $\Omega$ as in \eqref{omega}. Suppose that $\partial\Omega_2\e{+}$ and $\partial\Omega_2\e{-}$ are in $C\e{2}$ with $\partial\Omega_2\e{+}\cap\partial\Omega_2\e{-}=\emptyset$ and are parametrized as in \eqref{gama} and \eqref{gamaepsilon}, respectively. Let be $u\in C\e{2}(\Omega;\R\e{2})\cap C(\bar{D};\R\e{2})\cap X\e{1}(\Omega;\R\e{2})$ and $u=U$ in $\partial D$. The function $p\in C\e{1}(\Omega)\cap X\e{0}(\Omega)$. We define the functionals $\s_\varepsilon(u,p)\in\mathcal{D}'(D;\R\e{2})$ as
\begin{align}\label{operadorepsilon}
<\s_\varepsilon(u,p),\varphi>=&\int_{D}[\mu(x)u\cdot(\Delta\varphi+\nabla\cdot\nabla\varphi\e{T})+p\nabla\cdot\varphi]dV+\\\nonumber
&+(\mu-\mu_2)\int_{\om{+}}u\cdot[(\nabla\varphi+\nabla\varphi\e{T})n]ds+(\mu-\mu_2)\int_{\om{-}}u\cdot[(\nabla\varphi+\nabla\varphi\e{T})n_\varepsilon]ds
\end{align}  
for $\varphi\in C\e{\infty}_c(D;\R\e{2})$ and $\mu(x)$ as in \eqref{viscosidad}.

Then the boundary value problem \eqref{stokesepsilon} with boundary conditions \eqref{cond1}-\eqref{cond7} holds if and only if
\begin{equation}\label{stokesdist}
\s_\varepsilon(u,p)=\kappa_\varepsilon n_\varepsilon\delta_{\om{-}}-\kappa n\delta_{\om{+}}, \quad\nabla\cdot u=0\quad\textit{in}\quad \mathcal{D}'(D,\R\e{2})
\end{equation}
where the distributions $n_\varepsilon\delta_{\om{-}}, n\delta_{\om{+}}\in\mathcal{D}'(D,\R\e{2})$ are defined by means of  $<n_\varepsilon\delta_{\om{-}},\varphi>=\int_{\om{-}}n_\varepsilon\cdot\varphi ds$ and $<n\delta_{\om{-}},\varphi>=\int_{\om{+}}n\cdot\varphi ds$ with $\varphi\in C\e{\infty}_c(\bar{D},\R\e{2})$.

\end{lem}

\begin{proof} Suppose first that \eqref{stokesepsilon}-\eqref{cond7} holds.
The first integral on the right in \eqref{operadorepsilon} can be written as the sum
\begin{align*}
&\int_{\Omega}[\mu(x)u\cdot(\Delta\varphi+\nabla\cdot\nabla\varphi\e{T})+p\nabla\cdot\varphi]dV=\int_{\Omega\e{1}\cup\Omega\e{3}}[\mu u\cdot(\Delta\varphi+\nabla\cdot\nabla\varphi\e{T})+p\nabla\cdot\varphi]dV\\
&+\int_{\Omega\e{2}}[\mu_2 u\e{2}\cdot(\Delta\varphi+\nabla\cdot\nabla\varphi\e{T})+p\nabla\cdot\varphi]dV
\end{align*}

Integrating by parts the right hand side of this equations and using \eqref{sigmas} it becomes:
\begin{align*}
&\int_{\Omega\e{1}\cup\Omega\e{3}}[\mu\nabla\cdot(\nabla u+\nabla u\e{T})-\nabla p]\cdot\varphi dV-\mu\int_{\om{+}}u\e{3}\cdot[(\nabla\varphi+\nabla\varphi\e{T})n]ds-\mu\int_{\om{-}}u\e{1}\cdot[(\nabla\varphi+\nabla\varphi\e{T})n_\varepsilon]ds\\
&\int_{\Omega\e{2}}[\mu_2\nabla\cdot(\nabla u\e{2}+(\nabla u\e{2})\e{T})-\nabla p\e{2}]\cdot\varphi dV+\mu_2\int_{\om{+}}u\e{2}\cdot[(\nabla\varphi+\nabla\varphi\e{T})n]ds+\mu_2\int_{\om{-}}u\e{2}\cdot[(\nabla\varphi+\nabla\varphi\e{T})n_\varepsilon]ds\\
&+\int_{\om{+}}(\Sigma\e{3}-\Sigma\e{2})n\cdot\varphi ds+\int_{\om{-}}(\Sigma\e{1}-\Sigma\e{2})n_\varepsilon\cdot\varphi ds\equiv W[u,p;\varphi]
\end{align*}
Using \eqref{stokesepsilon} it follows that $\mu(x)\nabla\cdot(\nabla u\e{k}+(\nabla u\e{k})\e{T})-\nabla p\e{k}=0$ in each $\Omega_{k}$ for $k=1,2,3$. Combining this identity with \eqref{operadorepsilon} and taking to account also the continuity of $u$, we then obtain
\begin{align*}
&<\s_\varepsilon(u,p),\varphi>=\int_{\om{+}}(\Sigma\e{3}-\Sigma\e{2})n\cdot\varphi ds+\int_{\om{-}}(\Sigma\e{1}-\Sigma\e{2})n_\varepsilon\cdot\varphi ds
\end{align*}
Therefore, using the boundary conditions \eqref{cond6} and \eqref{cond7}, we arrive at
\begin{equation}
<\s_\varepsilon(u,p),\varphi>=\int_{\om{-}}\kappa_\varepsilon n_\varepsilon\cdot\varphi ds-\int_{\om{+}}\kappa n\cdot\varphi ds
\end{equation}
whence the first equation in \eqref{stokesdist} holds. The second equation in \eqref{stokesdist} follows form the analogous equation in \eqref{stokesepsilon}. To prove the ``only if" statement we just retrace the same computations in the reverse order to obtain
\begin{equation}\label{eqonlyif}
W[u,p;\varphi]+(\mu-\mu_2)\int_{\om{+}}u\cdot[(\nabla\varphi+\nabla\varphi\e{T})n]ds+(\mu-\mu_2)\int_{\om{-}}u\cdot[(\nabla\varphi+\nabla\varphi\e{T})n_\varepsilon]ds=\int_{\om{-}}\kappa_\varepsilon n_\varepsilon\cdot\varphi ds-\int_{\om{+}}\kappa n\cdot\varphi ds.
\end{equation}
for all $\varphi\in C\e{\infty}_c(\bar{D},\R\e{2})$. Since $u\in C(\bar{D};\R\e{2})$ we can rewrite \eqref{eqonlyif} as follows
\begin{align}\nonumber
&\int_{\Omega}[\mu(x)\nabla\cdot(\nabla u+\nabla u\e{T})-\nabla p]\cdot\varphi dV+\int_{\om{+}}[\tau(\Sigma\e{3}-\Sigma\e{2})\cdot n]\tau\cdot\varphi ds+\int_{\om{+}}[n(\Sigma\e{3}-\Sigma\e{2})\cdot n+\kappa]n\cdot\varphi ds\\\label{eqonlyif2}
&+\int_{\om{-}}[\tau_\varepsilon(\Sigma\e{1}-\Sigma\e{2})\cdot n_\varepsilon]\tau_\varepsilon\cdot\varphi ds+\int_{\om{-}}[n_\varepsilon(\Sigma\e{1}-\Sigma\e{2})\cdot n_\varepsilon-\kappa_\varepsilon]n_\varepsilon\cdot\varphi ds=0
\end{align}

Taking $\varphi\in C_c\e{1}(\Omega;\R\e{2})$ all the integral along the boundaries $\om{+}$ and $\om{-}$ vanish, whence first equation in \eqref{stokesepsilon} follows. Therefore, \eqref{eqonlyif2} reduces to

\begin{align}\nonumber
&\int_{\om{+}}[\tau(\Sigma\e{3}-\Sigma\e{2})\cdot n]\tau\cdot\varphi ds+\int_{\om{+}}[n(\Sigma\e{3}-\Sigma\e{2})\cdot n+\kappa]n\cdot\varphi ds+\int_{\om{-}}[\tau_\varepsilon(\Sigma\e{1}-\Sigma\e{2})\cdot n_\varepsilon]\tau_\varepsilon\cdot\varphi ds\\\nonumber
&+\int_{\om{-}}[n_\varepsilon(\Sigma\e{1}-\Sigma\e{2})\cdot n_\varepsilon-\kappa_\varepsilon]n_\varepsilon\cdot\varphi ds=0
\end{align}
Taking now $\varphi\in C\e{\infty}_c(\Omega\cup\om{+};\R\e{2})$, if we use Lemma \ref{lemavarphi} for $\sigma=\om{+}$, $a=0$ , $b=\tau(\Sigma\e{3}-\Sigma\e{2})\cdot n$ and $c=n(\Sigma\e{3}-\Sigma\e{2})\cdot n$ we obtain \eqref{cond6}. On the other hand, if $\varphi\in C\e{\infty}_c(\Omega\cup\om{-})$ and using Lemma \ref{lemavarphi} for $\sigma=\om{-}$, $a=0$ , $b=\tau_\varepsilon(\Sigma\e{3}-\Sigma\e{2})\cdot n_\varepsilon$ and $c=n_\varepsilon(\Sigma\e{3}-\Sigma\e{2})\cdot n_\varepsilon$ we derive condition \eqref{cond7}. Finally, the second equation in \eqref{stokesepsilon} follows from the analogous equation in \eqref{stokesdist}.
\end{proof}
\begin{nota}
Notice that for functions $(u,p)$ under the assumptions of this lemma, the operator $\s_\varepsilon(u,p)$ is only $\nabla\cdot(\mu(x)\nabla u)=\nabla p$ in the sense of distributions. Therefore, \eqref{stokesdist} can be interpreted as 
\begin{displaymath}
\nabla\cdot(\mu(x)\nabla u)-\nabla p=\kappa_\varepsilon n_\varepsilon\delta_{\om{-}}-\kappa n\delta_{\om{+}}, \quad\nabla\cdot u=0\quad\textit{in}\quad \mathcal{D}'(D,\R\e{2})
\end{displaymath}
It is relevant to remark that $\mu(x)\nabla u$ can not be defined in the sense of distributions of any arbitrary functions $u\in C\e{2}(\Omega;\R\e{2})\cap X\e{1}(\Omega;\R\e{2})$. Due to the presence of possible discontinuities of function $u$ at the boundaries $\om{+}$ and $\om{-}$. This is the reason why we have assumed in the Lemma that $u$ is in $C(\bar{D},\R\e{2})$.
\end{nota}
\begin{prop}\label{leadingext} For any $\varepsilon>0$ small let $(u_\varepsilon,p_\varepsilon)\in C(\bar{D})\times L\e{1}(D)$ a family of solutions of \eqref{stokesepsilon}-\eqref{cond7} where the boundaries $\om{+}$ and $\om{-}$ are parametrized by \eqref{gama} and \eqref{gamaepsilon}, respectively. Suppose also that $u_\varepsilon\to u_0\in C(\bar{D})$ when $\varepsilon\to 0$ in the uniform topology and $p_\varepsilon\to p_0$ in $L\e{1}(D)$. Then
\begin{equation}\label{u0dist}
\s_0(u_0,p_0)=-2\kappa n\delta_{\om{+}},\quad \nabla\cdot u_0=0\quad\textit{in}\quad \mathcal{D}'(D;\R\e{2})
\end{equation}
where $\s_0(u_0,p_0)\in\mathcal{D}'(D;\R\e{2})$ is defined by 
\begin{equation}\label{operadors0}
<\s_0(u_0,p_0),\varphi>=\mu\int_{D}u_0\cdot(\Delta\varphi+\nabla\cdot\nabla\varphi\e{t})+p_0\nabla\cdot\varphi dV 
\end{equation}
for all $\varphi\in C\e{\infty}_c(D;\R\e{2})$.
\end{prop}
\begin{proof}
Using $\om{-}$ is parametrized by \eqref{gamaepsilon} as well as \eqref{derivnormal}, we obtain that the arc length of $\om{-}$ can be written as
\begin{equation}\label{arcleneps}
\abs{d\Gamma_\varepsilon}=\sqrt{(1+\varepsilon h(\xi)\kappa(\xi))\e{2}+\varepsilon\e{2}(\dx{h}(\xi))\e{2}}
\end{equation}
where we also  use that $\xi$ is the arc length of the curve $\om{+}$. 
Therefore,
\begin{align}\label{operadorintlinea}
<\s_\varepsilon(u_\varepsilon,p_\varepsilon),\varphi>&=\int_{\om{-}}\kappa_\varepsilon n_\varepsilon\cdot\varphi ds -\int_{\om{+}}\kappa n\cdot\varphi ds=\\\nonumber
&= \int_0\e{L}\big(\kappa_\varepsilon(\xi)n_\varepsilon(\xi)\cdot\varphi(\Gamma_\varepsilon(\xi))\sqrt{(1+\varepsilon h(\xi)\kappa(\xi))\e{2}+\varepsilon\e{2}(\dx{h}(\xi))\e{2}}-\kappa(\xi)n(\xi)\cdot\varphi(\Gamma(\xi))\big)d\xi
\end{align}

Taking the limit $\varepsilon\to 0$ in \eqref{gamaepsilon}, \eqref{normalepsilon} and \eqref{curvaturaepsilon} we obtain:
\begin{equation}\label{limgamnkapa}
\Gamma_\varepsilon(\xi)\to\Gamma(\xi),\quad n_\varepsilon(\xi)\to -n(\xi),\quad\kappa_\varepsilon(\xi)\to\kappa(\xi)
\end{equation}
then, the right hand side of \eqref{operadorintlinea} converges to
\begin{equation}\label{limit1}
 -2\int_{\om{+}}\kappa n\cdot\varphi ds\quad\textit{as}\quad\varepsilon\to 0.
\end{equation}
On the other hand, using \eqref{arcleneps}, \eqref{limgamnkapa} and the hypotheses $u_\varepsilon\to u_0$ and $p_\varepsilon\to p_0$ when $\varepsilon\to 0$ it follows that the sum of last two terms in the right hand side of \eqref{operadorepsilon} tends to zero as $\varepsilon\to 0$. Finally, using again the convergence $u_\varepsilon\to u_0$ and $p_\varepsilon\to p_0$ when $\varepsilon\to 0$ as well as the fact that $\om{-}$ is parametrized as in \eqref{gamaepsilon} we arrive at
\begin{equation}\label{limit2}
\int_{D}[\mu(x)u_\varepsilon\cdot(\Delta\varphi+\nabla\cdot\nabla\varphi\e{T})+p_\varepsilon\nabla\cdot\varphi]dV\to\mu\int_{D}[ u_0\cdot(\Delta\varphi+\nabla\cdot\nabla\varphi\e{T})+p_0\nabla\cdot\varphi]dV
\end{equation}

Therefore, $<\s_\varepsilon(u_\varepsilon,p_\varepsilon)>$ converges to the right hand side of \eqref{limit2}. Thus, taking the limit  as $\varepsilon\to 0$ in \eqref{operadorintlinea} we obtain the first equation in \eqref{u0dist} (c.f. \eqref{limit1}). The second equation in \eqref{u0dist} follows taking the limit of the equation $\nabla\cdot u_\varepsilon=0$ as $\varepsilon\to 0$ in the sense of distributions.
\end{proof}

\begin{nota}
The meaning of  Proposition \ref{leadingext} is that, to leading order, we can approximate the velocity field $u_\varepsilon$ defined by means \eqref{stokesepsilon}-\eqref{cond7}, neglecting the role of the fluid $2$ and assuming that there is a membrane separating two portions of fluid $1$ with a surface tension is the double of the surface tension between fluid $1$ and $2$.
\end{nota}

We now reformulate \eqref{u0dist} as a problem for Stokes equations satisfying suitable jump conditions across the interface $\om{+}$. Notice that the domains $\Omega_1$ and $\Omega_2$ depend on $\varepsilon$. In the limit $\varepsilon\to 0$, the domain $\Omega_2$ converges to a curve and the domain $\Omega_1$ expands to a domain $D_1=D\setminus\overline{\Omega_3}$. Therefore, $D=D_1\cup\om{+}\cup\Omega_3$ and the sets $D_1$, $\om{+}$ and $\Omega_3$ are mutually disjoints. We will denote as $\Omega_{lim}$ the open set $D_1\cup\Omega_3$.

\begin{lem}\label{bcu0}
Let be $(u_0,p_0)\in [C\e{2}(\Omega_{lim})\cap X\e{1}(\Omega_{lim})]\times [C\e{1}(\Omega_{lim})\cap X\e{0}(\Omega_{lim})]$ with $u_0=U$ in $\partial D$. Then \eqref{u0dist} holds if and only if $(u_0,p_0)$ solves 
\begin{equation}\label{stokesu0}
\mu\Delta u_0=\nabla p_0,\quad \nabla\cdot u_0=0\quad\textit{in}\quad D_1\cup\Omega_3
\end{equation} with the following boundary conditions:
\begin{equation}\label{condicionesu0}
\begin{cases}
u_0\e{3}=u_0\e{1}\quad\textit{in}\quad\om{+}\\
(\Sigma\e{3}_0-\Sigma\e{1}_0)n=-2\kappa n\quad\textit{in}\quad\om{+}
\end{cases}
\end{equation} 
where $(u_0\e{1},p_0\e{1})=(u_0,p_0)|_{D_1}$ and $(u_0\e{3},p_0\e{3})=(u_0,p_0)|_{\Omega_3}$ and we write $\Sigma\e{k}_0=\mu(\nabla u_0\e{k}+(\nabla u_0\e{k})\e{T})-p_0\e{k}I$ for $k=1,3$.
\end{lem}
\begin{proof}
First, suppose that \eqref{u0dist} holds. Using the regularity properties of $(u_0,p_0)$ and integrating by parts in \eqref{operadors0}, we arrive at
\begin{align}\nonumber
&<\s_0(u_0,p_0),\varphi>=-\mu\int_{\om{+}}(u_0\e{3}-u_0\e{1})\cdot(\nabla\varphi+\nabla\varphi\e{T})nds+\mu\int_{\om{+}}[(\partial_x u_{0y}\e{3}+\partial_y u_{0x}\e{3})-(\partial_x u_{0y}\e{1}+\partial_y u_{0x}\e{1})]\tau\cdot\varphi ds\\\label{s0cond}
&+\int_{\om{+}}[2\mu(\partial_y u_{0y}\e{3}-\partial_y u_{0y}\e{1})-(p_0\e{3}-p_0\e{1})]n\cdot\varphi ds=-2\int_{\om{+}}\kappa n\cdot\varphi ds
\end{align}
for all $\varphi\in C\e{\infty}_c(D;\R\e{2})$. Thus, using the definition of $\Sigma_0\e{k}$ we obtain
\begin{align*}
&0=-\mu\int_{\om{+}}(u_0\e{3}-u_0\e{1})\cdot(\nabla\varphi+\nabla\varphi\e{T})nds+\int_{\om{+}}[\tau(\Sigma_0\e{3}-\Sigma_0\e{1})\cdot n]\tau\cdot\varphi ds+\int_{\om{+}}[n(\Sigma_0\e{3}-\Sigma_0\e{1})\cdot n+2\kappa]n\cdot\varphi ds
\end{align*}

Using then Lemma \ref{lemavarphi} for $\sigma=\om{+}$, $M=-\mu(u_0\e{3}-u_0\e{1})\otimes n$, $a=\tau(\Sigma_0\e{3}-\Sigma_0\e{1})\cdot n$ and $b=n(\Sigma_0\e{3}-\Sigma_0\e{1})\cdot n+2\kappa$, we derive \eqref{condicionesu0} taking into account that $\tau\cdot(M+M\e{T})\tau=0$. Finally, we can prove the reverse statement retracing the steps in the reverse order. 
\end{proof}
\begin{nota}\label{notau0}
 Notice that the second equation in \eqref{stokesu0} as well as \eqref{condicionesu0} imply that $\nabla u_0$ is continuous in $\om{+}$. Indeed differentiating the first equation in \eqref{condicionesu0} along the curve $\om{+}$ we obtain $(\nabla u_0\e{3})\tau=(\nabla u_0\e{1})\tau$, then $n\cdot(\nabla u_0\e{3})\tau=n\cdot(\nabla u_0\e{1})\tau$ and $\tau\cdot(\nabla u_0\e{3})\tau=\tau\cdot(\nabla u_0\e{1})\tau$. On the other hand, the incompressibility condition in \eqref{stokesu0} implies that $n(\nabla u_0\e{3})\cdot n=n(\nabla u_0\e{1})\cdot n$. Multiplying the second equation in \eqref{condicionesu0} by $\tau$ we obtain that $\tau\cdot\nabla u_0\e{3}n+n\cdot \nabla u_0\e{3}\tau=\tau\cdot\nabla u_0\e{1}n+n\cdot \nabla u_0\e{1}\tau$. Therefore, it follows that $\nabla u_0\e{3}=\nabla u_0\e{1}$. Finally, multiplying the second equation in \eqref{condicionesu0} by $n$ we obtain $p_0\e{3}-p_0\e{1}=2\kappa$ in $\om{+}$.
\end{nota}
\section{Derivation of a thin film approximation using matched asymptotics}\label{S3}

Proposition \ref{leadingext} suggests that to the leading order the normal velocity at the interfaces $\om{+}$ and $\om{-}$ is given by $u_0\cdot n$ and $-u_0\cdot n$ respectively. (c.f. \eqref{normalepsilon}). Given that the distance between these interfaces is of order $\varepsilon$, we need to approximate the velocities at the interfaces up to the next order in order to estimate the range of times for which this distance remains of order $\varepsilon$.
In this section we compute this approximation of the velocities using formal matched asymptotic expansions.

In order to compute the inner expansion of the velocity field we need to reformulate \eqref{stokesepsilon}-\eqref{cond7} using a rescaled set of variables in a neighbourhood of any point $x_0\in\om{+}$. Suppose that $x_0=\Gamma(\xi_0)$ where $\Gamma$ is the arc length parametrization in \eqref{gama}. We then define:
\begin{equation}\label{cambiovariable}
y=\frac{x-x_0}{\varepsilon};\quad v_\varepsilon(y)=u_\varepsilon(x);\quad P_\varepsilon(y)=\varepsilon p_\varepsilon(x)\quad\textit{and}\quad \tilde{\mu}(y)=\mu(x).
\end{equation}
as well as the domains $\tilde{\Omega}_k=\frac{\Omega_k-x_0}{\varepsilon}$ for $k=1,2,3$.

Then using \eqref{stokesepsilon} we obtain that $(v_\varepsilon,P_\varepsilon)$ solves 

\begin{equation}\label{eqintepsilon}
\tilde{\mu}(y)\Delta_y v_\varepsilon-\nabla_y P_\varepsilon=0,\quad\nabla_y\cdot v_\varepsilon=0\quad\textit{in}\quad\tilde{\Omega}_k\quad\textit{for}\quad k=1,2,3
\end{equation}
Using the parametrizations of the curves $\om{+}$ and $\om{-}$ given in \eqref{gama} and \eqref{gamaepsilon}, it follows that the curves $\tilde{\om{+}}$ and $\tilde{\om{-}}$ can be parametrized as  $\{y=\tilde{\Gamma}(\tilde{\xi})\}$ and $\{y=\tilde{\Gamma}_\varepsilon(\tilde{\xi})\}$ respectively, where
\begin{align}\label{gamatilde1}
&\tilde{\Gamma}(\tilde{\xi})=\frac{\Gamma(\xi_0+\varepsilon\tilde{\xi})-x_0}{\varepsilon}\\\label{gamaepstilde1}
&\tilde{\Gamma}_\varepsilon(\tilde{\xi})=\frac{\Gamma(\xi_0+\varepsilon\tilde{\xi})-x_0}{\varepsilon}-h(\xi_0+\varepsilon\tilde{\xi})n(\xi_0+\varepsilon\tilde{\xi})
\end{align}
for $\tilde{\xi}=\frac{\xi-\xi_0}{\varepsilon}$. Notice that $\tilde{\Gamma}(\tilde{\xi})$ provides an arc length parametrization of the curve $\tilde{\om{+}}$.

We will denote as $\tilde{\tau}(\tilde{\xi})=\dx{\tilde{\Gamma}}$ and $\tilde{\tau}_\varepsilon(\tilde{\xi})=\frac{\dx{\tilde{\Gamma}}_\varepsilon}{\abs{\dx{\tilde{\Gamma}}_\varepsilon}}$ the tangent vectors of $\tilde{\om{+}}$, $\tilde{\om{-}}$ at the points $\tilde{\Gamma}(\tilde{\xi})$ and $\tilde{\Gamma}_\varepsilon(\tilde{\xi})$ respectively. Similarly, we denote as $\tilde{n}(\tilde{\xi})=\mathbb{A}\tilde{\tau}(\tilde{\xi})$, $\tilde{n}_\varepsilon(\tilde{\xi})=-\mathbb{A}\tilde{\tau}_\varepsilon(\tilde{\xi})$ the corresponding outer normal vectors where $\mathbb{A}$ is as in \eqref{matrizA}. 
The boundary conditions \eqref{cond2}-\eqref{cond7} imply:
\begin{align}
&v_\varepsilon\quad\textit{ is continuous in}\quad\tilde{\om{+}}.\\
&v_\varepsilon\quad\textit{ is continuous in}\quad\tilde{\om{-}}.\\
&(\tilde{\Sigma}\e{3}-\tilde{\Sigma}\e{2})\tilde{n}(\tilde{\xi})=-\tilde{\kappa}(\tilde{\xi})\tilde{n}(\tilde{\xi})\quad\textit{ in}\quad\tilde{\om{+}}.\\
&(\tilde{\Sigma}\e{1}-\tilde{\Sigma}\e{2})\tilde{n}_\varepsilon(\tilde{\xi})=\tilde{\kappa}_\varepsilon(\tilde{\xi})\tilde{n}_\varepsilon(\tilde{\xi})\quad\textit{ in}\quad\tilde{\om{-}}.
\end{align}
where $\tilde{\kappa}(\tilde{\xi})$ and $\tilde{\kappa}_\varepsilon(\tilde{\xi})$ are the curvatures at the points $\tilde{\Gamma}(\tilde{\xi})$ and $\tilde{\Gamma}_\varepsilon(\tilde{\xi})$.  The stress tensor is given by $\tilde{\Sigma}\e{i}$ in the new variables in each domain:
\begin{equation}\label{sigmastilde}
\begin{cases}
\tilde{\Sigma}\e{1}=-P_\varepsilon\e{1}I+\mu(\nabla_y v_\varepsilon\e{1}+(\nabla_y v_\varepsilon\e{1})\e{T})\\
\tilde{\Sigma}\e{2}=-P_\varepsilon\e{2}I+\mu_2(\nabla_y v_\varepsilon\e{2}+(\nabla_y v_\varepsilon\e{2})\e{T})\\
\tilde{\Sigma}\e{3}=-P_\varepsilon\e{3}I+\mu(\nabla_y v_\varepsilon\e{3}+(\nabla_y v_\varepsilon\e{3})\e{T})
\end{cases}
\end{equation}

The curves $\tilde{\om{+}}$ and $\tilde{\om{-}}$ can be approximated for $\abs{y}$ of order $1$ and $\varepsilon\to 0$, by means of some parabolas with very large radius of curvature. The main geometrical properties of $\tilde{\om{+}}$ and $\tilde{\om{-}}$ are collected in the following lemma: 
\begin{lem}
Suppose that $\Gamma\in C\e{3}(\R;D)$ and $h\in C\e{3}(\R)$. Then the following approximations hold uniformly on bounded sets of $\tilde{\xi}$.
\begin{align}\label{gammatilde}
&\tilde{\Gamma}(\tilde{\xi})=\tau(\xi_0)\tilde{\xi}+\frac{\varepsilon}{2}\kappa(\xi_0)n(\xi_0)\tilde{\xi}\e{2}+\mathcal{O}(\varepsilon\e{2})\\\label{gammaepsilontilde}
&\tilde{\Gamma}_\varepsilon(\tilde{\xi})=-h(\xi_0)n(\xi_0)+\tau(\xi_0)\tilde{\xi}+\varepsilon\Big(\frac{1}{2}\kappa(\xi_0)n(\xi_0)\tilde{\xi}\e{2}+(h(\xi_0)\kappa(\xi_0)\tau(\xi_0)-\dx{h}(\xi_0)n(\xi_0))\tilde{\xi}\Big)+\mathcal{O}(\varepsilon\e{2})\\\label{tautilde}
&\tilde{\tau}(\tilde{\xi})=\tau(\xi_0)+\varepsilon\kappa(\xi_0)n(\xi_0)\tilde{\xi}+\mathcal{O}(\varepsilon\e{2})\\\label{ntilde}
&\tilde{n}(\tilde{\xi})=n(\xi_0)-\varepsilon\kappa(\xi_0)\tau(\xi_0)\tilde{\xi}+\mathcal{O}(\varepsilon\e{2})\\\label{tautildeeps}
&\tilde{\tau}_\varepsilon(\tilde{\xi})=\tau(\xi_0)+\varepsilon(\kappa(\xi_0)n(\xi_0)\tilde{\xi}-\dx{h}(\xi_0)n(\xi_0))+\mathcal{O}(\varepsilon\e{2})\\\label{nepsilontildeeps}
&\tilde{n}_\varepsilon(\tilde{\xi})=-n(\xi_0)+\varepsilon(\kappa(\xi_0)\tau(\xi_0)\tilde{\xi}-\dx{h}(\xi_0)\tau(\xi_0))+\mathcal{O}(\varepsilon\e{2})\\\label{kapatilde}
&\tilde{\kappa}(\tilde{\xi})=\varepsilon\kappa(\xi_0)+\mathcal{O}(\varepsilon\e{2})\\\label{kapaepsilontilde}
&\tilde{\kappa}_\varepsilon(\tilde{\xi})=\varepsilon\kappa(\xi_0)+\mathcal{O}(\varepsilon\e{2})
\end{align}

\end{lem}
\begin{proof}
The approximations \eqref{gammatilde}, \eqref{tautilde}, \eqref{ntilde} and \eqref{kapatilde} follow from Taylor approximation at $\xi=\xi_0$ and \eqref{derivnormal}. On the other hand \eqref{gammaepsilontilde}, \eqref{tautildeeps},\eqref{nepsilontildeeps} and \eqref{kapaepsilontilde} follow from \eqref{derivnormaleps}, \eqref{taueps1} and  Taylor series.
\end{proof}

We now approximate $v_\varepsilon$ and $P_\varepsilon$ by means of the following power of series in $\varepsilon$ 
\begin{equation}\label{expansionvp}
v_\varepsilon=v_0+\varepsilon v_1+\cdots\quad\textit{and}\quad P_\varepsilon=P_0+\varepsilon P_1+\cdots
\end{equation}

Due to \eqref{gammatilde} and \eqref{gammaepsilontilde} the boundaries $\tilde{\om{+}}$ and $\tilde{\om{-}}$ converge as $\varepsilon\to 0$ to the straight lines
\begin{align}\label{c+}
&c\e{+}=\{y=\tau(\xi_0)\tilde{\xi}\}\\\label{c-}
&c\e{-}=\{y=-h(\xi_0)n(\xi_0)+\tau(\xi_0)\tilde{\xi}\}
\end{align} 

Therefore the functions $(v_0,P_0)\in X\e{1}(\tilde{D}_0;\R\e{2})\times X\e{0}(\tilde{D}_0;\R)$ satisfy
\begin{equation}\label{sistmv0}
\tilde{\mu}_0(y)\Delta_y v_0-\nabla_y P_0=0\quad,\nabla_y\cdot v_0=0\quad\textit{in}\quad \tilde{D}_0=\R\e{2}\setminus [c\e{+}\cup c\e{-}]
\end{equation}
where 
\begin{displaymath}
\tilde{\mu}_0(y)=\begin{cases}
\mu,\quad y\in\R\e{2}\setminus\overline{\tilde{\Omega}_0}\\
\mu_2,\quad y\in\tilde{\Omega}_0
\end{cases}
\end{displaymath}
and $\tilde{\Omega}_0=\{y=-rn(\xi_0)+\tilde{\xi}\tau(\xi_0),r\in(0,h(\xi_0)),\tilde{\xi}\in\R\}$. O pongo: $\tilde{\Omega}_0\e{2}=\{y=-rn(\xi_0)+\tilde{\xi}\tau(\xi_0),r\in(0,h(\xi_0)),\tilde{\xi}\in\R\}$, $\tilde{\Omega}_0\e{3}=\{y=-rn(\xi_0)+\tilde{\xi}\tau(\xi_0),r\in(-\infty,0),\tilde{\xi}\in\R\}$ and $\tilde{\Omega}_0\e{1}=\{y=-rn(\xi_0)+\tilde{\xi}\tau(\xi_0),r\in(h(\xi_0),\infty),\tilde{\xi}\in\R\}$.

Moreover, the functions $(v_0,P_0)$ satisfy the following boundary conditions:
 \begin{align}
&v_0\quad\textit{ is continuous in}\quad c\e{+}\cup c\e{-}.\\
&\big(\mu(\nabla_y v_0\e{3}+(\nabla_y v_0\e{3})\e{T})-\mu_2(\nabla_y v_0\e{2}+(\nabla_y v_0\e{2})\e{T})-(P_0\e{3}-P_0\e{2})I\big)n(\xi_0)=0\quad\textit{ in}\quad c\e{+}.\\
&\big(\mu(\nabla_y v_0\e{1}+(\nabla_y v_0\e{1})\e{T})-\mu_2(\nabla_y v_0\e{2}+(\nabla_y v_0\e{2})\e{T})-(P_0\e{1}-P_0\e{2})I\big)n(\xi_0)=0\quad\textit{ in}\quad c\e{-}.
\end{align}

We assume also that there exist an outer expansion 
\begin{equation}\label{expensionup}
(u_\varepsilon, p_\varepsilon)=(u_0,p_0)+\varepsilon (u_1,p_1)+\cdots\quad\textit{as}\quad\varepsilon\to 0.
\end{equation}  Therefore, Proposition \ref{leadingext} shows that $(u_0,p_0)$ solves \eqref{u0dist}.
 
Then, using \eqref{cambiovariable} it follows that $(v_0,P_0)$ must satisfy the following matching conditions 

\begin{equation}\label{limitesv0p0}
v_0\to u_0(x_0)\quad\textit{and}\quad P_0\to 0\quad\textit{as}\quad dist(y,c\e{+}\cup c\e{-})\to\infty.
\end{equation}

The unique solution of \eqref{sistmv0}-\eqref{limitesv0p0} is 
\begin{equation}\label{solucionv0P0}
v_0(y)=u_0(x_0),\quad P_0(y)=0\quad\textit{for}\quad k=1,2,3.
\end{equation}

Using \eqref{eqintepsilon}, \eqref{gammatilde},\eqref{gamaepsilon}, as well as the fact that $\nabla_y v_0=0$ and $P_0=0$ (c.f. \eqref{limitesv0p0})it follows that $(v_1,P_1)\in X\e{1}(\tilde{D}_0;\R\e{2})\times X\e{0}(\tilde{D}_0;\R)$ solves 

\begin{equation}\label{sistmv1}
\tilde{\mu}_0(y)\Delta_y v_1-\nabla_y P_1=0,\quad\nabla_y\cdot v_1=0\quad\textit{in}\quad\tilde{D}_0
\end{equation}
with the following boundary conditions:
 \begin{align}\label{condO11}
 &v_1\quad\textit{ is continuous in}\quad c\e{+}\cup c\e{-}.\\\label{condO13}
&\big(\mu(\nabla_y v_1\e{3}+(\nabla_y v_1\e{3})\e{T})-\mu_2(\nabla_y v_1\e{2}+(\nabla_y v_1\e{2})\e{T})-(P_1\e{3}-P_1\e{2})I\big)n(\xi_0)=-\kappa(\xi_0)n(\xi_0)\quad\textit{ in}\quad c\e{+}.\\\label{condO14}
&\big(\mu(\nabla_y v_1\e{1}+(\nabla_y v_1\e{1})\e{T})-\mu_2(\nabla_y v_1\e{2}+(\nabla_y v_1\e{2})\e{T})-(P_1\e{1}-P_1\e{2})I\big)n(\xi_0)=\kappa(\xi_0)n(\xi_0)\quad\textit{ in }\quad c\e{-}.
\end{align}

\begin{lem}\label{matchingv1}
Let us denote that $(v_1\e{k},P_1\e{k})=(v_1|_{\tilde{\Omega}_0\e{k}},P_1|_{\tilde{\Omega}_0\e{k}})$ for $k=1,3$. There exists a piecewise affine solution of \eqref{sistmv1}-\eqref{condO14} holds, satisfying the matching conditions:
\begin{align}\label{v1ext}
&v_1\e{k}\sim M_k(x_0)y\\\label{P1ext}
&P_1\e{k}\sim P\e{k}
\end{align}
as $dist(y,c\e{+}\cup c\e{-})\to\infty$, if and only if 
\begin{equation}\label{compatibilityconditions}
M_1=M_3,\quad tr(M_1)=tr(M_3)=0\quad\textit{and}\quad P\e{3}-P\e{1}=2\kappa.
\end{equation}
 Moreover, this solution has the following form
\begin{align}\label{v1}
&v_1\e{k}(y)=B_{k}y+w_{k}\quad\textit{in}\quad \tilde{\Omega}_0\e{k}\quad\textit{for}\quad k=1,2,3.\\\label{P1}
&P_1\e{k}(y)=P\e{k}\quad\textit{in}\quad \tilde{\Omega}_0\e{k}\quad\textit{for}\quad k=1,2,3.
\end{align} 
The matrices $B_{k}$ are given by 
\begin{equation}\label{matricesBk}
B_{k}=M_k\quad\textit{for}\quad k=1,3
\end{equation}
and 
\begin{equation}\label{matrizB2}
B_{2}-B_{1}=(\frac{\mu}{\mu_2}-1)(\tau(\xi_0)(B_{1}+(B_1)\e{T})\cdot n(\xi_0))(\tau(\xi_0)\otimes n(\xi_0))
\end{equation}
The vectors $w_{k}$ are given by
\begin{equation}\label{vectoresw}
w_{3}=w_{2},\quad w_{1}=w_{2}+h(\xi_0)(B_{1}-B_{2})n(\xi_0)
\end{equation}
On the other hand, 
\begin{equation}\label{presionesP1}
P_1\e{k}=P\e{k}\quad\textit{for}\quad k=1,3 
\end{equation}
and
\begin{equation}\label{presionP12}
P_1\e{2}=\frac{P_1\e{3}+P_1\e{1}}{2}-2(\mu-\mu_2)n(\xi_0)B_{1}\cdot n(\xi_0)
\end{equation}

\end{lem}
\begin{nota}
Notice that the family of solutions depend on the vector $w_2\in\R\e{2}$ that can be chosen arbitrarily.
\end{nota}
\begin{proof}
We look for piecewise affine solutions of \eqref{sistmv1}-\eqref{P1ext}, i.e. with the form \eqref{v1}-\eqref{P1}. Then the Stokes equations \eqref{sistmv1} immediately hold in $\tilde{\Omega}_0\e{k}$ for $k=1,2,3$. The matching conditions \eqref{v1ext}-\eqref{P1ext} imply \eqref{matricesBk} and \eqref{presionesP1}. Moreover, the incompressibility condition in \eqref{sistmv1} implies that 
\begin{equation}\label{trazas}
tr(B_{k})=0\quad\textit{for}\quad k=1,2,3.
\end{equation}
Therefore, the matrix $M_k$ hold that $tr(M_k)=0$ for $k=1,3$.

Using condition \eqref{condO11}, we obtain $$B_{3}\tau(\xi_0)\tilde{\xi}+w_{3}=B_{2}\tau(\xi_0)\tilde{\xi}+w_{2}$$
$$B_{1}(-h(\xi_0)n(\xi_0)+\tau(\xi_0)\tilde{\xi})+w_{1}=B_{2}(-h(\xi_0)n(\xi_0)+\tau(\xi_0)\tilde{\xi})+w_{2}$$ for all $\tilde{\xi}\in\R$. Therefore 
\begin{equation}\label{relacionbktau}
B_{1}\tau(\xi_0)=B_{2}\tau(\xi_0)=B_{3}\tau(\xi_0)
\end{equation} 
and we obtain \eqref{vectoresw}.
Since $\nabla_y v_1\e{k}(y)=B_{k}$ for $k=1,2,3$, the tangent component of the conditions \eqref{condO13} and \eqref{condO14} are
\begin{align*}
&\tau(\xi_0)\cdot\big(\mu(B_{3}+(B_{3})\e{T})-\mu_2(B_{2}+(B_{2})\e{T})\big) n(\xi_0)=0\\
&\tau(\xi_0)\cdot\big(\mu(B_{1}+(B_{1})\e{T})-\mu_2(B_{2}+(B_{2})\e{T})\big)n(\xi_0)=0
\end{align*}
 Thus, the symmetric matrices $B_{k}+(B_{k})\e{T}$ for $k=1,2,3$ are related by means of the following formulas:
\begin{align}\label{tangtensor1}
&\mu(B_{3}+(B_{3})\e{T})-\mu_2(B_{2}+(B_{2})\e{T})=D_{3}\\\label{tangtensor2}
&\mu(B_{1}+(B_{1})\e{T})-\mu_2(B_{2}+(B_{2})\e{T})=D_{1}
\end{align}
where $D_{k}$ are diagonal matrices that due to \eqref{trazas} satisfy $tr(D_{1})=tr(D_{3})=0$. Therefore, each of then can be characterised by means of just to a number. Multiplying \eqref{relacionbktau} by $\tau(\xi_0)$ we obtain that $\tau(x_0)\cdot D_{1}\tau(x_0)=\tau(x_0)\cdot D_{3}\tau(x_0)$, then $D_1=D_3$. Now if we subtract \eqref{tangtensor1} and \eqref{tangtensor2} we have that $B_3+(B_3)\e{T}=B_1+(B_1)\e{T}$. Multiplying by $\tau(\xi_0)$ and $n(\xi_0)$ and using \eqref{relacionbktau} we obtain $B_1=B_3$. Thus, multiplying \eqref{tangtensor1} by $\tau(\xi_0)$ and $n(\xi_0)$ we obtain that $$\tau(\xi_0)\cdot B_2 n(\xi_0)=(\frac{\mu}{\mu_2}-1)n(\xi_0)\cdot B_3\tau(\xi_0)+\frac{\mu}{\mu_2}\tau(\xi_0)\cdot B_1 n(\xi_0).$$ Therefore, we obtain \eqref{matrizB2}.

The normal component of conditions \eqref{condO13} and \eqref{condO14} give us
\begin{align*}
&n(\xi_0)\big(\mu(B_3+(B_3)\e{T})-\mu_2(B_2+B_2\e{T})-(P\e{3}-P\e{2})I\big)n(\xi_0)=-\kappa(\xi_0)\\
&n(\xi_0)\big(\mu(B_1+(B_1)\e{T})-\mu_2(B_2+B_2\e{T})-(P\e{1}-P\e{2})I\big)n(\xi_0)=\kappa(\xi_0).
\end{align*}
Hence, using that $B_1=B_3$ and \eqref{matrizB2}
\begin{align}\label{normtensor1}
&2(\mu -\mu_2)n(\xi_0)\cdot B_1 n(\xi_0)-(P\e{3}-P\e{2})=-\kappa(\xi_0),\\\label{normaltensor2}
&2(\mu-\mu_2)n(\xi_0)\cdot B_1 n(\xi_0)-(P\e{1}-P\e{2})=\kappa(\xi_0).
\end{align}
Subtracting \eqref{normtensor1} and \eqref{normaltensor2} we have the compatibility condition $$P_1\e{3}-P_1\e{1}=2\kappa(\xi_0)$$ and adding \eqref{normtensor1} and \eqref{normaltensor2} we obtain \eqref{presionP12}.
\end{proof}

We now compute the effect in the outer region of the thickness of the layer of fluid in $\Omega_2$. We recall that we are assuming that $(u_\varepsilon, p_\varepsilon)$ admit an expansion as \eqref{expensionup}. This expansion is expected to be valid at distances of $\Omega_2$ much larger than $\varepsilon$. 

On the other hand, assuming that the limits of the functions $u_0\e{k}(x)$, $\nabla u_0\e{k}(x)$ and $p\e{k}_0(x)$ as $x\to x_0$ at each side of $\om{+}$ exist and using the Taylor approximation at $x=x_0$ we obtain
\begin{align}\label{tayloru0p0}
&u_0\e{k}(x)=u_0\e{k}(x_0)+\nabla u_0\e{k}(x_0)(x-x_0)+\cdots,\quad p_0\e{k}(x)=p_0\e{k}(x_0)+\cdots
\end{align}
for $k=1,3$, then using \eqref{cambiovariable} we obtain the following matching condition for $(v_\varepsilon, P_\varepsilon)$:
\begin{align*}
v\e{k}_\varepsilon(y)\sim u_0\e{k}(x_0)+\varepsilon \nabla u_0\e{k}(x_0)y+\cdots,\quad P\e{k}_\varepsilon(y)\sim \varepsilon p\e{k}_0(x_0)+\cdots
\end{align*} 
which is valid for $dist(x,x_0)\to 0$ and $dist(y,c\e{+}\cup c\e{-})\to\infty$, for any $x_0\in\om{+}$ and $k=1,3$. Using the expansion \eqref{expansionvp}, for the terms of order 1 we recover the matching condition \eqref{limitesv0p0} whence \eqref{solucionv0P0} follows. The terms of order $\varepsilon$ yields the matching conditions for $(v_1,P_1)$:
\begin{align}\label{matchingv1sinu1}
v_1\e{k}(y)\sim \nabla u_0\e{k}(x_0)y\quad \textit{and}\quad P_1\e{k}\sim p_0\e{k}(x_0)
\end{align}
as $dist(y,c\e{+}\cup c\e{-})\to\infty$ for any $x_0\in\om{+}$ and $k=1,3$. Notice that Remark \ref{notau0} implies that 

\begin{equation}\label{continuidadu0p0}
\nabla u_0\e{3}(x_0)=\nabla u_0\e{1}(x_0),\quad p_0\e{3}(x_0)-p_0\e{1}(x_0)=2\kappa.
\end{equation}
 Then, Lemma \ref{matchingv1} suggests that $(v_1,P_1)$ are given by \eqref{v1}-\eqref{presionP12} where
\begin{equation}\label{mkpk}
M_k=\nabla u_0\e{k}(x_0)\quad\textit{and}\quad P\e{k}=p_0\e{k}
\end{equation}
for $k=1,3$. Notice that $w_2$, and therefore $w_1$ and $w_3$, it is not yet prescribed.

\begin{prop} \label{orden1ext}For any $\varepsilon>0$ small let $(u_\varepsilon,p_\varepsilon)\in (C\e{2}(\Omega;\R\e{2})\cap C(\bar{D};\R\e{2})\cap X\e{1}(\Omega;\R\e{2}))\times (C\e{1}(\Omega)\cap X\e{0}(\Omega))$ a family of solutions of \eqref{stokesepsilon}-\eqref{cond7} where the boundaries $\om{+}$ and $\om{-}$ are parametrized by \eqref{gama} and \eqref{gamaepsilon}, respectively. Suppose also that $\abs{\frac{u_\varepsilon-u_0}{\varepsilon}}$ is bounded and that $u_1=\lim_{\varepsilon\to 0}\frac{u_\varepsilon - u_0}{\varepsilon}$ and $p_1=\lim_{\varepsilon\to 0}\frac{p_\varepsilon-p_0}{\varepsilon}$ in $\Omega_{lim}$ where $(u_1,p_1)\in X\e{0}(\Omega_{lim};\R\e{2})\times X\e{0}(\Omega_{lim})$. Let us assume that $\frac{v_\varepsilon(y)-u_0(x_0)}{\varepsilon}$ and $P_\varepsilon(y)$ converge as $\varepsilon\to 0$ uniformly in compact sets of $y\in\R\e{2}$ to the functions $v_1$ and $P_1$ in \eqref{v1}-\eqref{presionP12}, \eqref{mkpk} for some $w_2\in\R\e{2}$. Then,
\begin{equation}\label{u1dist}
\begin{cases}
<\s_0(u_1,p_1),\varphi>=(\mu_2-\mu)\int_{\om{+}}hB_2:(\nabla\varphi+(\nabla\varphi)\e{T})ds+2(\mu-\mu_2)\int_{\om{+}}h(n\cdot B_2 n)\nabla\cdot\varphi ds\\
+\int_{\om{+}}\dx\e{2}hn\cdot\varphi ds+\int_{\om{+}}\dx{}\kappa h\tau\cdot\varphi ds\\
\nabla\cdot u_1=0\\
u_1=0\quad\textit{in}\quad\partial D
\end{cases}
\end{equation}
for all $\varphi\in C_c\e{\infty}(\Omega)$, where 
\begin{equation}\label{B2funcnabla}
B_{2}=B_2(x_0)=\nabla u_0\e{1}(x_0)+(\frac{\mu}{\mu_2}-1)(\tau(\xi_0)(\nabla u_0\e{1}(x_0)+(\nabla u_0\e{1}(x_0))\e{T})\cdot n(\xi_0))(\tau(\xi_0)\otimes n(\xi_0))
\end{equation}
for any $x_0\in\om{+}$, $x_0=\Gamma(\xi_0)$ and  $\s_0(u_1,p_1)\in\mathcal{D}'(D;\R\e{2})$ is given by \eqref{operadors0}.

\end{prop}
\begin{proof}
Using \eqref{stokesdist} in Lemma \ref{distristokeseps} and \eqref{u0dist} in Proposition \ref{leadingext} we arrive at
\begin{equation}\label{restases0}
\frac{1}{\varepsilon}(<\s_\varepsilon(u_\varepsilon,p_\varepsilon),\varphi>-<\s_0(u_0,p_0),\varphi>)=\frac{1}{\varepsilon}(\int_{\om{-}}\kappa_\varepsilon n_\varepsilon\cdot\varphi ds+\int_{\om{+}}\kappa n\cdot\varphi ds)
\end{equation}

We now compute the limit as $\varepsilon\to 0$ of the right hand side of \eqref{restases0}. To this end we use \eqref{gamaepsilon},\eqref{normalepsilon} and \eqref{curvaturaepsilon}. Then,
\begin{align}\nonumber
&\lim_{\varepsilon\to 0}\frac{1}{\varepsilon}\big(\int_{\om{-}}\kappa_\varepsilon n_\varepsilon\cdot\varphi ds+\int_{\om{+}}\kappa n\cdot\varphi ds\big)=\\\nonumber
&=\lim_{\varepsilon\to 0}\frac{1}{\varepsilon}\big(\int_0\e{1}\big(\kappa_\varepsilon(\xi)n_\varepsilon(\xi)\cdot\varphi(\Gamma_\varepsilon(\xi))\sqrt{(1+\varepsilon h(\xi)\kappa(\xi))\e{2}+\varepsilon\e{2}(\dx{h}(\xi))\e{2}}+\kappa(\xi)n(\xi)\cdot\varphi(\Gamma(\xi))\big)d\xi\big)=\\\nonumber
&=\int_{0}\e{1}\dxx{h}(\xi)n(\xi)\cdot\varphi(\Gamma(\xi))d\xi-\int_0\e{1}\kappa(\xi)\dx{h}(\xi)\tau(\xi)\cdot\varphi(\Gamma(\xi))d\xi-\int_0\e{1}(\kappa(\xi))\e{2}h(\xi)n(\xi)\cdot\varphi(\Gamma(\xi))d\xi\\\label{limitedcha}
&+\int_{0}\e{1}\kappa(\xi)h(\xi)n(\xi)\nabla\varphi(\Gamma(\xi))\cdot n(\xi)d\xi
\end{align}
On the other hand, we also have 
\begin{align*}
&<\s_\varepsilon(u_\varepsilon,p_\varepsilon),\varphi>=\int_{\Omega\e{1}\cup\Omega\e{3}}\big(\mu u_\varepsilon(\Delta\varphi+\nabla\cdot\nabla\varphi\e{T})+p_\varepsilon\nabla\cdot\varphi\big)dV+\int_{\Omega\e{2}}\big(\mu_2 u_\varepsilon\e{2}(\Delta\varphi+\nabla\cdot\nabla\varphi\e{T})+p_\varepsilon\e{2}\nabla\cdot\varphi\big)dV+\\
&+(\mu-\mu_2)\int_{\om{+}}u_\varepsilon\e{2}(\nabla\varphi+\nabla\varphi\e{T})nds+(\mu-\mu_2)\int_{\om{-}}u_\varepsilon\e{2}(\nabla\varphi+\nabla\varphi\e{T})n_\varepsilon ds\equiv J_\varepsilon\e{1}+J_\varepsilon\e{2}+J_\varepsilon\e{3}+J_\varepsilon\e{4},\\
&<\s_0(u_0,p_0),\varphi>=\int_{\Omega\e{1}\cup\Omega\e{3}}\big(\mu u_0(\Delta\varphi+\nabla\cdot\nabla\varphi\e{T})+p_0\nabla\cdot\varphi\big)dV+\int_{\Omega\e{2}}\big(\mu u_0(\Delta\varphi+\nabla\cdot\nabla\varphi\e{T})+p_0\nabla\cdot\varphi\big)dV\\
&\equiv J_0\e{1}+J_0\e{2}.
\end{align*}
where we recall that $u_\varepsilon\e{k}$, $p_\varepsilon\e{k}$ denote the restriction of $u$ and $p$ in the domain $\Omega_k$ for $k=1,2,3$. Moreover, the function $p_0$ in $J_0\e{2}$ must be chosen as the one in $D_1$.

Using the Lebesgue Dominated Convergence Theorem we obtain that 
\begin{equation}\label{limiteop}
\lim_{\varepsilon\to 0}\frac{J_\varepsilon\e{1}-J_0\e{1}}{\varepsilon}=<\s_0(u_1,p_1),\varphi>
\end{equation}

Using integration by parts,
\begin{align*}
&J_\varepsilon\e{2}=-\mu_2\int_{\Omega\e{2}}\nabla u_\varepsilon\e{2}(\nabla\varphi+\nabla\varphi\e{T})dV+\mu_2\int_{\om{+}}u_\varepsilon\e{2}(\nabla\varphi+\nabla\varphi\e{T})nds+\\
&+\mu_2\int_{\om{-}}u_\varepsilon\e{2}(\nabla\varphi+\nabla\varphi\e{T})n_\varepsilon ds+\int_{\Omega_2}p_\varepsilon\e{2}\nabla\cdot\varphi dV
\end{align*}
therefore,
\begin{align*}
&J_\varepsilon\e{2}+J_\varepsilon\e{3}+J_\varepsilon\e{4}=-\mu_2\int_{\Omega\e{2}}\nabla u_\varepsilon\e{2}(\nabla\varphi+\nabla\varphi\e{T})dV+\mu\int_{\om{+}}u_\varepsilon\e{2}(\nabla\varphi+\nabla\varphi\e{T})nds+\mu\int_{\om{-}}u_\varepsilon\e{2}(\nabla\varphi+\nabla\varphi\e{T})n_\varepsilon ds\\
&+\int_{\Omega_2}p_\varepsilon\e{2}\nabla\cdot\varphi dV
\end{align*}
Thus, integrating by parts again
\begin{align*}
&J_\varepsilon\e{2}+J_\varepsilon\e{3}+J_\varepsilon\e{4}=(\mu-\mu_2)\int_{\Omega\e{2}}\nabla u_\varepsilon\e{2}(\nabla\varphi+\nabla\varphi\e{T})dV+\mu\int_{\Omega\e{2}}u_\varepsilon\e{2}(\Delta\varphi+\nabla\cdot\nabla\varphi\e{T})dV+\int_{\Omega_2}p_\varepsilon\e{2}\nabla\cdot\varphi dV
\end{align*}
Hence, 
\begin{align}\nonumber
&J_\varepsilon\e{2}+J_\varepsilon\e{3}+J_\varepsilon\e{4}-J_0\e{2}=(\mu-\mu_2)\int_{\Omega\e{2}}\nabla u_\varepsilon\e{2}(\nabla\varphi+\nabla\varphi\e{T})dV+\mu\int_{\Omega\e{2}}(u_\varepsilon\e{2}-u_0)(\Delta\varphi+\nabla\cdot\nabla\varphi\e{T})dV\\\label{sumajotas}
&+\int_{\Omega\e{2}}(p_\varepsilon\e{2}-p_0)\nabla\cdot\varphi dV\equiv I_\varepsilon\e{1}+I_\varepsilon\e{2}+I_\varepsilon\e{3}
\end{align}

We now use that, by assumption, for any $x_0\in\om{+}$ we can approximate $u_\varepsilon$ at distances of order $\varepsilon$ of $x_0$ by means of \eqref{cambiovariable}, \eqref{v1}-\eqref{presionP12} and \eqref{mkpk}. In particular, we use that 
\begin{align}\label{limitgrad}
\nabla u_\varepsilon(x)\to B_2\quad\textit{for}\quad\abs{x-x_0}\leq C\varepsilon,\quad x\in\Omega_2
\end{align}
with $B_2$ given by \eqref{matricesBk}, \eqref{matrizB2} and \eqref{mkpk}. Moreover, our assumption of the functions $(u_\varepsilon,p_\varepsilon)$ imply, using \eqref{cambiovariable}, \eqref{presionesP1} and \eqref{presionP12}, that $p_\varepsilon-p_0=\frac{P_1\e{3}-P_1\e{1}}{2}-2(\mu-\mu_2)n(\xi_0)\cdot B_1 n(\xi_0)+o(1)$ for $\abs{x-x_0}\leq C\varepsilon$ whence
\begin{equation}\label{limitp}
p_\varepsilon-p_0\to\kappa(\xi_0)-2(\mu-\mu_2)n(\xi_0)\cdot B_1 n(\xi_0)\quad\textit{for}\quad\abs{x-x_0}\leq C\varepsilon,\quad x\in\Omega_2
\end{equation}
where $B_1=\nabla u_0\e{1}(x_0)$.

Hence, using that the thickness of the domain $\Omega_2$ is of order $\varepsilon$ (c.f. \eqref{gamaepsilon}) and \eqref{limitgrad}, it follows that
\begin{equation}\label{limite3}
\lim_{\varepsilon\to 0}\frac{I_\varepsilon\e{1}}{\varepsilon}=(\mu-\mu_2)\int_{\om{+}}hB_2(\nabla\varphi+\nabla\varphi\e{T})ds
\end{equation}
with $B_2$ as in \eqref{B2funcnabla}. By assumption $\frac{u_\varepsilon-u_0}{\varepsilon}$ is bounded then using again that the thickness of $\Omega_2$ is of order $\varepsilon$ we readily obtain  
\begin{equation}\label{limite4}
\lim_{\varepsilon\to 0}\frac{I_\varepsilon\e{2}}{\varepsilon}=0.
\end{equation}
Finally, \eqref{gamaepsilon} and \eqref{limitp} yield

\begin{equation}\label{limite5}
\lim_{\varepsilon\to 0}\frac{I_\varepsilon\e{3}}{\varepsilon}=\int_{\om{+}}h\kappa\nabla\cdot\varphi ds-2(\mu-\mu_2)\int_{\om{+}}h(n\cdot B_2 n)\nabla\cdot\varphi ds
\end{equation}
where we use that $n\cdot B_1 n=n\cdot B_2 n$.

Combining \eqref{limiteop},\eqref{sumajotas}, \eqref{limite3}-\eqref{limite5} we can compute the limit as $\varepsilon\to 0$ of the left hand side of \eqref{restases0}. The limit of the right hand side has been computed in \eqref{limitedcha}. Then, 
\begin{align}\nonumber
&<\s_0(u_1,p_1),\varphi>+(\mu-\mu_2)\int_{\om{+}}hB_2(\nabla\varphi+(\nabla\varphi)\e{T})ds+\int_{\om{+}}h\kappa\nabla\cdot\varphi ds-2(\mu-\mu_2)\int_{\om{+}}h(n\cdot B_1 n)\nabla\cdot\varphi ds\\\label{limitefinal}
&=\int_{\om{+}}\dx\e{2}hn\cdot\varphi ds-\int_{\om{+}}\kappa\dx{}h\tau\cdot\varphi ds-\int_{\om{+}}\kappa\e{2} h n\cdot\varphi ds+\int_{\om{+}}h\kappa n\cdot\nabla\varphi n ds
\end{align}

Using that $\nabla\cdot\varphi=tr(\nabla\varphi)=\tau\cdot\nabla\varphi\tau+n\cdot\nabla\varphi n$ as well as \eqref{phi11} we obtain
\begin{align*}
&\int_{\om{+}}\kappa h n\cdot\nabla\varphi nds-\int_{\om{+}}\kappa h\nabla\cdot\varphi ds=-\int_{\om{+}}\kappa h\tau\cdot\nabla\varphi\tau ds=\int_{\om{+}}\dx(\kappa h)\tau\cdot\varphi ds+\int_{\om{+}}\kappa\e{2}hn\cdot\varphi ds,
\end{align*}
where we have integrated by parts in order to derive the first integral on the right. Plugging this identity into \eqref{limitefinal} we obtain \eqref{u1dist} after some simplifications. 
The condition $u_1=0$ in $\partial D$ follows from \eqref{cond1} as well as the fact that $u_0=U$ in $\partial D$.
\end{proof}

Equation \eqref{u1dist} can be interpreted as an equation for $(u_1,p_1)$ in the sense of distributions. In particular assuming that $(u_1,p_1)$ are smooth away of the curve $\om{+}$, we can obtain suitable jump conditions for $u_1$, $\nabla u_1$ and $p_1$ along this curve.

\begin{lem}
Let be  $(u_0,p_0)\in [C\e{2}(\Omega_{lim})\cap X\e{1}(\Omega_{lim})]\times [C\e{1}(\Omega_{lim})\cap X\e{0}(\Omega_{lim})]$ with $u_0=U$ in $\partial D$ satisfying \eqref{stokesu0} and \eqref{condicionesu0}. Suppose that $(u_1,p_1)\in [C\e{2}(\Omega_{lim})\cap X\e{1}(\Omega_{lim})]\times [C\e{1}(\Omega_{lim})\cap X\e{0}(\Omega_{lim})]$ with $u_1=0$ in $\partial D$. Then \eqref{u1dist} is satisfied for every $\varphi\in C_c\e{\infty}(\Omega)$ if and only if $(u_1,p_1)$ solves 
\begin{equation}\label{stokesu1}
\mu\Delta u_1=\nabla p_1,\quad \nabla\cdot u_1=0\quad\textit{in}\quad \Omega_{lim}
\end{equation} with the following boundary conditions:
\begin{equation}\label{bcu1}
\begin{cases}
(u_1\e{3}-u_1\e{1})\cdot\tau=\frac{\mu-\mu_2}{\mu}h(\tau B_{2}\cdot n+nB_{2}\cdot\tau)\quad\textit{in}\quad\om{+}\\
(u_1\e{3}-u_1\e{1})\cdot n=0\quad\textit{in}\quad\om{+}\\
\tau\cdot(\Sigma\e{3}_1-\Sigma\e{1}_1) n=h\dx{}\kappa-4(\mu-\mu_2)\dx(hn\cdot B_2 n)\quad\textit{in}\quad\om{+}\\
n\cdot(\Sigma\e{3}_1-\Sigma\e{1}_1)n=\dx\e{2}h-4(\mu-\mu_2)\kappa h n\cdot B_2 n\quad\textit{in}\quad\om{+}
\end{cases}
\end{equation}
where $\Sigma_1\e{k}=\mu(\nabla u_1\e{k}+(\nabla u_1\e{k})\e{T})-p_1\e{k}I$ for $k=1,3$ and $B_2$ is as in \eqref{B2funcnabla}.
\end{lem}
\begin{proof}

Integrating by parts as in the proof of Lemma \ref{bcu0} (c.f. \eqref{s0cond}), we can write the operator $\s_0(u_1,p_1)$ as follows:
\begin{align}\label{s0u1cond}
&<\s_0(u_1,p_1),\varphi>=-\mu\int_{\om{+}}([u_1]\e{3}_1)\cdot(\nabla\varphi+\nabla\varphi\e{T})nds+\int_{\om{+}}\tau(\Sigma_1\e{3}-\Sigma_1\e{1})\cdot n(\tau\cdot\varphi)ds+\int_{\om{+}}n(\Sigma_1\e{3}-\Sigma_1\e{1})\cdot n(n\cdot\varphi)ds
\end{align}
where $u_1\e{3}-u_1\e{1}=[u_1]\e{3}_1$.

On the other hand, the first two terms on the right hand side of \eqref{u1dist} can be written as
\begin{align*}
&-(\mu-\mu_2)\int_{\om{+}}h(B_2-n\cdot B_2 nI):(\nabla\varphi+\nabla\varphi\e{T})ds
\end{align*}
Therefore using Lemma \ref{lemavarphi} with $M=(\mu-\mu_2)h(B_2-n\cdot B_2 nI)-\mu[u_1]\e{3}_1\otimes n$, $a=\tau\cdot(\Sigma_1\e{3}-\Sigma_1\e{1})n-h\dx{}\kappa$ and $b=n\cdot(\Sigma_1\e{3}-\Sigma_1\e{1})n-\dx\e{2}h$, as well as the fact that $tr(B_2)=\tau\cdot B_2\tau+n\cdot B_2 n=0$, then $\tau\cdot B_2\tau=-n\cdot B_2 n$. Thus, we obtain the desired result. 
\end{proof}
\subsection{Formulation of the Geometric Free Boundary Problem that yields the evolution of the thin domain}

We can now determine the functions $(v_1,P_1)$ in \eqref{expansionvp}. This will allow us to compute the velocity of the interfaces $\om{+}$ and $\om{-}$ up to order $\varepsilon$. To this end, we derive suitable matching conditions for the function $(v_1,P_1)$. Assuming that the limits of the functions $u_0\e{k}(x)$, $u_1\e{k}(x)$, $\nabla u_0\e{k}(x)$, $p_0\e{k}(x)$ and $p\e{k}_1(x)$ as $x\to x_0$ at each side of $\om{+}$ exist for $k=1,3$, let $x_0\in\om{+}$, then the function $(u_0,p_0)$ satisfy \eqref{tayloru0p0} and $(u_1,p_1)$ satisfy
\begin{align}\label{tayloru1p1}
u_1\e{k}(x)=u_1\e{k}(x_0)+\cdots,\quad p_1\e{k}(x)=p_1\e{k}(x_0)+\cdots
\end{align} 
for $k=1,3$. Therefore, using \eqref{cambiovariable}, \eqref{expensionup}, \eqref{tayloru0p0} and \eqref{tayloru1p1}, we obtain the following matching condition for $(v_\varepsilon, P_\varepsilon)$:
\begin{align}\label{expansionv1conu1}
v\e{k}_\varepsilon(y)\sim u_0\e{k}(x_0)+\varepsilon \nabla u_0\e{k}(x_0)y+\varepsilon u
_1\e{k}(x_0)\cdots,\quad P\e{k}_\varepsilon(y)\sim \varepsilon p\e{k}_0(x_0)+\cdots
\end{align} 
which is valid for $dist(x,x_0)\to 0$ and $dist(y,c\e{+}\cup c\e{-})\to\infty$, for any $x_0\in\om{+}$ and $k=1,3$. Notice that due to \eqref{bcu1} we can write
\begin{equation}\label{condicu1nabla}
u_1\e{3}(x_0)-u_1\e{1}(x_0)=(1-\frac{\mu}{\mu_2})h(\xi_0)(\tau(\xi_0)(\nabla u_{0}(x_0)+(\nabla u_0(x_0))\e{T})\cdot n(\xi_0))\tau(\xi_0)
\end{equation}  Using the expansion \eqref{expansionvp}, for the terms of order 1 we recover the matching condition \eqref{limitesv0p0} whence \eqref{solucionv0P0} follows. On the other hand, $(v_1,P
_1)$ satisfy the matching conditions \eqref{matchingv1sinu1}. Then, Lemma \ref{matchingv1} implies that $(v_1,P_1)$ are given by \eqref{v1} and \eqref{P1} with $B_1=\nabla u_0\e{1}(x_0)$ and $B_3=\nabla u_0\e{3}(x_0)$. We remark that the compatibility conditions in \eqref{compatibilityconditions} hold due to \eqref{continuidadu0p0}. Moreover, \eqref{expansionv1conu1} yields
\begin{equation}\label{wk}
w_k=u_1\e{k}(x_0)
\end{equation}
for $k=1,3$. It is worth to emphasize that condition \eqref{vectoresw} yields a relation between $w_1$ and $w_3$ which is satisfied because of the condition \eqref{condicu1nabla}.

To sum up, we have obtained the following approximation for the velocity field near each point $x_0\in\om{+}$:
\begin{equation}\label{aproxuepsilon}
u_\varepsilon(x)=
\begin{cases}
u_\epsilon\e{k}(x)=u_0(x_0)+\varepsilon(\nabla u_0(x_0)(x-x_0)+u_1\e{k}(x_0))+o(\varepsilon)\quad\textit{as}\quad\varepsilon\to 0,\quad\textit{for}\quad k=1,3\\
u_\varepsilon\e{2}(x)=u_0(x_0)+\varepsilon(B_2(x-x_0)+w_2)+o(\varepsilon)\quad\textit{for}\quad\varepsilon\to 0
\end{cases}
\end{equation}
for each $x\in D\cap\Omega_k$ for $k=1,2,3$ satisfying $\abs{x-x_0}\leq C\varepsilon$, with $B_2$ and $w_2$ obtain by means of \eqref{matrizB2}, \eqref{vectoresw} and \eqref{wk}.
We recall that the evolution of the interfaces $\om{+}(t)$ and $\om{-}(t)$ can be computed using \eqref{ecevol3} and \eqref{ecevol1}.

We recall that we parametrize the curves $\om{+}(t)$ and $\om{-}(t)$ as in \eqref{gama} and \eqref{gamaepsilon} respectively, for all $t\geq 0$. More precisely, we need to determine the functions $\Gamma$, $h$ as follows. We assume that $\Gamma\in C\e{1}([0,\infty);C\e{2}(\R,D))$ with $\Gamma(\xi+L(t),t)=\Gamma(\xi,t)$ for all $\xi\in\R$ and $t\geq 0$ where $L(t)=length(\om{+}(t))$. We assume also that $h\in C\e{1}([0,\infty);C\e{2}(\R))$ where $h(\xi,t)=h(\xi+L(t),t)$ with $h>0$.

We now reformulate \eqref{ecevol3} and \eqref{ecevol1} as a system of partial differential equations for $\Gamma$ and $h$. We emphasize that we need to choose the function $\Gamma$ in such a way that the arc length parametrization remains for all time, i.e., $\abs{\dx{\Gamma}(\xi,t)}=1$ for all $t\geq 0$. To this end, we choose $\Gamma(\xi,t)$ with the form $\partial_t\Gamma(\xi,t)=(V_n)n+\psi(\xi,t)\tau$ with $\psi\in C\e{1}([0,\infty);C\e{1}(\R))$ and $\psi(\xi+L(t),t)=\psi(\xi,t)$ where $\psi$ provide the tangential component of the velocity  of $\Gamma$. It turns out that in order to preserve the arc length parametrization we must choose the function $\psi$ in a very specific way.
\begin{lem}\label{arclengtime}Let $\sigma(t)=\{x\in D: x=\Gamma(\xi,t), \xi\in\R\}$ a Jordan curve parametrized by $\Gamma\in C\e{1}([0,\infty);C\e{2}(\R,D))$ with $\Gamma(\xi+L(t),t)=\Gamma(\xi,t)$ for all $\xi\in\R$ and $t\geq 0$ where $L(t)=length(\sigma(t))$ and $\abs{\dx{}\Gamma(\xi,0)}=1$ for all $\xi\in\R$. Suppose that 

\begin{equation}\label{lema34eq2}
\partial_t\Gamma(\xi,t)=\phi(\xi,t)n+\psi(\xi,t) \tau
\end{equation}
where $\phi,\psi:\R\times\R\e{+}\to\R$ are periodic functions with period $L(t)$. Then $\abs{\dx{}\Gamma(\xi,t)}=1$ for $t\geq 0$ if and only if 
\begin{equation}\label{lemapsi}
\psi(\xi,t)=c(t)+\int_{0}\e{\xi}\kappa(r,t)\phi(r,t)dr\quad\textit{for}\quad t\geq 0
\end{equation}
for some function $c: \R\e{+}\to\R$.
\end{lem}
\begin{proof}
Differentiating on $t$ we obtain that the condition $\abs{\dx{\Gamma}}\e{2}=1$ for $t\geq 0$ is equivalent to $\dx{\Gamma}\cdot\dx{\partial_t\Gamma}=0$ for $t\geq 0$ with $\abs{\dx{}\Gamma(\xi,0)}=1$. Therefore, if we use \eqref{lema34eq2}, we obtain that as long as $\abs{\dx{}\Gamma(\xi,t)}\neq 0$, the identity $\dx{\Gamma}\cdot\dx{\partial_t\Gamma}=0$ holds if and only if
\begin{equation}\label{lema34eq1}
(\dx{\phi}n+\phi\dx{n}+\dx{\psi}\tau+\psi\dx{\tau})\cdot\tau=-\phi\kappa+\dx{\psi}=0
\end{equation}
where we have used \eqref{derivnormal}.
Integrating \eqref{lema34eq1} with respect to $\xi$, we obtain \eqref{lemapsi} and the lemma follows.

\end{proof}
\begin{nota} 
The function $c(t)$ is an arbitrary smooth function which gives the tangential velocity of the point $x=\Gamma(0,t)$ at any time $t\geq 0$. Note that this tangential velocity can be chosen in an arbitrary manner at one point of the curve $\sigma(t)$ but it is uniquely determined at all the other points of the curve due to \eqref{lemapsi}.
\end{nota}
Now we use \eqref{ecevol3} and \eqref{ecevol1} in order to compute the evolution equations of $\Gamma(\xi,t)$ and $h(\xi,t)$. 

Equation \eqref{ecevol3} implies $\partial_t\Gamma\cdot n=u_\varepsilon(\Gamma)\cdot n$ then using \eqref{aproxuepsilon} with $k=3$ we obtain
\begin{equation}\label{evolgamma}
\partial_t\Gamma\cdot n=u_0(\Gamma)\cdot n+\varepsilon u_1\e{3}(\Gamma)\cdot n+o(\varepsilon)
\end{equation}
Using \eqref{ecevol1} and \eqref{gamaepsilon}, we can write
\begin{equation*}
(\partial_t(\Gamma-\varepsilon hn))\cdot n_\varepsilon=u_\varepsilon(\Gamma_\varepsilon)\cdot n_\varepsilon
\end{equation*} 
Combining this identity with  \eqref{normalepsilon}, \eqref{aproxuepsilon} for $k=1$ and the fact that $\partial_t n\cdot n=0$, we arrive at
\begin{align*}
&-\partial_t\Gamma(\xi,t)\cdot n(\xi,t)-\varepsilon(\dx{h}(\xi,t)\partial_t\Gamma(\xi,t)\tau(\xi,t)-\partial_t h(\xi,t))=\\
&-u_0(\Gamma(\xi,t))n(\xi,t)+\varepsilon\Big(h(\xi,t)(n(\xi,t)\cdot\nabla u_0\e{1}(\Gamma(\xi,t)) n(\xi,t))-u_1\e{1}(\Gamma(\xi,t))\cdot n(\xi,t)-\dx{h}(u_0(\Gamma(\xi,t))\cdot\tau(\xi,t))\Big)+o(\varepsilon)
\end{align*}
Notice that we can apply \eqref{normalepsilon} (c.f. Lemma \ref{L1}) because we choose $\Gamma(\xi,t)$ in such a way that $\abs{\dx{}\Gamma(\xi,t)}=1$ for $t\geq 0$. Then Lemma \ref{arclengtime} implies that $\partial_t\Gamma(\xi,t)\cdot\tau(\xi,t)=\psi(\xi,t)$ with $\psi$ as in \eqref{lemapsi} and $\phi=u_\varepsilon\cdot n$. Henceforth, using \eqref{aproxuepsilon} we obtain that $\psi(\xi,t)=\psi_0(\xi,t)+\mathcal{O}(\varepsilon)$ with
\begin{equation}\label{psi0}
\psi_0(\xi,t)=c(t)+\int_0\e{\xi}\kappa(r,t)u_0(\Gamma(r,t),t)\cdot n(r,t)dr
\end{equation} 
Thus using  also \eqref{evolgamma} and keeping only the terms of order $\varepsilon$, we derive the following equation for $h$:
\begin{align*}
&\partial_t h(\xi,t)+(u_0(\Gamma(\xi,t))\cdot\tau(\xi,t)-\psi_0(\xi,t))\dx{h}(\xi,t)=(n(\xi,t)\nabla u_0(\Gamma(\xi,t))\cdot n(\xi,t))h(\xi,t)\\
&+(u_1\e{3}(\Gamma(\xi,t))-u_1\e{1}(\Gamma(\xi,t)))\cdot n(\xi,t)
\end{align*}
The last term in this identity vanishes due to  the second equation in \eqref{bcu1}, therefore
\begin{align}\label{hevol}
\partial_t h(\xi,t)+(u_0(\Gamma(\xi,t))\cdot\tau(\xi,t)-\psi_0(\xi,t))\dx{h}(\xi,t)=(n(\xi,t)\nabla u_0(\Gamma(\xi,t))\cdot n(\xi,t))h(\xi,t)
\end{align}

Taking the limit as $\varepsilon\to 0$ of the equations \eqref{evolgamma},\eqref{psi0} and \eqref{hevol} we obtain the following system:  
\begin{align}\label{GFBP1}
&\partial_t\Gamma=(u_0(\Gamma)\cdot n)n+\psi_0(\xi,t)\tau\\\label{GFBP2}
&\psi_0(\xi,t)=\int_0\e{\xi}\kappa(r,t)u_0(\Gamma(r,t),t)\cdot n(r,t)dr\\\label{GFBP3}
&\partial_t h(\xi,t)+(u_0(\Gamma(\xi,t))\cdot\tau(\xi,t)-\psi_0(\xi,t))\dx{h}(\xi,t)=(n(\xi,t)\nabla u_0(\Gamma(\xi,t))\cdot n(\xi,t))h(\xi,t)
\end{align} 
where $u_0$ is given by \eqref{u0dist} (c.f. alternatively \eqref{stokesu0} and \eqref{condicionesu0}).
 This model provides the evolution of the thin domain $\Omega_2$ in the rescaling limit considered in this paper. We remark that the model \eqref{GFBP1}-\eqref{GFBP3} does not depend on the function $u_1$ and therefore it is independent of the viscosity $\mu_2$ of the fluid contained in the region $\Omega_2(t)$. 
 
 Notice that we are considering $c(t)=0$ in \eqref{psi0}. This does not suppose any lost of generality because if we had kept the function $c(t)$ we will have obtained different functions  $\Gamma$ and $h$ that we will parametrize the same curves $\om{+}$ and $\om{-}$.

The system \eqref{GFBP1}-\eqref{GFBP3} has to be solved in the domain $0<\xi<L(t)$ where $L(t)$ is the length of the curve $\om{+}$, with periodic boundary conditions $\Gamma(0,t)=\Gamma(L(t),t)$, $h(0,t)=h(L(t),t)$ and if $\om{+}\in C\e{k}$ we have to impose in addition  $\dx\e{l}\Gamma(0,t)=\dx\e{l}\Gamma(L(t),t)$ for $l\leq k$. In order to have a well defined PDE problem we obtain an evolution equation for $L(t)$.

Since $\Gamma(L(t),t)=\Gamma(0,t)$, differentiating in $t$ we obtain, $\dx\Gamma(L(t),t)\partial_t L(t)+\partial_t\Gamma(L(t),t)=\partial_t\Gamma(0,t)$. Now multiplying by $\dx\Gamma(L(t),t)$ and using that  $\dx\Gamma(L(t),t)=\dx\Gamma(0,t)$, $\abs{\dx\Gamma}=1$ and \eqref{GFBP1} we arrive at $\partial_t L(t)=\psi_0(0,t)-\psi_0(L(t),t)$. Therefore, using \eqref{GFBP2}, we can conclude that the evolution equation for $L$ is
\begin{equation}\label{GFBP4}
\partial_t L(t)=-\int_0\e{L(t)}\kappa(r,t)u_0(\Gamma(r,t),t)\cdot n(r,t)dr
\end{equation}
We refer to the system \eqref{u0dist}, \eqref{GFBP1}-\eqref{GFBP4} as the \textit{Geometric Free Boundary Problem }(GFBP).

The evolution of the boundaries of the region $\Omega_2(t)$ by means of \eqref{gama} and \eqref{gamaepsilon} can be approximated taking $\Gamma$ and $h$ as a solution of the PDE problem (GFBP).
\section{Well posedness of the geometric free boundary problem}\label{S4}
In this section we prove that the Geometric Free Boundary Problem is well posed locally in time. In order to do this, we introduce a technical property that will be repeatedly use in the following.

We will denote as $f\in C\e{\infty}_{p}([0,L(t)])$ the set of functions $f\in C\e{\infty}([0,L(t)])$ such that $\partial\e{\ell}f(0)=\partial\e{\ell}f(L(t))$ for $l=0,1,\cdots$. We will denote as $H_p\e{k}([0,L(t)])$ the Sobolev spaces obtained by means of the closure of the functions $f\in C\e{\infty}_p([0,L(t)])$ in the topology generated by the norm $\norm{f}_{H_p\e{k}}\e{2}=\sum_{l=0}\e{k}\int_0\e{1}\abs{\partial\e{\ell}f}\e{2}dx$. In several cases it will be convenient to identify the functions $f\in H_p\e{k}([0,L(t)])$ with the functions $f\in H_{loc}\e{k}(\R)$ with period $L(t)$.  Notice that for all $f\in H_p\e{k}([0,L(t)])$ we have that $\partial\e{\ell}f(0)=\partial\e{\ell}f(L(t))$ for $l=0,1,\cdots,k-1$. Moreover, we define 
\begin{equation}\label{GT}
\mathcal{G}_T=\{(\xi,t):0<\xi<L(t),0\leq t<T\}.
\end{equation}

Given $\Gamma\in H_p\e{1}([0,L(t)])$ we say that the Arc-Chord condition holds if $\mathcal{F}(\Gamma)\in L^{\infty}(\R\times\R)$ where 
\begin{equation}\label{arcchordcondition}
\F(\alpha,\xi)=\frac{\xi^{2}}{\abs{\Gamma(\alpha)-\Gamma(\alpha-\xi)}^{2}},\quad\alpha,\xi\in\R\quad\textit{for}\quad\xi\neq 0
\end{equation}
with
\begin{displaymath}
\F(\alpha,0)=\frac{1}{\abs{\partial_{\alpha}\Gamma(\alpha)}^{2}}.
\end{displaymath}
In these formulas as well as in similar expressions in the rest of the paper, we extend $\Gamma$ to the whole real line in a periodic manner with period $L(t)$.

\subsection{Well posedness for the curve $\Gamma$}
In this subsection we prove that the system of equations \eqref{GFBP1}, \eqref{GFBP2} and \eqref{GFBP4} is well posed locally in time.
\begin{thm}\label{thexist}
Suppose that $U\in W\e{1,\infty}([0,T],H\e{\frac{1}{2}}(\partial D))$, $\int_{\partial D}U\cdot n=0$. Let $L_0>0$, $\Gamma_{0}\in H_p^{k}([0,L_0];\R\e{2})$ for $k\ge 3$, $\Gamma_0([0,L_0])\subset D$ such that $\abs{\dx\Gamma_0}=1$ for all $\xi\in[0,L_0]$ and $\mathcal{F}(\Gamma_0)\in L\e{\infty}(S\e{1})$. Then, there exists $T=T(\Gamma_{0},U)$ as well as functions $L$, $\Gamma$ and $(u_0,p_0)$ with the following properties:
\begin{itemize}
\item[i.] $L\in C\e{1}([0,T])$ with $L(0)=L_0$.
\item[ii.] $u_0\in C\e{1}([0,T],H\e{1}(D))$ with $u_0=U$ in $\partial D$ and $p_0\in W\e{1,\infty}([0,T], L\e{2}(D))$. The pair $(u_0,p_0)$ satisfies \eqref{u0dist}. 
\item[iii.] $\Gamma\in W\e{1,\infty}([0,T],H\e{2}_p([0,L(t)]))\cap L\e{\infty}([0,T],H\e{k}_p([0,L(t)]))$ with $\Gamma(\cdot,0)=\Gamma_0(\cdot)$ and 

$\sup_{0\leq t\leq T}\norm{\partial_t\Gamma(\cdot,t)}_{H\e{k-2}([0,L(t)])}<\infty$. The function $(\xi,t)\to u_0(\Gamma(\xi,t))$ is continuous in $\mathcal{G}_T$ and it is Lipschitz continuous in the variable $\xi$ for every $t\in[0,T]$.
\end{itemize}
 The triplet $(\Gamma, L, u_0)$ satisfies \eqref{GFBP1}, \eqref{GFBP2} and \eqref{GFBP4} for $(\xi,t)\in\mathcal{G}_T$.
\end{thm}
We define 
\begin{equation}\label{GTtilde}
\mathcal{\tilde{G}}_T=\{(\eta,t):\eta\in S\e{1},t\in [0,T]\}
\end{equation}
Due to the periodicity of all the functions involved in the arguments we can identify $S\e{1}$ with the interval $[-\frac{1}{2},\frac{1}{2}]$ or with any other interval of length $1$. We use this identification repeatedly in the following arguments. 
It is convenient to reformulate of the problem \eqref{GFBP1}, \eqref{GFBP2} and \eqref{GFBP4} in a fixed domain.
\begin{prop}\label{reformulacion}
Suppose that $U$, $L_0$ and $\Gamma_0$ are as in the statement of Theorem \ref{thexist}. 
There exists a triplet $(\Gamma,L,u_0)$ as in Theorem \ref{thexist} if and only if, there exists a triplet $(\tilde{\Gamma},L,u_0)$ with the following properties
\begin{itemize}
\item [a)]$L$ and $(u_0,p_0)$ verify the properties i. and ii. of Theorem \ref{thexist}, respectively.
\item [b)]$\tilde{\Gamma}\in W\e{1,\infty}([0,T],H\e{2}(S\e{1}))\cap L\e{\infty}([0,T],H\e{k}(S\e{1}))$ with $\tilde{\Gamma}(\cdot,0)=\tilde{\Gamma}_0(\cdot)$ and $\sup_{0\leq t\leq T}\norm{\partial_t\tilde{\Gamma}(\cdot,t)}_{H\e{k-2}(S\e{1})}<\infty$. The function $(\eta,t)\to u_0(\tilde{\Gamma}(\eta,t))$ is continuous in $\mathcal{\tilde{G}}_T$ and it is Lipschitz continuous in the variable $\eta$ for every $t\in[0,T]$.
\item[c)] The following equations hold in $\mathcal{\tilde{G}}_T$.

\begin{align}\label{reformGFBP1}
&\partial_t\tilde{\Gamma}(\eta,t)=(u_0(\tilde{\Gamma}(\eta,t),t)\cdot \tilde{n}(\eta,t))\tilde{n}(\eta,t)+\tilde{\psi}(\eta,t)\tilde{\tau}(\eta,t)+(\partial_t L(t))\eta\tilde{\tau}(\eta,t)\\\label{reformGFBP6}
&\partial_t L(t)=-L(t)\int_0\e{1}\tilde{\kappa}(\eta,t)u_0(\tilde{\Gamma}(\eta,t),t)\cdot\tilde{n}(\eta,t)d\eta
\end{align}
where
\begin{align}\label{reformGFBP2}
&\tilde{\psi}(\eta,t)=L(t)\int_0\e{\eta}\tilde{\kappa}(r,t)u_0(\tilde{\Gamma}(r,t),t)\cdot\tilde{n}(r,t)dr\\\label{reformGFBP3}
&\tilde{\tau}(\eta,t)=\frac{1}{L(t)}\partial_\eta\tilde{\Gamma}(\eta,t)\\\label{reformGFBP4}
&\tilde{n}(\eta,t)=\frac{1}{L(t)}(\partial_\eta\tilde{\Gamma}(\eta,t))\e{\perp}\\\label{reformGFBP5}
&\tilde{\kappa}(\eta,t)=\frac{1}{L\e{2}(t)}\partial_\eta\e{2}\tilde{\Gamma}(\eta,t)\cdot\tilde{n}(\eta,t)
\end{align} 

\end{itemize}
\end{prop}
\begin{proof}
It is straightforward using the change of variables $\eta=\frac{\xi}{L(t)}$ and  writing $\tilde{\Gamma}(\eta,t)=\Gamma(\xi,t)$. Notice that $\tilde{\psi}(\eta,t)=\psi(\xi,t)$, $\tilde{\tau}(\eta,t)=\tau(\xi,t)$, $\tilde{n}(\eta,t)=n(\xi,t)$ and $\tilde{\kappa}(\eta,t)=\kappa(\xi,t)$.
\end{proof}
\begin{nota}
Notice that for the original problem \eqref{GFBP1},\eqref{GFBP2} and \eqref{GFBP4} we have $\abs{\dx\Gamma(\xi,t)}=1$ and for the reformulated problem \eqref{reformGFBP1}-\eqref{reformGFBP6}, the identity $\abs{\partial
_\eta\tilde{\Gamma}(\eta,t)}=L(t)$ holds.
\end{nota}

\begin{thm}\label{thexistreform}
Suppose that $U\in W\e{1,\infty}([0,T],H\e{\frac{1}{2}}(\partial D))$, $\int_{\partial D}U\cdot n=0$. Let $L_0>0$, $\tilde{\Gamma}_{0}\in H_p^{k}(S\e{1};\R\e{2})$ for $k\ge 3$, $\tilde{\Gamma}_0(S\e{1})\subset D$ such that $\abs{\da{}\tilde{\Gamma}_0}=L(t)$ for all $\eta\in S\e{1}$ and $\mathcal{F}(\tilde{\Gamma}_0)\in L\e{\infty}(S\e{1})$. Then, there exists $T=T(\tilde{\Gamma}_{0},U)$ as well as functions $\tilde{L}$, $\tilde{\Gamma}$ and $(u_0,p_0)$ with the following properties:
\begin{itemize}
\item [i.] $L\in C\e{1}([0,T])$ with $L(0)={L}_0$.
\item [ii.]$u_0\in C([0,T],H\e{1}(D))$ with $u_0=U$ in $\partial D$ and $p_0\in C([0,T]; L\e{2}(D))$. The pair $(u_0,p_0)$ satisfies \eqref{u0dist}.
\item [iii.]$\tilde{\Gamma}\in W\e{1,\infty}([0,T],H\e{2}(S\e{1}))\cap L\e{\infty}([0,T],H\e{k}(S\e{1}))$ for $k\geq 3$, $\tilde{\Gamma}(\cdot,t)\in H_p\e{k}(S\e{1})$ for all $t\in[0,T]$ with $\tilde{\Gamma}(\cdot,0)=\tilde{\Gamma}_0(\cdot)$ and $\sup_{0\leq t\leq T}\norm{\partial_t\tilde{\Gamma}(\cdot,t)}_{H\e{k-2}(S\e{1})}\leq\infty$. The function $(\eta,t)\to u_0(\tilde{\Gamma}(\eta,t))$ is continuous in $\mathcal{\tilde{G}}_T$.
\end{itemize}
 The triplet $(\tilde{\Gamma}, \tilde{L}, u_0)$ satisfies \eqref{reformGFBP1}-\eqref{reformGFBP5} for $(\xi,t)\in\mathcal{\tilde{G}}_T$.
\end{thm}

We first state some of the main properties of the velocity field $u_0$. 
\begin{lem}\label{lemu0galdi}
Suppose that $U\in W\e{1,\infty}([0,T],H\e{\frac{1}{2}}(\partial D))$, $\int_{\partial D}U\cdot n=0$. Let $L>0$, $\tilde{\Gamma}\in C([0,T];H^{k}(S\e{1};\R\e{2}))$ for $k\ge 3$, $\tilde{\Gamma}(S\e{1})\subset D$ for all $t\in[0,T]$ such that $\abs{\da{}\tilde{\Gamma}}=L(t)$ for all $\eta\in S\e{1}$.  Then there exist $u_0\in C([0,T],H\e{1}(D))$ with $u_0=U$ in $\partial D$ and $p_0\in C([0,T]; L\e{2}(D))$ such that $(u_0,p_0)$ satisfies \eqref{u0dist} for each $t\in[0,T]$. The following estimates holds
\begin{equation}\label{cotau0galdi}
\norm{u_0(\cdot,t)}_{H\e{1}(D)}\leq C\norm{U(\cdot,t)}_{H\e{\frac{1}{2}}(\partial D)}+K\norm{\tilde{\Gamma}(\cdot,t)}_{H\e{2}(S\e{1})}\quad\textit{for}\quad t\in[0,T]
\end{equation}
with $C$ is a numerical constant and $K=K(\norm{\tilde{\Gamma}}_{H\e{2}(S\e{1})}, L)$.
Moreover, we can write 
\begin{equation}\label{solucionfundamentalu0}
(u_0,p_0)=(\bar{u}_0,\bar{p}_0)+(w,q)
\end{equation}
where
\begin{equation}\label{equ0barra}
\bar{u}_0(x,t)=-\frac{L(t)}{2\pi\mu}\int_{S\e{1}}\tilde{\kappa}(r)\log\abs{x-\tilde{\Gamma}(r)}\tilde{n}(r)dr+\frac{L(t)}{2\pi\mu}\int_{S\e{1}}\tilde{\kappa}(r)\frac{(x-\tilde{\Gamma}(r))\cdot \tilde{n}(r)}{\abs{x-\tilde{\Gamma}(r)}\e{2}}(x-\tilde{\Gamma}(r)) dr\quad\textit{for}\quad x\in D
\end{equation}
\begin{displaymath}
\bar{p}_0(x,t)=\frac{L}{\pi}\int_{S\e{1}}\tilde{\kappa}(r)\frac{(x-\tilde{\Gamma}(r))\cdot\tilde{n}(r)}{\abs{x-\tilde{\Gamma}(r)}\e{2}}dr
\end{displaymath}
and $(w,q)$ is a weak solution of 
\begin{equation}\label{stokesw}
\mu\Delta w=\nabla q\quad,\nabla\cdot w=0\quad\textit{in}\quad D
\end{equation}
with 
\begin{equation}\label{wsolucionu0}
w=U-\bar{u}_0\quad\textit{in}\quad \partial D
\end{equation}
In addition, we have  
\begin{equation}\label{cotaw}
\norm{w}_{H\e{j}(\Omega_\delta)}\leq C_j\norm{U(\cdot,t)}_{H\e{\frac{1}{2}}_{p}(\partial D)}+\frac{C_j}{L}\norm{\tilde{\Gamma}}_{H\e{3}(S\e{1})}
\end{equation}
 for $j=1,2,\cdots$ and for all $t\in[0,T]$, where we have defined the domain $\Omega_\delta\equiv\om{+}+B_{\delta}(0)$ for any $\delta<\inf_{t\in[0,T]}\frac{1}{2}dist(\om{+},\partial D)$. 
\end{lem}

\begin{proof}
In order to prove the existence of $(u_0,p_0)$ for $t\in[0,T]$ we apply Theorem IV.1.1. of \cite{galdi} where the boundary value problem for 
$\Delta u_0-\nabla p_0=f, \nabla\cdot u_0=0$ where $f\in H\e{-1}_0(D)$ is considered . In our case we apply this result with $f=2\kappa n\delta_{\om{+}}$. We recall that
\begin{align*}
\norm{f}_{H\e{-1}(D)}=\sup_{\varphi\in H\e{1}_0(D),\norm{\varphi}_{H\e{1}_0(D)}=1}\abs{<f,\varphi>}
\end{align*}
Therefore if $f=2\kappa n\delta_{\om{+}}$  we have
\begin{align*}
\abs{<f,\varphi>}=2\abs{\int_{\om{+}}\kappa n\cdot\varphi ds}\leq \norm{\kappa}_{L\e{2}(\om{+})}\norm{\varphi}_{L\e{2}(\om{+})}\leq \frac{C}{L\e{2}}\norm{\tilde{\Gamma}}_{H\e{2}(S\e{1})}\norm{\varphi}_{H\e{1}_0(D)}
\end{align*}
where we use that $\om{+}$ is Lipchitz as well as the trace theorem II.4.1. in \cite{galdi}. Thus
\begin{align*}
\norm{f}_{H\e{-1}(D)}\leq K\norm{\tilde{\Gamma}}_{H\e{2}(S\e{1})}\quad\textit{for}\quad t\in[0,T]
\end{align*}
where $K=K(\norm{\tilde{\Gamma}}_{H\e{2}(S\e{1})}, L)$. Henceforth we deduce \eqref{cotau0galdi}. 

In order to prove the continuity of $u_0$ in time  we need to show 
\begin{align*}
\norm{u_0(\cdot,t_1)-u_0(\cdot,t_2)}_{H\e{1}(D)}\to 0\quad\textit{as}\quad\abs{t_1-t_2}\to 0
\end{align*}
To this end, due to the linearity of the problem and Theorem IV. 1.1. in \cite{galdi},  it is enough to show that 
\begin{align*}
&\norm{U(\cdot,t_1)-U(\cdot,t_2)}_{H\e{-\frac{1}{2}}(\partial D)}\to 0\quad\textit{as}\quad\abs{t_1-t_2}\to 0\\
&\norm{f_1-f_2}_{H\e{-1}(D)}\to 0\quad\textit{as}\quad\abs{t_1-t_2}\to 0\\
\end{align*}
where $f_j=2\kappa(\cdot,t_j)n(\cdot,t_j)\delta_{\om{+}(\cdot,t_j)}$ for $j=1,2$. The first result follows from the regularity of $U$ and the second we need to proceed in the following way:
\begin{align}\nonumber
&<f_1-f_2,\varphi>=2L\int_{S\e{1}}\tilde{\kappa}(\eta,t_1)\tilde{n}(\eta,t_1)\cdot\varphi(\tilde{\Gamma}(\eta,t_1))d\eta-2L\int_{S\e{1}}\tilde{\kappa}(\eta,t_2)\tilde{n}(\eta,t_2)\cdot\varphi(\tilde{\Gamma}(\eta,t_2))d\eta\\\nonumber
&=2L\int_{S\e{1}}(\tilde{\kappa}(\eta,t_1)\tilde{n}(\eta,t_1)-\tilde{\kappa}(\eta,t_2)\tilde{n}(\eta,t_2))\cdot\varphi(\tilde{\Gamma}(\eta,t_1))d\eta+2L\int_{S\e{1}}\tilde{\kappa}(\eta,t_2)\tilde{n}(\eta,t_2)(\varphi(\tilde{\Gamma}(\eta,t_1))-\varphi(\tilde{\Gamma}(\eta,t_2))d\eta\\\label{J1J2}
&\equiv J_1+J_2
\end{align}
To estimate $J_1$ it is enough to use Cauchy-Schwartz as well as $$\norm{\tilde{\kappa}(\cdot,t_1)\tilde{n}(\cdot,t_1)-\tilde{\kappa}(\cdot,t_2)\tilde{n}(\eta,t_2)}_{L\e{2}(S\e{1})}\leq \frac{C}{L\e{3}}\sup_{t\in[0,T]}\norm{\tilde{\Gamma}}_{H\e{3}(S\e{1})}\norm{\tilde{\Gamma}(\cdot,t_1)-\tilde{\Gamma}(\cdot,t_2)}_{H\e{3}(S\e{1})}$$
therefore, $J_1\to 0$ as $\abs{t_1-t_2}\to 0$.

We now estimate $J_2$. Using the mapping $(\eta,r)\to\tilde{\Gamma}(\eta,t_1)+r\tilde{n}(\eta,t_1)$ and the invert function theorem it follows that for $\abs{r}\leq\delta$ for $\delta$ small, the variables $(\eta,r)$ are a local coordinate system of $\om{+}(t_1)$. In particular, if $\abs{t_1-t_2}$ is sufficiently small if follows that for any $\eta\in S\e{1}$ there exist $r=R(\eta,t_1,t_2)$ such that $\tilde{\Gamma}(\eta,t_1)+R(\eta)\tilde{n}(\eta,t_1)=\tilde{\Gamma}(\eta,t_2)$. We define $\gamma(r)=\int_{S\e{1}}\tilde{\kappa}(\eta,t_2)\tilde{n}(\eta,t_2)\cdot\varphi(\tilde{\Gamma}(\eta,t_1)+rR(\eta)\tilde{n}(\eta,t_1))d\eta$. Notice that $\gamma(0)=\int_{S\e{1}}\tilde{\kappa}(\eta,t_2)\tilde{n}(\eta,t_2)\cdot\varphi(\tilde{\Gamma}(\eta,t_1))d\eta$ and $\gamma(1)=\int_{S\e{1}}\tilde{\kappa}(\eta,t_2)\tilde{n}(\eta,t_2)\cdot\varphi(\tilde{\Gamma}(\eta,t_2))d\eta$. Our next goal is to show that for any $\varphi\in H\e{1}(D)$ we have $\gamma\in H\e{1}([0,1])\subset C([0,1])$.

\begin{align*}
&\int_{0}\e{1}\abs{\partial_r\gamma(r)}\e{2}dr\leq \int_{S\e{1}}(\int_{S\e{1}}\tilde{\kappa}(\eta,t_2)\tilde{n}(\eta,t_2)\cdot\nabla\varphi(\tilde{\Gamma}(\eta,t_1)+rR(\eta)\tilde{n}(\eta,t_1))\tilde{n}(\eta,t_1)rR(\eta)d\eta)\e{2} dr\\
&\leq C\norm{\tilde{\Gamma}}_{H\e{3}(S\e{1})}\int_{0}\e{1}\int_{S\e{1}}\abs{\nabla\varphi(\tilde{\Gamma}(\eta,t_1)+rR(\eta)\tilde{n}(\eta,t_1))}\e{2}r\e{2}(R(\eta))\e{2}d\eta dr\\
&=C\norm{\tilde{\Gamma}}_{H\e{3}(S\e{1})}\int_{S\e{1}}\int_0\e{R(\eta)}\abs{\nabla\varphi(\tilde{\Gamma}(\eta,t_1)+\xi\tilde{n}(\eta,t_1))}\e{2}\frac{\xi\e{2}}{\abs{R(\eta)}}d\xi d\eta\leq \frac{C}{L}\norm{\tilde{\Gamma}}_{H\e{3}(S\e{1})}\norm{R}_{L\e{\infty}(S\e{1})}\int_{D}\abs{\nabla\varphi}\e{2}dV
\end{align*}
where we have used that $\norm{\tilde{\kappa}}_{L\e{\infty}(S\e{1})}\leq C\norm{\tilde{\Gamma}}_{H\e{3}(S\e{1})}$, the change of variables $\xi=rR(\eta)$ and $\abs{\frac{r\e{2}(R(\eta))\e{2}}{R(\eta)}}=\frac{\xi\e{2}}{\abs{R(\eta)}}\leq\frac{(R(\eta))\e{2}}{\abs{R(\eta)}}\leq\abs{R(\eta)}\leq \sup_{\eta\in S\e{1}}\abs{R(\eta)}=\norm{R}_{L\e{\infty}(S\e{1})}$. 

Using Morrey's inequality and $\norm{\varphi}_{H\e{1}}=1$ we obtain that
\begin{align*}
\abs{\gamma(1)-\gamma(0)}\leq C(\int_{0}\e{1}\abs{\partial_r\gamma(r)}\e{2}dr)\e{\frac{1}{2}}\leq \frac{C}{L\e{\frac{1}{2}}}\norm{\tilde{\Gamma}}\e{\frac{1}{2}}_{H\e{3}(S\e{1})}\norm{R}\e{\frac{1}{2}}_{L\e{\infty}(S\e{1})}
\end{align*}
Using that $\tilde{\Gamma}\in C([0,T],H\e{k}_p(S\e{1}))$ for $k\geq 3$ we have $\norm{R}_{L\e{\infty}(S\e{1})}\to 0$ as $\abs{t_1-t_2}\to 0$, therefore $J_2\to 0$ as $\abs{t_1-t_2}\to 0$.

Combining this with the analogous result for $J_1$ it then follows from \eqref{J1J2} that
\begin{displaymath}
\norm{f_1-f_2}_{H\e{-1}(D)}\to 0\quad\textit{as}\quad\abs{t_1-t_2}\to 0
\end{displaymath}

In order to prove \eqref{solucionfundamentalu0}-\eqref{wsolucionu0} we use the fundamental solution for the Stokes Equation (c.f. IV.2.4. in \cite{galdi}). It then follows that $(\bar{u}_0,\bar{p}_0)$ satisfies $\Delta\bar{u}_0-\nabla\bar{p}_0=2\kappa n\delta_{\om{+}}$ and $\nabla\cdot\bar{u}_0=0$ in $\R\e{2}$. Thus, using \eqref{solucionfundamentalu0}, it follows that $(w,q)$ solves \eqref{stokesw}-\eqref{wsolucionu0}. We deduce \eqref{cotaw} using interior regularity estimates for Stokes equation (c.f. Theorem IV.4.1. in \cite{galdi}).
\end{proof}
\begin{nota} Notice that \eqref{cotau0galdi} as well as the trace theorem implies
\begin{equation}\label{cotau0curva}
\norm{u_0(\tilde{\Gamma}(\cdot))}_{L\e{2}(S\e{1})}\leq K(\norm{U(\cdot,t)}_{H\e{\frac{1}{2}}(\partial D)}+\norm{\tilde{\Gamma}(\cdot,t)}_{H\e{2}(S\e{1})})\quad\textit{for}\quad t\in[0,T]
\end{equation}
where $K=K(\norm{\tilde{\Gamma}}_{H\e{2}(S\e{1})}, L)$.
\end{nota}
\begin{nota} Lebesgue dominated convergence theorem implies that $\bar{u}_0(\cdot,t)$ defined in \eqref{equ0barra} is continuous in $D$ for all $t\in[0,T]$. In particular, we have 
\begin{equation}\label{u0gamma}
\bar{u}_0(\tilde{\Gamma}(\eta),t)=-\frac{L(t)}{2\pi\mu}\int_{0}\e{1}\tilde{\kappa}(r)\log\abs{\tilde{\Gamma}(\eta)-\tilde{\Gamma}(r)}\tilde{n}(r)dr+\frac{L(t)}{2\pi\mu}\int_{0}\e{1}\tilde{\kappa}(r)\frac{(\tilde{\Gamma}(\eta)-\tilde{\Gamma}(r))\cdot \tilde{n}(r)}{\abs{\tilde{\Gamma}(\eta)-\tilde{\Gamma}(r)}\e{2}}(\tilde{\Gamma}(\eta)-\tilde{\Gamma}(r)) dr\equiv F+G
\end{equation}
\end{nota}
In order to prove Theorem \ref{thexistreform} we show first some auxiliary results which will allow us to estimate the most singular terms in the equations \eqref{reformGFBP1}-\eqref{reformGFBP6}. 
\begin{lem}\label{lemcotadificil}
Let $L_0>0$, $\tilde{\Gamma}\in H^{k}(S\e{1};\R\e{2})$ for $k\ge 3$, $\tilde{\Gamma}(S\e{1})\subset D$ such that $\abs{\da{}\tilde{\Gamma}}=L$ for all $\eta\in S\e{1}$, $f\in C\e{\infty}_{p}(S\e{1})$ and $\mathcal{F}(\tilde{\Gamma})\in L\e{\infty}(S\e{1})(\R\e{2})$ with $\Ftil$ as in \eqref{arcchordcondition}. Let $\Delta \tilde{\Gamma}$ as
\begin{equation}\label{operdiferencia}
\Delta\tilde{\Gamma}=\tilde{\Gamma}(\eta)-\tilde{\Gamma}(\eta-r)
\end{equation}
 then the following estimate holds
\begin{align}\nonumber
&-\int_{S\e{1}}\int_{S\e{1}}f(\eta)\da{}f(\eta-r)\frac{\Delta\tilde{\Gamma}\cdot\tilde{\tau}(\eta)}{\abs{\Delta\tilde{\Gamma}}\e{2}}drd\eta\leq-\frac{1}{L}\norm{\Lambda\e{\frac{1}{2}}(f)}_{L\e{2}}\e{2} + C\norm{f}_{L\e{2}}\e{2}[\frac{1}{L}+\frac{1}{L}\norm{\Ftil}_{L\e{\infty}(S\e{1})}\norm{\tilde{\Gamma}}_{H\e{3}(S\e{1})}\e{2}\\\label{cotatermdificil}
&+\norm{\Ftil}_{L\e{\infty}(S\e{1})}\norm{\tilde{\Gamma}}_{H\e{3}(S\e{1})}+L\e{2}\norm{\Ftil}\e{2}_{L\e{\infty}(S\e{1})}\norm{\tilde{\Gamma}}_{H\e{3}(S\e{1})}+L\norm{\Ftil}\e{2}_{L\e{\infty}(S\e{1})}\norm{\tilde{\Gamma}}\e{2}_{H\e{3}(S\e{1})}]
\end{align}
where $\Lambda$ is the operator given by
\begin{equation}
\Lambda(f)=\pi PV\int_{S\e{1}}\frac{f(x)-f(x-y)}{\sin\e{2}(\pi y)}dy
\end{equation}
and $\Lambda\e{\frac{1}{2}}$ is the unique non negative operator such that $(\Lambda\e{\frac{1}{2}})\e{2}=\Lambda$.
\end{lem}
\begin{nota}
If $f(x)=\sum_{n=-\infty}\e{n=\infty}f_n e\e{2\pi nix}$ we have that the operator $\Lambda f$ in the Fourier representation is given by $(\Lambda(f))_n=\abs{n}f_n$, i.e., the operator $\Lambda(f)$ is just the half-Laplacian acting on the space of 1-periodic functions. Therefore, the operator $\Lambda\e{\frac{1}{2}}$ is well defined and it is given in the Fourier representation as $(\Lambda\e{\frac{1}{2}}(f))_n=\abs{n}\e{\frac{1}{2}}f_n$.
\end{nota}
\begin{proof} Let $f\in C_p\e{\infty}(S\e{1})$, we can write the left hand side of \eqref{cotatermdificil} as follows 
\begin{align*}
&-\frac{1}{L}\int_{S\e{1}}\int_{S\e{1}}f(\eta)\frac{\da{}f(\eta-r)}{\tan(\pi r)}dr d\eta-\int_{S\e{1}}\int_{S\e{1}}f(\eta)\da{}f(\eta-r)A(\eta,r)dr d\eta\\
&\equiv F_1+F_2
\end{align*} 
where $A(\eta,r)=\frac{\Delta\tilde{\Gamma}\cdot\tilde{\tau}(\eta)}{\abs{\Delta\tilde{\Gamma}}\e{2}}-\frac{\pi\da{}\tilde{\Gamma}(\eta)\cdot\tilde{\tau}(\eta)}{\tan(\pi r)\abs{\da{}\tilde{\Gamma}(\eta)}\e{2}}$.
Using Fourier it follows that  $\Lambda=\da{}H=H\da{}$ where $H$ is the periodic version of the Hilbert operator (c.f. \cite{harmonic}).
\begin{equation}\label{hilbertop}
H(f)=PV\int_{S\e{1}}\frac{f(x-y)}{\tan(\pi y)}dy
\end{equation}
then we can write $F_1$ as
\begin{align}\nonumber
&F_1=-\frac{1}{L}\int_{S\e{1}}f(\eta)H(\da{}f)(\eta)d\eta=-\frac{1}{L}\int_{S\e{1}}f(\eta)\Lambda(f)(\eta)d\eta=-\frac{1}{L}\int_{S\e{1}}(\Lambda\e{\frac{1}{2}}(f)(\eta))\e{2}d\eta\\\label{estimF1}
&=-\frac{1}{L}\norm{\Lambda\e{\frac{1}{2}}(f)}_{L\e{2}}\e{2}
\end{align}

In order to estimate $F_2$ we use that $\da{}f(\eta-r)=-\partial_r f(\eta-r)$. Then integrating by parts,
\begin{align*}
F_2=-\int_{-\frac{1}{2}}\e{\frac{1}{2}}\int_{-\frac{1}{2}}\e{\frac{1}{2}}f(\eta)f(\eta-r)\partial_rA(\eta,r)dr d\eta
\end{align*}
We now claim that
\begin{equation}\label{operadorA}
\abs{\dr{}A(\eta,r)}\leq C( K_1+\abs{r}\e{-\frac{1}{2}}K_2)\quad\textit{for}\quad\eta,r\in[-\frac{1}{2},\frac{1}{2}]
\end{equation}
where $$K_1=\frac{1}{L}+\frac{1}{L}\norm{\Ftil}_{L\e{\infty}(S\e{1})}\norm{\tilde{\Gamma}}_{H\e{3}(S\e{1})}\e{2}+L\norm{\Ftil}\e{2}_{L\e{\infty}(S\e{1})  }\norm{\tilde{\Gamma}}\e{2}_{H\e{3}(S\e{1})  }$$ and $$K_2=\norm{\Ftil}_{L\e{\infty}(S\e{1})  }\norm{\tilde{\Gamma}}_{H\e{3}(S\e{1})}(1+L\e{2}\norm{\Ftil}_{L\e{\infty}(S\e{1})})$$

Let us assume for the moment that \eqref{operadorA} holds. Then, using Cauchy-Schwarz and Young's convolution inequality to obtain $\norm{f*(\cdot)\e{-\frac{1}{2}}}_{L\e{2}}\leq C\norm{f}_{L\e{2}}$, we derive the estimate
\begin{align*}
F_2\leq C\norm{\tilde{f}}\e{2}_{L\e{2}}(K_1+K_2)
\end{align*}
and combining this estimate with \eqref{estimF1} we obtain the estimate \eqref{cotatermdificil}.

Therefore, in order to finish the proof of the Lemma we just need to show \eqref{operadorA}. 

We can decompose $A(\eta,r)$ as follows
\begin{align}\nonumber
&\dr{ }A(\eta,r)=-\frac{\da{}\Delta\tilde{\Gamma}\cdot\tilde{\tau}(\eta)}{\abs{\Delta\tilde{\Gamma}}\e{2}}-\da{}\tilde{\Gamma}(\eta)\cdot\tilde{\tau}(\eta)(\frac{1}{\abs{\Delta\tilde{\Gamma}}\e{2}}-\frac{1}{r\e{2}\abs{\da{}\tilde{\Gamma}(\eta)}\e{2}})+2\frac{\Delta\tilde{\Gamma}\cdot\tilde{\tau}(\eta)\Delta\tilde{\Gamma}\cdot\da{}\Delta\tilde{\Gamma}}{\abs{\Delta\tilde{\Gamma}}\e{4}}\\\nonumber
&-2\frac{\Delta\tilde{\Gamma}\cdot\tilde{\tau}(\eta)(\Delta\tilde{\Gamma}-r\da{}\tilde{\Gamma}(\eta))\cdot\da{}\tilde{\Gamma}(\eta)}{\abs{\Delta\tilde{\Gamma}}\e{4}}-2\frac{(\Delta\tilde{\Gamma}-r\da{}\tilde{\Gamma}(\eta))\cdot\tilde{\tau}(\eta)r\abs{\da{}\tilde{\Gamma}(\eta)}\e{2}}{\abs{\Delta\tilde{\Gamma}}\e{4}}\\\label{Mi}
&-\frac{2r\e{2}\da{}\tilde{\Gamma}(\eta)\cdot\tilde{\tau}(\eta)\abs{\da{}\tilde{\Gamma}(\eta)}\e{2}}{\abs{\Delta\tilde{\Gamma}}\e{2}}(\frac{1}{\abs{\Delta\tilde{\Gamma}}\e{2}}-\frac{1}{r\e{2}\abs{\da{}\tilde{\Gamma}(\eta)}\e{2}})-\frac{1}{L}(\frac{1}{r\e{2}}-\frac{\pi\e{2}}{\sin\e{2}(\pi r)})\equiv \sum_{i=1}\e{7}M_i(\eta,r)
\end{align}

Using $\da{2}\tilde{\Gamma}(\eta)\cdot\tilde{\tau}(\eta)=0$ and 
\begin{equation}\label{restaderiv}
\da{}\tilde{\Gamma}(\eta-r)-\da{}\tilde{\Gamma}(\eta)=r\int_0\e{1}\da{2}\tilde{\Gamma}(\eta-r s)ds
\end{equation}
 we have that $$M_1(\eta,r)=\frac{r}{\abs{\Delta\tilde{\Gamma}}\e{2}}\int_0\e{1}(\da{2}\tilde{\Gamma}(\eta-r s)-\da{2}\tilde{\Gamma}(\eta))\cdot\tilde{\tau}(\eta)ds$$
Let us assume $\delta=\frac{1}{2}$. Therefore using that $\Ftil\in L\e{\infty}(S\e{1})(\R\e{2})$ with $\Ftil$ as in \eqref{arcchordcondition}, we obtain
\begin{align*}
\abs{M_1(\eta,r)}\leq\abs{r}\e{\delta-1}\frac{r\e{2}}{\abs{\Delta\tilde{\Gamma}}\e{2}}\int_0\e{1}\frac{\abs{\da{2}\tilde{\Gamma}(\eta-r s)-\da{2}\tilde{\Gamma}(\eta)}}{\abs{r s}\e{\delta}}\abs{s}\e{\delta}ds\leq \abs{r}\e{\delta-1}\norm{\Ftil}_{L\e{\infty}(S\e{1})  }\norm{\tilde{\Gamma}}_{\mathcal{C}\e{2,\delta}}\leq C \abs{r}\e{-\frac{1}{2}}\norm{\Ftil}_{L\e{\infty}(S\e{1})  }\norm{\tilde{\Gamma}}_{H\e{3}(S\e{1})}
\end{align*}
where in the last inequality we have used that Morrey's inequality implies $\norm{\tilde{\Gamma}}_{C\e{2,\frac{1}{2}}}\leq C\norm{\tilde{\Gamma}}_{H\e{3}(S\e{1})}$.

Using 
\begin{equation}\label{restagamma}
\Delta\tilde{\Gamma}=r\int_0\e{1}\da{}\tilde{\Gamma}(\alpha)ds\quad\textit{for}\quad\eta,r\in[-\frac{1}{2},\frac{1}{2}]
\end{equation}
and proceding as in \eqref{restaderiv}
 we can write 
\begin{equation}\label{diferenciaderivdelta}
r\da{}\tilde{\Gamma}(\eta)-\Delta\tilde{\Gamma}=r\e{2}\int_0\e{1}\int_0\e{1}\da{2}\tilde{\Gamma}(\phi)(1-t)dtds
\end{equation} 
 and thus we rewrite $M_2(\eta,r)$, using $\da{}\tilde{\Gamma}(\eta)\cdot\tilde{\tau}(\eta)=L$, as
\begin{align}\label{M2}
&M_{2}(\eta,r)=\frac{rL\int_0\e{1}\int_0\e{1}\da{2}\tilde{\Gamma}(\phi)(1-t)dsdt\cdot\int_0\e{1}(\da{}\tilde{\Gamma}(\eta)+\da{}\tilde{\Gamma}(\alpha))ds}{\abs{\Delta\tilde{\Gamma}}\e{2}\abs{\da{}\tilde{\Gamma}(\eta)}\e{2}}
\end{align}
where we have introduced by shortness the notation  $\phi=\eta-r+sr+tr-str$ and $\alpha=\eta-r+sr$. We now write $\da{2}\tilde{\Gamma}(\phi)=(\da{2}\tilde{\Gamma}(\phi)-\da{2}\tilde{\Gamma}(\eta))+\da{2}\tilde{\Gamma}(\eta)$ and we use the fact that $\da{2}\tilde{\Gamma}(\eta)\cdot\da{}\tilde{\Gamma}(\eta)=0$ and 
$\phi-\eta=-r(1-s)(1-t)$.
\begin{align*}
&M_2(\eta,r)=\frac{r\e{\delta+1}}{\abs{\Delta\tilde{\Gamma}}\e{2}L}\int_0\e{1}\int_0\e{1}\frac{\da{2}\tilde{\Gamma}(\phi)-\da{2}\tilde{\Gamma}(\eta)}{\abs{\phi-\eta}\e{\delta}}(1-s)\e{\delta}(1-t)\e{\delta+1}dsdt\cdot\int_0\e{1}(\da{}\tilde{\Gamma}(\eta)+\da{}\tilde{\Gamma}(\alpha))ds+\\
&+\frac{\Ftil}{L}\da{2}\tilde{\Gamma}(\eta)\cdot\int_0\e{1}\int_0\e{1}\da{2}\tilde{\Gamma}(\eta-sr+str)(1-t)dsdt\equiv B_1(\eta,r)+B_2(\eta,r)
\end{align*}
where in order to obtain $B_2(\eta,r)$ we have used that $\da{2}\tilde{\Gamma}(\eta)\cdot\int_0\e{1}\da{}\tilde{\Gamma}(\alpha)ds=\da{2}\tilde{\Gamma}(\eta)\cdot\int_0\e{1}(\da{}\tilde{\Gamma}(\alpha)-\da{}\tilde{\Gamma}(\eta))ds=r\da{2}\tilde{\Gamma}(\eta)\cdot\int_0\e{1}\int_0\e{1}\da{2}\tilde{\Gamma}(\eta-sr+str)(1-t)dsdt$.

We then have the estimates
\begin{align*}
&\abs{B_1(\eta,r)}\leq C \abs{r}\e{\delta-1}\norm{\Ftil}_{L\e{\infty}(S\e{1})  }\norm{\tilde{\Gamma}}_{\mathcal{C}\e{2,\delta}}\leq C \abs{r}\e{-\frac{1}{2}}\norm{\Ftil}_{L\e{\infty}(S\e{1})  }\norm{\tilde{\Gamma}}_{H\e{3}(S\e{1})} \\
&\abs{B_2(\eta,r)}\leq \frac{1}{L}\norm{\Ftil}_{L\e{\infty}(S\e{1})}\norm{\da{2}\tilde{\Gamma}}_{L\e{\infty}(S\e{1})}\e{2}\leq C \frac{1}{L}\norm{\Ftil}_{L\e{\infty}(S\e{1})  }\norm{\tilde{\Gamma}}_{H\e{3}(S\e{1})}\e{2}
\end{align*}
for $\eta,r\in[-\frac{1}{2},\frac{1}{2}]$.
Using a similar procedure we can rewrite $M_3(\eta,r)$,
\begin{align*}
&M_3(\eta,r)=2\frac{r\e{3}}{\abs{\Delta\tilde{\Gamma}}\e{4}}(\int_0\e{1}\da{}\tilde{\Gamma}(\alpha)ds\cdot\tilde{\tau}(\eta))(\int_0\e{1}\da{}\tilde{\Gamma}(\alpha)ds)\cdot(\int_0\e{1}\da{2}\tilde{\Gamma}(\alpha)ds)=\\
&2\Ftil\e{2}\int_0\e{1}\int_0\e{1}\da{2}\tilde{\Gamma}(\phi)(1-t)dsdt\cdot\tilde{\tau}(\eta)\int_0\e{1}\da{}\tilde{\Gamma}(\alpha)ds\cdot\int_0\e{1}\da{2}\tilde{\Gamma}(\alpha)ds+2\Ftil\e{2}L\int_0\e{1}\int_0\e{1}\da{2}\tilde{\Gamma}(\phi)(1-t)dsdt\cdot\int_0\e{1}\da{2}\tilde{\Gamma}(\alpha)ds\\
&+2L\Ftil\e{2}r\e{\delta-1}\da{}\tilde{\Gamma}(\eta)\cdot\int_0\e{1}\frac{\da{2}\tilde{\Gamma}(\alpha)-\da{2}\tilde{\Gamma}(\eta)}{\abs{\alpha-\eta}\e{\delta}}(1-s\e{\delta})ds
\end{align*}
where we have repeatedly written $\da{}\tilde{\Gamma}(\alpha)=(\da{}\tilde{\Gamma}(\alpha)-\da{}\tilde{\Gamma}(\eta))+\da{}\tilde{\Gamma}(\eta)$ as well as the derivative of this formula and the fact that $\da{2}\tilde{\Gamma}\cdot\da{}\tilde{\Gamma}=0$ and $\da{}\tilde{\Gamma}\cdot\tilde{\tau}=L$.
Therefore, we obtain
\begin{displaymath}
\abs{M_3(\eta,r)}\leq CL\norm{\Ftil}\e{2}_{L\e{\infty}(S\e{1})  }\norm{\tilde{\Gamma}}\e{2}_{H\e{3}(S\e{1})  }+C\abs{r}\e{-\frac{1}{2}}L\e{2}\norm{\Ftil}\e{2}_{L\e{\infty}(S\e{1})  }\norm{\tilde{\Gamma}}_{H\e{3}(S\e{1})  }\quad\textit{for}\quad\eta,r\in[-\frac{1}{2},\frac{1}{2}]
\end{displaymath}
The terms $M_4(\eta,r)$ and $M_5(\eta,r)$ can be estimated in a completely analogous way. The term $M_6(\eta,r)$ can be written as $M_6(\eta,r)=2\Ftil L\e{2}M_2(\eta,r)$.
 Finally, an elementary computation solves that  $\abs{M_7(\eta,r)}=\abs{M_7(r)}\leq C$ for $r\in[-\frac{1}{2},\frac{1}{2}]$. Then combining all the estimates for the terms $M_i(\eta,r)$ for $i=1,\cdots,7$, we obtain \eqref{operadorA} and we finish the proof of the lemma. 
\end{proof}
\begin{nota}We note that we have obtained the following estimated for further reference 
\begin{equation}\label{cotaMi}
\abs{M_i(\eta,r)}\leq C(1+\abs{r}\e{-\frac{1}{2}})\exp(\norm{\Ftil}_{L\e{\infty}(S\e{1})}\e{2}+\norm{\tilde{\Gamma}}_{H\e{3}(S\e{1})\e{2}})\quad\textit{for}\quad i=1,\cdots,7
\end{equation}
with $\abs{r}\leq\frac{1}{2}$ and $\abs{\eta}\leq\frac{1}{2}$.
\end{nota}
We now proof an analogous estimate for the third derivative of a suitable function.
\begin{lem}\label{lema2der}Let $L>0$, $\tilde{\Gamma}\in H^{k}(S\e{1};\R\e{2})$ for $k\ge 3$, $\tilde{\Gamma}(S\e{1})\subset D$ such that $\abs{\da{}\tilde{\Gamma}}=L$ for all $\eta\in S\e{1}$ and $\mathcal{F}(\tilde{\Gamma})\in L\e{\infty}(S\e{1})(\R\e{2})$ with $\Ftil$ as in \eqref{arcchordcondition}. Suppose that we identify $S\e{1}$with the interval $[-\frac{1}{2},\frac{1}{2}]$. Let be $\bar{\Delta} \tilde{\Gamma}=\tilde{\Gamma}(\eta)-\tilde{\Gamma}(r)$ then the following estimate holds
\begin{align}\label{cotaderiv3G}
&\da{3}(\frac{\bar{\Delta}\tilde{\Gamma}\cdot \tilde{n}(r)}{\abs{\bar{\Delta}\tilde{\Gamma}}\e{2}}\bar{\Delta}\tilde{\Gamma})\cdot \tilde{n}(\eta)=-\frac{6\pi L\e{4}(\tilde{\kappa}(\eta))\e{2}}{\tan(\pi(\eta-r))}+Q(\eta,r)\quad\textit{for}\quad \eta\neq r
\end{align}
where $\norm{Q}_{L\e{2}(S\e{1})}\leq\est{3}$ with $C$ a geometrical constant.
\end{lem}
\begin{nota}
The factor $\est{3}$ is far from optimal, however in the forthcoming arguments we will estimate several terms with the form $(\norm{\Ftil}_{L\e{\infty}(S\e{1})}\e{2}+\norm{\tilde{\Gamma}}_{H\e{3}(S\e{1})}\e{2}+1)\e{\alpha}$ by means of these exponential terms in order to simplify the notation.
\end{nota}
\begin{proof}
Using Leibniz rule we obtain
\begin{align}\nonumber
&\da{3}(\frac{\bar{\Delta}\tilde{\Gamma}\cdot \tilde{n}(r)}{\abs{\bar{\Delta}\tilde{\Gamma}}\e{2}}\Delta\tilde{\Gamma})\cdot \tilde{n}(\eta)=\frac{(\da{3}\bar{\Delta}\tilde{\Gamma})\cdot\tilde{n}(r)}{\abs{\bar{\Delta}\tilde{\Gamma}}\e{2}}\bar{\Delta}\tilde{\Gamma}\cdot\tilde{n}(\eta)+\frac{\bar{\Delta}\tilde{\Gamma}\cdot\tilde{n}(r)}{\abs{\bar{\Delta}\tilde{\Gamma}}\e{2}}\da{3}\bar{\Delta}\tilde{\Gamma}\cdot\tilde{n}(\eta)+\bar{\Delta}\tilde{\Gamma}\cdot\tilde{n}(r)\bar{\Delta}\tilde{\Gamma}\cdot\tilde{n}(\eta)\da{3}(\frac{1}{\abs{\bar{\Delta}\tilde{\Gamma}}\e{2}})+l.o.t.\\\label{C1C2C3}
&\equiv C_1(\eta,r)+C_2(\eta,r)+C_3(\eta,r)+l.o.t.
\end{align}
where $l.o.t.$ are the terms containing at most two derivatives of $\bar{\Delta}\tilde{\Gamma}$.
All these terms can be estimated as $\norm{l.o.t.}_{L\e{2}(S\e{1})}\leq\est{3}$. Therefore we only study in detail the terms $C_i(\eta,r)$ for $i=1,2,3$. To this end it is convenient to derive a suitable approximation for $\frac{\bar{\Delta}\tilde{\Gamma}}{\abs{\bar{\Delta}\tilde{\Gamma}}\e{2}}$. We have

\begin{align*}\nonumber
&\frac{\bar{\Delta}\tilde{\Gamma}}{\abs{\bar{\Delta}\tilde{\Gamma}}\e{2}}-\frac{\da{}\tilde{\Gamma}(\eta)}{(\eta-r)\abs{\da{}\tilde{\Gamma}(\eta)}\e{2}}=\frac{\Delta\tilde{\Gamma}-(\eta-r)\da{}\tilde{\Gamma}(\eta)}{\abs{\Delta\tilde{\Gamma}}\e{2}}+\frac{\eta-r}{L}\da{}\tilde{\Gamma}(\eta)M_2(\eta,\eta-r)\equiv \textbf{A}_1(\eta,r)
\end{align*}
with $M_2(\eta,\eta-r)$ as in \eqref{Mi} where we have used that $\da{}\tilde{\Gamma}\cdot\tilde{\tau}=L$.
Using \eqref{diferenciaderivdelta} and \eqref{M2} we then obtain 
\begin{align}\nonumber
&\textbf{A}_1(\eta,r)=\Ftil\int_0\e{1}\int_0\e{1}\da{2}\tilde{\Gamma}(\eta+t(\beta-\eta))(s-1)dsdt\\\label{D1D2}
&-\Ftil\frac{\da{}\tilde{\Gamma}(\eta)}{\abs{\da{}\tilde{\Gamma}(\eta)}\e{2}}\int_0\e{1}\int_0\e{1}\da{2}\tilde{\Gamma}(\eta+t(\beta-\eta))(s-1)dsdt\cdot\int_0\e{1}(\da{}\tilde{\Gamma}(\beta)+\da{}\tilde{\Gamma}(\eta))ds\equiv \textbf{D}_1(\eta,r)+\textbf{D}_2(\eta,r)
\end{align}
where $\beta=r+s(\eta-r)$. Notice that 
\begin{equation}\label{cotaA1}
\norm{\textbf{A}_1}_{L\e{\infty}(S\e{1})}\leq C\norm{\Ftil}_{L\e{\infty}}\norm{\da{2}\tilde{\Gamma}}_{L\e{\infty}(S\e{1})}\leq\est{3}
\end{equation} Moreover, using that $\da{}\tilde{\Gamma}(r)\cdot\tilde{n}(r)=0$
\begin{align}\label{A2}
\frac{\bar{\Delta}\tilde{\Gamma}\cdot \tilde{n}(r)}{\abs{\bar{\Delta}\tilde{\Gamma}}\e{2}}=\frac{r-\eta}{\abs{\bar{\Delta}\tilde{\Gamma}}\e{2}}\int_0\e{1}\da{}\tilde{\Gamma}(\beta)\cdot \tilde{n}(r)ds=\Ftil\int_0\e{1}\int_0\e{1}(s-1)\da{2}\tilde{\Gamma}(r+t(\beta-r))\cdot \tilde{n}(r)dtds\equiv A_2(\eta,r)
\end{align}
where $\beta=r+s(\eta-r)$
Moreover, we can estimate
\begin{equation}\label{cotaA2}
\norm{A_2}_{L\e{\infty}(S\e{1})}\leq C\norm{\Ftil}_{L\e{\infty}(S\e{1})}\norm{\da{2}\tilde{\Gamma}}_{L\e{\infty}(S\e{1})}\leq\est{3}
\end{equation}
Using that $\textbf{D}_2(\eta,r)\cdot\tilde{n}(\eta)=0$ and $\abs{\tilde{n}}=1$ we obtain that the term $C_1(\eta,r)$ in \eqref{C1C2C3} can be estimated as 
\begin{displaymath}
\norm{C_1}_{L\e{2}(S\e{1})}\leq \norm{\da{3}\bar{\Delta}\tilde{\Gamma}}_{L\e{2}(S\e{1})}\norm{\Ftil}_{L\e{\infty}}\norm{\da{2}\tilde{\Gamma}}_{L\e{\infty}}\leq C\norm{\Ftil}_{L\e{\infty}}\norm{\tilde{\Gamma}}_{H\e{3}(S\e{1})}\leq\est{3}
\end{displaymath}
where we have used again Morrey's inequality.

Writing $C_2(\eta,r)=A_2(\eta,r)\da{3}\bar{\Delta}\tilde{\Gamma}\cdot\tilde{n}(\eta)$ and using \eqref{A2} we obtain 
\begin{displaymath}
\norm{C_2}_{L\e{2}}\leq C\norm{\Ftil}_{L\e{\infty}}\norm{\da{2}\tilde{\Gamma}}_{L\e{\infty}}\norm{\da{3}\bar{\Delta}\tilde{\Gamma}}_{L\e{2}(S\e{1})}\leq C\norm{\Ftil}_{L\e{\infty}(S\e{1})}\norm{\tilde{\Gamma}}_{H\e{3}(S\e{1})}\e{2}\leq\est{3}
\end{displaymath}
Computing the third derivative of $\frac{1}{\abs{\bar{\Delta}\tilde{\Gamma}}\e{2}}$ and using $\da{2}\tilde{\Gamma}\cdot\da{}\tilde{\Gamma}=0$ we have that
\begin{align*}
&C_3(\eta,r)=\frac{24L\e{2}\bar{\Delta}\tilde{\Gamma}\cdot\tilde{n}(r)\bar{\Delta}\tilde{\Gamma}\cdot\tilde{n}(\eta)\bar{\Delta}\tilde{\Gamma}\cdot\da{}\tilde{\Gamma}(\eta)}{\abs{\bar{\Delta}\tilde{\Gamma}}\e{6}}-\frac{2\bar{\Delta}\tilde{\Gamma}\cdot\tilde{n}(r)\bar{\Delta}\tilde{\Gamma}\cdot\tilde{n}(\eta)\bar{\Delta}\tilde{\Gamma}\cdot\da{3}\tilde{\Gamma}(\eta)}{\abs{\bar{\Delta}\tilde{\Gamma}}\e{4}}\\
&+\frac{24\bar{\Delta}\tilde{\Gamma}\cdot\tilde{n}(r)\bar{\Delta}\tilde{\Gamma}\cdot\tilde{n}(\eta)\bar{\Delta}\tilde{\Gamma}\cdot\da{2}\tilde{\Gamma}(\eta)\bar{\Delta}\tilde{\Gamma}\cdot\da{}\tilde{\Gamma}(\eta)}{\abs{\bar{\Delta}\tilde{\Gamma}}\e{6}}-\frac{48\bar{\Delta}\tilde{\Gamma}\cdot\tilde{n}(r)\bar{\Delta}\tilde{\Gamma}\cdot\tilde{n}(\eta)(\bar{\Delta}\tilde{\Gamma}\cdot\da{}\tilde{\Gamma}(\eta))\e{3}}{\abs{\bar{\Delta}\tilde{\Gamma}}\e{8}}\equiv \sum_{i=1}\e{4}C_3\e{i}(\eta,r)
\end{align*}

 The terms $C_3\e{i}(\eta,r)$ for $i=2,3$ can be readily estimated rewriting them in terms of $\textbf{A}_1(\eta,r)$ and $A_2(\eta,r)$ as in \eqref{D1D2} and \eqref{A2}. Then $\norm{C_3\e{i}}_{L\e{2}(S\e{1})}\leq\est{3}$ for $i=2,3$. The terms $C_3\e{1}(\eta,r)$ and $C_3\e{4}(\eta,r)$ give a more singular contribution and need to be studied more in detail. We describe the requited computations for $C_3\e{1}(\eta,r)$ since the term $C_3\e{4}(\eta,r)$ can be treated in the similar manner.

Using $\textbf{A}_1(\eta,r)$ and $A_2(\eta,r)$ we can rewrite $C_3\e{1}(\eta,r)$ as follows
\begin{align*}
&C_3\e{1}(\eta,r)=-8L\e{2}A_2(\eta,r)\textbf{A}_1(\eta,r)\cdot\tilde{n}(\eta) \textbf{A}_1(\eta,r)\cdot\da{}\tilde{\Gamma}(\eta)-\\
&-8L\e{2}\frac{\Ftil\e{2}}{\eta-r}\int_0\e{1}\int_0\e{1}(s-1)\da{2}\tilde{\Gamma}(r+t(\beta-r))\cdot \tilde{n}(r)dtds\int_0\e{1}\int_0\e{1}\da{2}\tilde{\Gamma}(\eta+t(\beta-\eta))\cdot\tilde{n}(\eta)(s-1)dsdt\equiv C_3\e{11}(\eta,r)+C_3\e{12}(\eta,r)
\end{align*}
Therefore, we estimate $\norm{C_3\e{11}(\eta,r)}\leq\est{3}$ but we need to study carefully $C_3\e{12}(\eta,r)$. 
Using that $$\da{2}\tilde{\Gamma}(\eta+t(\beta-\eta))(s-1)\cdot\tilde{n}(\eta)=(s-1)\da{2}\tilde{\Gamma}(\eta)\cdot\tilde{n}(\eta)+(\eta-r)\e{\delta}\frac{\da{2}\tilde{\Gamma}(\eta+t(\beta-\eta))-\da{2}\tilde{\Gamma}(\eta)}{(t(\beta-\eta))\e{\delta}}(s-1)\e{\delta+1}t\e{\delta}\cdot\tilde{n}(\eta)$$
we can rewrite $C_3\e{12}(\eta,r)$ as
\begin{align}\nonumber
&C_3\e{12}(\eta,r)=12L\e{2}\frac{\Ftil}{\eta-r} A_2(\eta,r)\da{2}\tilde{\Gamma}(\eta)\cdot\tilde{n}(\eta)\\\label{C312}
&+24L\e{2}\Ftil(\eta-r)\e{\delta-1}A_2(\eta,r)\int_{0}\e{1}\int_{0}\e{1}\frac{\da{2}\tilde{\Gamma}(\eta+t(\beta-\eta))-\da{2}\tilde{\Gamma}(\eta)}{(t(\beta-\eta))\e{\delta}}(s-1)\e{\delta+1}t\e{\delta}\cdot\tilde{n}(\eta)dsdt
\end{align} 
The $L\e{2}(S\e{1})$ norm of the last term in the right hand side of \eqref{C312} is bounded by $\est{3}$ and using that
\begin{align}\label{A2desar}
&\da{2}\tilde{\Gamma}(r+t(\beta-r))(s-1)\cdot\tilde{n}(r)=(s-1)\da{2}\tilde{\Gamma}(\eta)\cdot\tilde{n}(\eta)+(\eta-r)\e{\delta}\frac{\da{2}\tilde{\Gamma}(\eta+t(\beta-\eta))-\da{2}\tilde{\Gamma}(\eta)}{((\eta-r)(st-1))\e{\delta}}(s-1)(ts-1)\e{\delta}\cdot\tilde{n}(r)\\
&+(r-\eta)\da{2}\tilde{\Gamma}(\eta)\cdot\int_0\e{1}\tilde{\kappa}(\beta)\tilde{\tau}(\beta)ds
\end{align}
we can write the second term in the right hand side of \eqref{C312} as
\begin{displaymath}
12L\e{2}\frac{\Ftil}{\eta-r} A_2(\eta,r)\da{2}\tilde{\Gamma}(\eta)\cdot\tilde{n}(\eta)=6L\e{2}\frac{\Ftil\e{2}}{\eta-r}( \da{2}\tilde{\Gamma}(\eta)\cdot\tilde{n}(\eta))\e{2}+K(\eta,r)
\end{displaymath}
where $\norm{K}_{L\e{2}(S\e{1})}\leq\est{3}$.

Using \eqref{reformGFBP5} we have
\begin{align*}
&6L\e{2}\frac{\Ftil\e{2}}{\eta-r}( \da{2}\tilde{\Gamma}(\eta)\cdot\tilde{n}(\eta))\e{2}=6L\e{6}\frac{\eta-r}{\abs{\bar{\Delta}\tilde{\Gamma}}\e{2}}(\tilde{\kappa}(\eta))\e{2}=6L\e{4}\frac{(\tilde{\kappa}(\eta))\e{2}}{\eta-r}+6L\e{6}(\eta-r)(\frac{1}{\abs{\bar{\Delta}\tilde{\Gamma}}\e{2}}-\frac{1}{(\eta-r)\e{2}\abs{\da{}\tilde{\Gamma}(\eta)}\e{2}})( \tilde{\kappa}(\eta))\e{2}\\
&=6\pi L\e{4}\frac{(\tilde{\kappa}(\eta))\e{2}}{\tan(\pi(\eta-r))}+6L\e{4}(\tilde{\kappa}(\eta))\e{2}(\frac{1}{\eta-r}-\frac{\pi}{\tan(\pi(\eta-r))})\\
&-6L\e{6}\Ftil(\tilde{\kappa}(\eta))\e{2}\int_0\e{1}\int_0\e{1}\da{2}\Gamma(\eta+t(\beta-\eta))(s-1)dsdt\cdot\int_0\e{1}\da{}\tilde{\Gamma}(\beta)+\da{}\tilde{\Gamma}(\eta)ds
\end{align*}
Since $L\e{4}\norm{(\tilde{\kappa}(\eta))\e{2}(\frac{1}{\eta-r}-\frac{\pi}{\tan(\pi(\eta-r))})}_{L\e{2}(S\e{1})}\leq CL\e{4}\norm{\tilde{\kappa}}_{L\e{\infty}}\e{2}\norm{\frac{1}{\eta-r}-\frac{\pi}{\tan(\pi(\eta-r))}}_{L\e{2}}\leq C\norm{\tilde{\Gamma}}_{H\e{3}(S\e{1})}\e{2}$, we deduce \eqref{cotaderiv3G} and the result follows.
\end{proof}
We can obtain also estimates for the first derivative of $\frac{\bar{\Delta}\tilde{\Gamma}\cdot \tilde{n}(r)}{\abs{\bar{\Delta}\tilde{\Gamma}}\e{2}}\bar{\Delta}\tilde{\Gamma}$,
\begin{lem}\label{lemderivG}Let $L_0>0$, $\tilde{\Gamma}\in H^{k}(S\e{1};\R\e{2})$ for $k\ge 3$, $\tilde{\Gamma}(S\e{1})\subset D$ such that $\abs{\da{}\tilde{\Gamma}}=L$ for all $\eta\in S\e{1}$ and $\mathcal{F}(\tilde{\Gamma})\in L\e{\infty}(S\e{1})(\R\e{2})$ with $\Ftil$ as in \eqref{arcchordcondition}. Let  $\bar{\Delta}\tilde{\Gamma}=\tilde{\Gamma}(\eta)-\tilde{\Gamma}(r)$ then the following estimate holds
\begin{equation}\label{derivG}
\abs{\da{}(\frac{\bar{\Delta}\tilde{\Gamma}\cdot \tilde{n}(r)}{\abs{\bar{\Delta}\tilde{\Gamma}}\e{2}}\bar{\Delta}\tilde{\Gamma} )}\leq CL\norm{\Ftil}_{L\e{\infty}(S\e{1})}\norm{\tilde{\Gamma}}_{H\e{3}(S\e{1})}+CL\e{3}\norm{\Ftil}_{L\e{\infty}(S\e{1})}\e{2}\norm{\tilde{\Gamma}}_{H\e{3}(S\e{1})}
\end{equation}
\end{lem}
\begin{proof}
Differentiating $\frac{\bar{\Delta}\tilde{\Gamma}\cdot \tilde{n}(r)}{\abs{\bar{\Delta}\tilde{\Gamma}}\e{2}}\bar{\Delta}\tilde{\Gamma}$ we arrive at 
\begin{align*}
&\da{}(\frac{\bar{\Delta}\tilde{\Gamma}\cdot \tilde{n}(r)}{\abs{\bar{\Delta}\tilde{\Gamma}}\e{2}}\bar{\Delta}\tilde{\Gamma})=\frac{\da{}\tilde{\Gamma}(\eta)\cdot \tilde{n}(r)}{\abs{\bar{\Delta}\tilde{\Gamma}}\e{2}}\bar{\Delta}\tilde{\Gamma}-\frac{2\bar{\Delta}\tilde{\Gamma}\cdot\tilde{n}(r)\bar{\Delta}\tilde{\Gamma}\cdot\da{}\tilde{\Gamma}(\eta)\bar{\Delta}\tilde{\Gamma}}{\abs{\bar{\Delta}\tilde{\Gamma}}\e{4}}+\frac{\bar{\Delta}\tilde{\Gamma}\cdot \tilde{n}(r)}{\abs{\bar{\Delta}\tilde{\Gamma}}\e{2}}\da{}\tilde{\Gamma}(\eta) \equiv G_1+G_2+G_3
\end{align*}
We rewrite $G_1$, using $\da{}\tilde{\Gamma}(\eta)\cdot\tilde{n}(r)=(\da{}\tilde{\Gamma}(\eta)-\da{}\tilde{\Gamma}(r))\cdot\tilde{n}(r)=(\eta-r)\int_0\e{1}\da{2}\tilde{\Gamma}(\beta)ds$, $\bar{\Delta}\tilde{\Gamma}=(\eta-r)\int_0\e{1}\da{}\tilde{\Gamma}(\beta)ds$ where $\beta=r+s(\eta-r)$ and \eqref{arcchordcondition}
 \begin{align*}
&G_1=\Ftil\int_0\e{1}\da{2}\tilde{\Gamma}(\beta)ds\int_0\e{1}\da{}\tilde{\Gamma}(\beta)ds\leq CL\norm{\Ftil}_{L\e{\infty}(S\e{1})}\norm{\tilde{\Gamma}}_{H\e{3}(S\e{1})}
\end{align*}
On the other hand, 
\begin{align*}
&-\frac{G_2}{2}=\frac{(r-\eta)\e{3}}{\abs{\bar{\Delta}\tilde{\Gamma}}\e{4}}\int_0\e{1}\da{}\tilde{\Gamma}(\beta)\cdot\tilde{n}(r)ds\da{}\tilde{\Gamma}(\eta)\cdot\int_0\e{1}\da{}\tilde{\Gamma}(\beta)ds\int_0\e{1}\da{}\tilde{\Gamma}(\beta)ds\\
&=\Ftil\e{2}\int_0\e{1}\int_0\e{1}\da{2}\tilde{\Gamma}(r+t(\beta-r))dt\cdot\tilde{n}(r)ds\da{}\tilde{\Gamma}(\eta)\cdot\int_0\e{1}\da{}\tilde{\Gamma}(\beta)ds\int_0\e{1}\da{}\tilde{\Gamma}(\beta)ds
\end{align*}
where $\beta=r+s(\eta-r)$, then
\begin{align*}
G_2\leq CL\e{3}\norm{\Ftil}\e{2}_{L\e{\infty}(S\e{1})}\norm{\tilde{\Gamma}}_{H\e{3}(S\e{1})}
\end{align*}
Finally, using \eqref{A2} we can write
\begin{align*}
G_3\leq L\norm{A_2}_{L\e{\infty}(S\e{1})}\leq CL\norm{\Ftil}_{L\e{\infty}(S\e{1})}\norm{\tilde{\Gamma}}_{H\e{3}(S\e{1})}
\end{align*}
\end{proof}
\begin{lem}
Let $L>0$, $\tilde{\Gamma}\in H^{k}(S\e{1};\R\e{2})$ for $k\ge 3$, $\tilde{\Gamma}(S\e{1})\subset D$ such that $\abs{\da{}\tilde{\Gamma}}=L$ for all $\eta\in S\e{1}$ and $\mathcal{F}(\tilde{\Gamma})\in L\e{\infty}(\R\e{2})$ with $\Ftil$ as in \eqref{arcchordcondition}. Let  $\bar{\Delta}\tilde{\Gamma}=\tilde{\Gamma}(\eta)-\tilde{\Gamma}(r)$. Then the following identity holds
\begin{equation}\label{deriv2u0t}
\da{2}\bar{u}_0(\tilde{\Gamma})(\eta)\cdot\tilde{\tau}(\eta)=\frac{L}{2\pi\mu}(\frac{1}{2}-L+L\e{3})\tilde{\kappa}(\eta)H(\tilde{\kappa})(\eta)+K(\eta,r)
\end{equation}
with $\norm{K}_{L\e{\infty}(S\e{1})}\leq\est{3}$ and $H$ is the periodic Hilbert transform \eqref{hilbertop}.
Moreover, we have that 
\begin{equation}\label{cotaderiv2u0L2}
\norm{\da{2}u_0\cdot\tilde{\tau}}_{L\e{2}(S\e{1})}\leq C\norm{K}_{L\e{\infty}}(1+\norm{H(\tilde{\kappa})}_{L\e{2}})\leq\est{3}
\end{equation}
\end{lem}
\begin{proof}
We first compute the derivatives formally assuming that the differentiations and integrations commute, even if integrals that do not converge absolutely arise. The correctness of this argument will be justified at the end of the proof.

 Differentiating twice in \eqref{u0gamma} we obtain
\begin{align*}
&\da{2}\bar{u}_0(\tilde{\Gamma}(\eta))\cdot\tilde{\tau}(\eta)=-\frac{L(t)}{2\pi\mu}\int_{S\e{1}}\tilde{\kappa}(r)\da{2}(\log\abs{\tilde{\Gamma}(\eta)-\tilde{\Gamma}(r)})\tilde{n}(r)\cdot\tilde{\tau}(\eta)dr+\frac{L(t)}{2\pi\mu}\int_{S\e{1}}\tilde{\kappa}(r)\da{2}(\frac{\bar{\Delta}\tilde{\Gamma}\cdot \tilde{n}(r)}{\abs{\Delta\tilde{\Gamma}}\e{2}}\Delta\tilde{\Gamma})\cdot\tilde{\tau}(\eta) dr\\
&\equiv \da{2}F\cdot\tilde{\tau}(\eta)+\da{2}G\cdot\tilde{\tau}(\eta)
\end{align*}
Differentiating $\log\abs{\tilde{\Gamma}(\eta)-\tilde{\Gamma}(r)}$ and using $M_2(\eta,\eta-r)$ as in \eqref{M2} and $\textbf{A}_1(\eta,r)$ as in \eqref{D1D2}, we can write $\da{2}F\cdot\tilde{\tau}(\eta)$ as follows:
\begin{align*}
&\da{2}(\log\abs{\tilde{\Gamma}(\eta)-\tilde{\Gamma}(r)})\tilde{n}(r)\cdot\tilde{\tau}(\eta)=\frac{L\e{2}}{\abs{\bar{\Delta}\tilde{\Gamma}}\e{2}}\tilde{n}(r)\cdot\tilde{\tau}(\eta)+\frac{\bar{\Delta}\tilde{\Gamma}}{\abs{\bar{\Delta}\tilde{\Gamma}}\e{2}}\cdot\da{2}\tilde{\Gamma}(\eta)\tilde{n}(r)\cdot\tilde{\tau}(\eta)=\frac{\tilde{\kappa}(\eta)}{2(\eta-r)}+LM_2(\eta,\eta-r)\tilde{n}(r)\cdot\tilde{\tau}(\eta)\\
&+\int_0\e{1}\int_0\e{1}(s-1)\da{}\tilde{\kappa}(\eta+t(\beta-\eta))dsdt+\int_0\e{1}\tilde{\kappa}(\beta)ds\int_0\e{1}\tilde{\kappa}(\beta)\tilde{n}(\beta)\cdot\tilde{\tau}(\eta)+\textbf{D}_1(\eta,r)\cdot\da{2}\tilde{\Gamma}(\eta)\tilde{n}(r)\cdot\tilde{\tau}(\eta)
\end{align*}
with $\beta=r+s(\eta-r)$ and where we have used that \eqref{derivnormal} yields $\tilde{n}(r)\cdot\tilde{\tau}(\eta)=-(\eta-r)\int_0\e{1}\tilde{\kappa}(\beta)\tilde{\tau}(\beta)ds$ . Then, we can write
\begin{displaymath}
\da{2}F\cdot\tilde{\tau}(\eta)=\frac{L}{4\pi\mu}\tilde{\kappa}(\eta)H(\tilde{\kappa})(\eta)+\frac{L}{2\pi\mu}\tilde{\kappa}(\eta)\int_{S\e{1}}\tilde{\kappa}(r)(\frac{1}{\eta-r}-\frac{\pi}{\tan(\pi(\eta-r))})dr+F_1(\eta,r)
\end{displaymath}
with $H$ as in \eqref{hilbertop}, $\norm{F_1}_{L\e{\infty}(S\e{1})}\leq\est{3}$. Using that $\abs{(\frac{1}{\eta-r}-\frac{\pi}{\tan(\pi(\eta-r))})}\leq C$ for $\abs{\eta}\leq\frac{1}{2}$ and $\abs{r}\leq\frac{1}{2}$ we arrive at
\begin{displaymath}
\abs{\frac{L}{2\pi\mu}\tilde{\kappa}(\eta)\int_{S\e{1}}\tilde{\kappa}(r)(\frac{1}{\eta-r}-\frac{\pi}{\tan(\pi(\eta-r))})dr}\leq\est{3}
\end{displaymath}

In order to deal with $\da{2}G\cdot\tilde{\tau}(\eta)$ we compute $\da{2}(\frac{\bar{\Delta}\tilde{\Gamma}\cdot \tilde{n}(r)}{\abs{\Delta\tilde{\Gamma}}\e{2}}\Delta\tilde{\Gamma})\cdot\tilde{\tau}(\eta)$. Using that $\da{2}\tilde{\Gamma}\cdot\tilde{\tau}=0$ and $\da{}\tilde{\Gamma}\cdot\tilde{\tau}=L$ we have
\begin{align*}
&\da{2}(\frac{\bar{\Delta}\tilde{\Gamma}\cdot \tilde{n}(r)}{\abs{\Delta\tilde{\Gamma}}\e{2}}\Delta\tilde{\Gamma})\cdot\tilde{\tau}(\eta)=\da{2}\tilde{\Gamma}(\eta)\cdot\tilde{n}(r)\frac{\bar{\Delta}\tilde{\Gamma}\cdot\tilde{\tau}(\eta)}{\abs{\bar{\Delta}\tilde{\Gamma}}\e{2}}+\bar{\Delta}\tilde{\Gamma}\cdot\tilde{n}(r)\bar{\Delta}\tilde{\Gamma}\cdot\tilde{\tau}(\eta)\da{2}(\frac{1}{\abs{\bar{\Delta}\tilde{\Gamma}}\e{2}})+\frac{2L\da{}\tilde{\Gamma}(\eta)\cdot\tilde{n}(r)}{\abs{\bar{\Delta}\tilde{\Gamma}}\e{2}}\\
&+2L\bar{\Delta}\tilde{\Gamma}\cdot\tilde{n}(r)\da{}(\frac{1}{\abs{\bar{\Delta}\tilde{\Gamma}}\e{2}})+2\da{}\tilde{\Gamma}(\eta)\cdot\tilde{n}(r)\bar{\Delta}\tilde{\Gamma}\cdot\tilde{\tau}(\eta)\da{}(\frac{1}{\abs{\bar{\Delta}\tilde{\Gamma}}\e{2}})\equiv \sum_{i=1}\e{5}H_i(\eta,r)
\end{align*}
and we can split $\da{2}G\cdot\tilde{\tau}(\eta)$ as the sum of the terms $$(\da{2}G\cdot\tilde{\tau}(\eta))_i=\frac{L}{2\pi\mu}\int_{S\e{1}}\tilde{\kappa}(r)H_i(\eta,r)dr$$ for $i=1,\cdots,5$.

Using $\textbf{A}_1(\eta,r)$ defined in $\eqref{D1D2}$ we can rewrite $H_1$ as
\begin{align*}
H_1(\eta,r)=\da{2}\tilde{\Gamma}(\eta)\cdot\tilde{n}(r)\textbf{A}_1(\eta,r)\cdot\tilde{\tau}(\eta)+\frac{\da{2}\tilde{\Gamma}(\eta)\cdot(\tilde{n}(r)-\tilde{n}(\eta))}{(\eta-r)L}+\frac{L\tilde{\kappa}(\eta)}{\eta-r}
\end{align*}
where we use \eqref{reformGFBP5}. Therefore, 
\begin{align*}
(\da{2}G\cdot\tilde{\tau}(\eta))_1=\frac{L\e{2}}{2\mu}\tilde{\kappa}(\eta)H(\tilde{\kappa})(\eta)+\frac{L\e{2}}{2\pi\mu}\tilde{\kappa}(\eta)\int_{S\e{1}}\tilde{\kappa}(r)(\frac{1}{\eta-r}-\frac{\pi}{\tan(\pi(\eta-r))})dr+K_1(\eta,r)
\end{align*}
where $\norm{K_1}_{L\e{\infty}}\leq\est{3}$ and we have used \eqref{cotaA1}.

Computing $\da{2}(\frac{1}{\abs{\bar{\Delta}\tilde{\Gamma}}\e{2}})$ and using $\textbf{A}_1(\eta,r)$ and $A_2(\eta,r)$  are as in \eqref{D1D2} and \eqref{A2}, we can rewrite $H_2(\eta,r)+H_4(\eta,r)$ as 
\begin{align*}
&H_2(\eta,r)+H_4(\eta,r)=-6L\e{2}\frac{\bar{\Delta}\tilde{\Gamma}\cdot\tilde{n}(r)}{\abs{\bar{\Delta}\tilde{\Gamma}}\e{2}}\frac{\bar{\Delta}\tilde{\Gamma}\cdot\tilde{\tau}(\eta)}{\abs{\bar{\Delta}\tilde{\Gamma}}\e{2}}+\frac{8\bar{\Delta}\tilde{\Gamma}\cdot\tilde{n}(r)}{\abs{\bar{\Delta}\tilde{\Gamma}}\e{2}}\frac{\bar{\Delta}\tilde{\Gamma}\cdot\tilde{\tau}(\eta)}{\abs{\bar{\Delta}\tilde{\Gamma}}\e{2}}\frac{\bar{\Delta}\tilde{\Gamma}\cdot\da{}\tilde{\Gamma}(\eta)}{\abs{\bar{\Delta}\tilde{\Gamma}}\e{2}}\bar{\Delta}\tilde{\Gamma}\cdot\da{}\tilde{\Gamma}(\eta)\\
&-\frac{2\bar{\Delta}\tilde{\Gamma}\cdot\tilde{n}(r)}{\abs{\bar{\Delta}\tilde{\Gamma}}\e{2}}\frac{\bar{\Delta}\tilde{\Gamma}\cdot\tilde{\tau}(\eta)}{\abs{\bar{\Delta}\tilde{\Gamma}}\e{2}}\bar{\Delta}\tilde{\Gamma}\cdot\da{2}\tilde{\Gamma}=-6L\frac{A_2(\eta,r)}{\eta-r}-2L(L+1)A_2(\eta,r)\textbf{A}_1(\eta,r)\cdot\tilde{\tau}(\eta)\\
&+8L\frac{A_2(\eta,r)}{\eta-r}+\frac{8}{L}A_2(\eta,r)\int_0\e{1}\int_0\e{1}(s-1)\da{2}\tilde{\Gamma}(\eta+t(\beta-\eta))dsdt\\
&+\frac{8}{L}A_2(\eta,r)\textbf{A}_1(\eta,r)\cdot\da{}\tilde{\Gamma}(\eta)\frac{\bar{\Delta}\tilde{\Gamma}\cdot\da{}\tilde{\Gamma}(\eta)}{\eta-r}+8A_2(\eta,r)\textbf{A}_1(\eta,r)\cdot\tilde{\tau}(\eta)\frac{\bar{\Delta}\tilde{\Gamma}\cdot\da{}\tilde{\Gamma}(\eta)}{\abs{\bar{\Delta}\tilde{\Gamma}}\e{2}}\bar{\Delta}\tilde{\Gamma}\cdot\da{}\tilde{\Gamma}(\eta)\\
&-2A_2(\eta,r)\textbf{A}_1(\eta,r)\cdot\tilde{\tau}(\eta)\bar{\Delta}\tilde{\Gamma}\cdot\da{2}\tilde{\Gamma}(\eta)-2A_2(\eta,r)\int_0\e{1}\da{}\tilde{\Gamma}(\beta)\cdot\da{2}\tilde{\Gamma}(\eta)ds\\
&=2L\frac{A_2(\eta,r)}{\eta-r}+H_{2,1}(\eta,r)
\end{align*}
where $\norm{H_{2,1}}_{L\e{\infty}(S\e{1})}\leq\est{3}$.
Using \eqref{reformGFBP5}, \eqref{A2} and \eqref{A2desar} we can write
\begin{align}\label{terminosingular}
2L\frac{A_2(\eta,r)}{\eta-r}=L\e{3}\frac{\tilde{\kappa}(\eta)}{\eta-r}+H_{2,2}(\eta,r)
\end{align}
where $\norm{H_{2,2}}_{L\e{\infty}}\leq\est{3}$.
Then, we obtain that
\begin{align*}
(\da{2}G\cdot\tilde{\tau}(\eta))_{2}+(\da{2}G\cdot\tilde{\tau}(\eta))_{4}=\frac{L\e{4}}{2\mu}\tilde{\kappa}(\eta)H(\tilde{\kappa})(\eta)+K_2(\eta,r)
\end{align*}
where $$K_2(\eta,r)=\frac{L\e{4}\tilde{\kappa}}{2\pi\mu}\int_{S\e{1}}\tilde{\kappa}(r)(\frac{1}{\eta-r}-\frac{\pi}{\tan(\pi(\eta-r))})dr+\frac{L}{2\pi\mu}\int_{S\e{1}}\tilde{\kappa}(r)H_{22}(\eta,r)dr$$ and therefore, $\norm{K_2}_{L\e{\infty}(S\e{1})}\leq\est{3}$.

Using $M_2(\eta,\eta-r)$ as in \eqref{M2} and $\da{}\tilde{\Gamma}\cdot\tilde{n}=0$ we can estimate
\begin{align*}
H_3(\eta,r)=2M_2(\eta,\eta-r)\da{}\tilde{\Gamma}(\eta)\cdot\tilde{n}(r)+\frac{2}{L}\da{}\tilde{\Gamma}(\eta)\cdot\int_0\e{1}\tilde{\kappa}(\beta)\cdot\tilde{\tau}(\beta)ds\leq\est{3}
\end{align*}
Finally, we can approximate $H_5(\eta,r)$ in the same way as we proceeded with $H_2(\eta,r)+H_4(\eta,r)$. Then, we can deduce that 
\begin{align*}
(\da{2}G\cdot\tilde{\tau}(\eta))_5=-\frac{L\e{2}}{\mu}\tilde{\kappa}(\eta)H(\tilde{\kappa})(\eta)+K_5(\eta,r)
\end{align*}
where $\norm{K_5}_{L\e{\infty}(S\e{1})}\leq\est{3}$. Combining all the formulas obtained for all the terms $(\da{2}G\cdot\tilde{\tau}(\eta))_i$ for $i=1,\cdots,5$ and $\da{2}F\cdot\tilde{\tau}(\eta)$ we arrive at \eqref{deriv2u0t}. Finally, \eqref{cotaderiv2u0L2} follows from the classical estimate $\norm{H(\tilde{\Gamma})}_{L\e{2}(S\e{1})}\leq C\norm{\tilde{\Gamma}}_{L\e{2}(S\e{1})}$.

In order to justify rigorously the computations in the lemma it is enough to replace in \eqref{u0gamma} the term $\log\abs{\tilde{\Gamma}(\eta)-\tilde{\Gamma}(r)}$ by $\frac{1}{2}\log(\abs{\tilde{\Gamma}(\eta)-\tilde{\Gamma}(r)}\e{2}+\varepsilon\e{2})$. Then all the differentiations can be carried on as above. Therefore, the terms $\frac{1}{\eta-r}$ yielding Hilbert transforms appear replaced by $\frac{\eta-r}{(\eta-r)\e{2}+\varepsilon\e{2}}$. Thus, taking the limit $\varepsilon\to 0$ we obtain the principal values appearing in the Hilbert transforms. Since the argument is standard will we will not provide more details here.
\end{proof}
\begin{lem}\label{lemderivu0}Let $L>0$, $\tilde{\Gamma}\in H^{k}(S\e{1};\R\e{2})$ for $k\ge 3$, $\tilde{\Gamma}(S\e{1})\subset D$ such that $\abs{\da{}\tilde{\Gamma}}=L$ for all $\eta\in S\e{1}$ and $\mathcal{F}(\tilde{\Gamma})\in L\e{\infty}(\R\e{2})$ with $\Ftil$ as in \eqref{arcchordcondition}. Let  $\bar{\Delta}\tilde{\Gamma}=\tilde{\Gamma}(\eta)-\tilde{\Gamma}(r)$. Then the following estimate holds
\begin{equation}\label{cotaderivu0}
\norm{\da{}\bar{u}_0(\tilde{\Gamma}(\cdot))}_{L\e{2}(S\e{1})}\leq\est{3}
\end{equation}
\end{lem}
\begin{proof}
In order to prove this lemma we proceed as in the previous one, assuming that the differentiations and the integrations commute. The rigorous proof requires to work with a regularization of the function $\log(\abs{\tilde{\Gamma}(\eta)-\tilde{\Gamma}(r)})$ as in Lemma \ref{lema2der}. Therefore, if we differentiate once in \eqref{u0gamma} we obtain
\begin{align*}
&\da{}\bar{u}_0(\tilde{\Gamma}(\eta))=-\frac{L}{2\pi\mu}\int_{S\e{1}}\tilde{\kappa}(r)\da{}(\log\abs{\tilde{\Gamma}(\eta)-\tilde{\Gamma}(r)})\tilde{n}(r)dr+\frac{L}{2\pi\mu}\int_{S\e{1}}\tilde{\kappa}(r)\da{}(\frac{\bar{\Delta}\tilde{\Gamma}\cdot \tilde{n}(r)}{\abs{\Delta\tilde{\Gamma}}\e{2}}\Delta\tilde{\Gamma})dr\\
&\equiv \da{}F+\da{}G
\end{align*}
The term $\da{}G$ is estimated in Lemma \ref{lemderivG} then we only need to study $\da{}F$. Since $\da{}(\log\abs{\tilde{\Gamma}(\eta)-\tilde{\Gamma}(r)})=\frac{\bar{\Delta}\tilde{\Gamma}\cdot\da{}\tilde{\Gamma}(\eta)}{\abs{\bar{\Delta}\tilde{\Gamma}}\e{2}}$, using \eqref{D1D2} we can rewrite $\da{}F$ as follows
\begin{align}\label{derivF}
&\da{}F=-\frac{L}{2\pi\mu}\int_{S\e{1}}\frac{\tilde{\kappa}(r)}{\eta-r}\tilde{n}(r)dr-\frac{L}{2\pi\mu}\int_{S\e{1}}\tilde{\kappa}(r)\textbf{A}_1(\eta,r)\cdot\da{}\tilde{\Gamma}(\eta)\tilde{n}(r)dr
\end{align}
Using \eqref{D1D2} we can estimate in $L\e{\infty}(S\e{1})$ the second integral in $\da{}F$. For the first integral in $\da{}F$ we proceed as follows
\begin{align*}
&-\frac{L}{2\pi\mu}\int_{S\e{1}}\frac{\tilde{\kappa}(r)}{\eta-r}\tilde{n}(r)dr=-\frac{L}{2\pi\mu}\int_{S\e{1}}\frac{\tilde{\kappa}(r)}{\eta-r}\tilde{n}(\eta)dr+\frac{L\e{2}}{2\pi\mu}\int_{S\e{1}}\tilde{\kappa}(r)\int_0\e{1}\tilde{\kappa}(\eta+s(r-\eta))\tilde{\tau}(\eta+s(r-\eta))dsdr\\
&=-\frac{L}{2\mu}H(\tilde{\kappa})(\eta)n(\eta)-\frac{L\tilde{n}(\eta)}{2\pi\mu}\int_{S\e{1}}\tilde{\kappa}(r)(\frac{1}{\eta-r}-\frac{\pi}{\tan(\pi(\eta-r))})dr+\frac{L\e{2}}{2\pi\mu}\int_{S\e{1}}\tilde{\kappa}(r)\int_0\e{1}\tilde{\kappa}(\eta+s(r-\eta))\tilde{\tau}(\eta+s(r-\eta))dsdr\\
&=-\frac{L}{2\mu}H(\tilde{\kappa})(\eta)n(\eta)+F_1(\eta,r)
\end{align*}
where $\norm{F_1}_{L\e{\infty}(S\e{1})}\leq\est{3}$. Thus, $$\norm{\da{}F}_{L\e{2}(S\e{1})}\leq CL\norm{H(\tilde{\kappa})}_{L\e{2}(S\e{1})}+\est{3}\leq\est{3}$$
\end{proof}
We can know conclude the proof of Theorem \ref{thexistreform}

\begin{proof}[Proof of Theorem \ref{thexistreform}]
\hspace{1cm}\\
\vspace{0.1cm}

 $(i)$ \textit{Existence}
 
\vspace{0.1cm}
Our strategy to prove existence of solutions for \eqref{reformGFBP1}-\eqref{reformGFBP5} is to derive an estimate of the form $\frac{d}{dt}E[\tilde{\Gamma}]\leq\exp(E[\tilde{\Gamma}])$ where $E[\tilde{\Gamma}](t)$ is defined as 
\begin{equation}\label{energia}
E[\tilde{\Gamma}](t)=\norm{\Ftil}_{L\e{\infty}(S\e{1})}\e{2}+\norm{\tilde{\Gamma}}_{L\e{2}(S\e{1})}\e{2}+\norm{\da{3}\tilde{\Gamma}}\e{2}_{L\e{2}(S\e{1})  }.
\end{equation}
 In order to illustrate the main ideas in the argument we will assume first that the solutions of the problem are as regular as needed in order to make the following computations. Later we will consider a regularized version of the problem \eqref{reformGFBP1}-\eqref{reformGFBP5}(c.f. \eqref{regulgam}-\eqref{regulini}) and we will describe the modifications needed in the argument  to obtain a fully rigorous proof. We assume for the moment that $\tilde{\Gamma}$ is well defined and it is sufficiently smooth for $t\in[0,T]$ where $T>0$ will be made precise at the end of the proof.

We first estimate the contribution to $\frac{dE}{dt}$ due to the term $\norm{\tilde{\Gamma}}_{L\e{2}}\e{2}$ in \eqref{energia}. Multiplying \eqref{reformGFBP1} by $\tilde{\Gamma}$ and integrating in $S\e{1}$ we obtain,
\begin{align}\label{normal2}
&\frac{1}{2}\frac{d}{dt}\norm{\tilde{\Gamma}}\e{2}_{L\e{2}}=\int_{S\e{1}}\tilde{\Gamma}(\eta)\cdot\partial_t\tilde{\Gamma}(\eta)d\eta\\\nonumber
&=\int_{S\e{1}}\tilde{\Gamma}(\eta)\cdot \tilde{n}(\eta)u_0(\tilde{\Gamma}(\eta))\cdot \tilde{n}(\eta)d\eta+\int_{S\e{1}}\tilde{\Gamma}(\eta)\cdot\tilde{\tau}(\eta)\tilde{\psi}(\eta)d\eta+\int_{S\e{1}}\tilde{\Gamma}(\eta)\cdot\tilde{\tau}(\eta)\eta\partial_t L(t)d\eta\equiv I_1+I_2+I_3
\end{align}
where we drop the dependence of the functions on $t$ by simplicity.
Using \eqref{cotau0curva} we obtain that 
\begin{displaymath}
I_1\leq C\norm{\tilde{\Gamma}}_{L\e{2}(S\e{1})}\norm{u_0}_{L\e{2}(S\e{1})}\leq\est{3}
\end{displaymath}

We estimate $I_2$ in \eqref{normal2} as
\begin{align*}
I_2\leq C\norm{\tilde{\Gamma}}_{L\e{2}  }\norm{\tilde{\tau}}_{L\e{2}}\norm{\tilde{\psi}}_{L\e{\infty}(S\e{1})}
\end{align*}
Using \eqref{cotau0curva} and \eqref{reformGFBP2} we have
\begin{align}\nonumber
&\norm{\tilde{\psi}}_{L\e{\infty}(S\e{1})}\leq \sup_{\eta\in S\e{1}}\int_0\e{\eta}\abs{\tilde{\kappa}(r)u_0(\tilde{\Gamma}(r))\cdot \tilde{n}(r)dr}\leq C\int_{S\e{1}}\abs{\tilde{\kappa}(\eta)}\abs{u_0(\tilde{\Gamma}(\eta))\cdot \tilde{n}(\eta)}d\eta\leq C\norm{\tilde{\kappa}}_{L\e{2}}\norm{u_0(\tilde{\Gamma})}_{L\e{2}}\\\label{cotapsi}
&\leq\est{3}
\end{align}
To deal with the integral $I_3$ first we estimate $\abs{\partial_t L(t)}$ using \eqref{reformGFBP6}:
\begin{align*}
&\abs{\partial_t L(t)}\leq \abs{L(t)}\int_{S\e{1}}\abs{\tilde{\kappa}(\eta)}\abs{(u_0(\tilde{\Gamma})\cdot\tilde{n})(\eta)}d\eta\leq L(t)\norm{\tilde{\kappa}}_{L\e{2}}\norm{u_0(\tilde{\Gamma})}_{L\e{2}}\leq\est{3}
\end{align*}
this gives a similar estimate for $I_3$.
Therefore, 
\begin{equation}\label{cotaevolL2}
\frac{d}{dt}\norm{\tilde{\Gamma}}_{L\e{2}}\leq\est{3}
\end{equation}
We now estimate $\frac{d}{dt}\norm{\da{3}\tilde{\Gamma}}_{L\e{2}  }$, we differentiate three times $\partial_t\tilde{\Gamma}(\eta)$, multiply it by $\da{3}\tilde{\Gamma}(\eta)$ and integrate in $S\e{1}$. Using \eqref{reformGFBP1} we then obtain
\begin{align*}
&\frac{1}{2}\frac{d}{dt}\norm{\da{3}\tilde{\Gamma}}\e{2}_{L\e{2}  }=\int_{S\e{1}}\da{3}\tilde{\Gamma}(\eta)\cdot\da{3}(\partial_t\tilde{\Gamma}(\eta))d\eta=\int_{S\e{1}}\da{3}\tilde{\Gamma}(\eta)\cdot\da{3}((u_0(\tilde{\Gamma}(\eta))\cdot \tilde{n}(\eta))\tilde{n}(\eta))d\eta+\int_{S\e{1}}\da{3}\tilde{\Gamma}(\eta)\cdot\da{3}(\tilde{\psi}(\eta)\tilde{\tau}(\eta))d\eta\\
&+\partial_t L(t)\int_{S\e{1}}\da{3}\tilde{\Gamma}(\eta)\cdot\da{3}(\eta\tilde{\tau}(\eta))d\eta
\end{align*}
Using the following identities, which follow from \eqref{derivnormal}, \eqref{reformGFBP3} as well as the fact that $\abs{\da{}\tilde{\Gamma}}=L$,
\begin{align*}
&\da{3}\tilde{\Gamma}(\eta)=L(t)\e{2}(\da{}\tilde{\kappa}(\eta)\tilde{n}(\eta)-L(t)(\tilde{\kappa}(\eta))\e{2}\tilde{\tau}(\eta))\\
&\da{}\tilde{n}(\eta)=-L(t)\tilde{\kappa}(\eta)\tilde{\tau}(\eta)\\
&\da{}\tilde{\tau}(\eta)=L(t)\tilde{\kappa}(\eta)\tilde{n}(\eta)\\
&\da{}\tilde{\psi}(\eta)=L(t)\tilde{\kappa}(\eta)u_0(\tilde{\Gamma})\cdot\tilde{n}(\eta)\\
&\da{3}(\eta\tilde{\tau})(\eta)=3L(t)\da{}\tilde{\kappa}(\eta)\tilde{n}(\eta)-3L(t)\e{2}\tilde{\kappa}\e{2}(\eta)\tilde{\tau}(\eta)+L(t)\eta\da{2}\tilde{\kappa}(\eta)\tilde{n}(\eta)-L(t)\e{2}\eta\tilde{\kappa}(\eta)\da{}\tilde{\kappa}(\eta)\tilde{\tau}(\eta)-L(t)\e{3}\eta\tilde{\kappa}\e{3}(\eta)\tilde{n}(\eta)
\end{align*} we can write
\begin{align*}
\frac{1}{2}\frac{d}{dt}\norm{\da{3}\tilde{\Gamma}}_{L\e{2}  }\equiv\sum_{i=1}\e{9}J_i
\end{align*}
where
\begin{align*}
&J_1=L(t)\e{2}\int_{S\e{1}}\da{}\tilde{\kappa}(\eta)\da{3}u_0(\tilde{\Gamma}(\eta))\cdot \tilde{n}(\eta) d\eta\\
&J_2=-3L(t)\e{3}\int_{S\e{1}}\tilde{\kappa}(\eta)\da{}\tilde{\kappa}(\eta)\da{2}u_0(\tilde{\Gamma}(\eta))\cdot\tilde{\tau}(\eta)d\eta\\
&J_3=2L(t)\e{4}\int_{S\e{1}}(\tilde{\kappa}(\eta))\e{3}\da{2}u_0(\tilde{\Gamma}(\eta))\cdot \tilde{n}(\eta)d\eta\\
&J_4=-3L(t)\e{3}\int_{S\e{1}}(\da{}\tilde{\kappa}(\eta))\e{2}\da{} u_0(\tilde{\Gamma}(\eta))\cdot\tilde{\tau}(\eta)d\eta\\
&J_5=-2L(t)\e{4}\int_{S\e{1}}(\tilde{\kappa}(\eta))\e{2}\da{}\tilde{\kappa}(\eta)\da{}u_0(\tilde{\Gamma}(\eta))\cdot \tilde{n}(\eta)d\eta\\
&J_6=-4L(t)\e{5}\int_{S\e{1}}(\tilde{\kappa}(\eta))\e{4}\da{} u_0(\tilde{\Gamma}(\eta))\cdot\tilde{\tau}(\eta)d\eta\\
&J_7=-L(t)\e{3}\int_{S\e{1}}\da{}\tilde{\kappa}(\eta)\da{2}\tilde{\kappa}(\eta)u_0(\tilde{\Gamma}(\eta))\cdot\tilde{\tau}(\eta)d\eta\\
&J_8=-2L(t)\e{5}\int_{S\e{1}}(\tilde{\kappa}(\eta))\e{3}\da{}\tilde{\kappa}(\eta)u_0(\tilde{\Gamma})\cdot \tilde{\tau}(\eta)d\eta\\
&J_9=L(t)\e{4}\int_{S\e{1}}\da{}\tilde{\kappa}(\eta)\da{2}\tilde{\kappa}(\eta)\tilde{\psi}(\eta)d\eta\\
&J_{10}=2L(t)\e{6}\int_{S\e{1}}(\tilde{\kappa}(\eta))\e{3}\da{}\tilde{\kappa}(\eta)\tilde{\psi}(\eta)d\eta\\
&J_{11}=L(t)\e{3}\int_{S\e{1}}[3(\da{}\tilde{\kappa}(\eta))\e{2}-3L(t)\e{2}(\tilde{\kappa}(\eta))\e{4}+\eta\da{2}\tilde{\kappa}(\eta)\da{}\tilde{\kappa}(\eta)-2L(t)\e{2}\eta(\tilde{\kappa}(\eta))\e{3}\da{}\tilde{\kappa}(\eta)]d\eta
\end{align*}
Notice that using that $\da{}\tilde{\kappa}\da{2}\tilde{\kappa}=\frac{1}{2}\da{}((\da{}\tilde{\kappa})\e{2})$ , $\tilde{\kappa}\e{3}\da{}\tilde{\kappa}=\frac{1}{4}\da{}(\tilde{\kappa}\e{4})$ and integrating by parts in $J_7$, $J_9$ and $J_{11}$ we obtain that
\begin{align*}
&J_7+J_9=-\frac{1}{6}J_4-\frac{L(t)}{2}\int_{S\e{1}}\tilde{\kappa}(\eta)(\da{}\tilde{\kappa}(\eta))\e{2}u_0(\tilde{\Gamma}(\eta))\cdot\tilde{n}(\eta)d\eta\\
&J_{11}=L(t)\e{3}\int_{S\e{1}}[\frac{5}{2}(\da{}\tilde{\kappa}(\eta))\e{2}-\frac{5}{2}L(t)\e{2}(\tilde{\kappa}(\eta))\e{4}]d\eta
\end{align*}
Using \eqref{cotaderiv2u0L2} we can estimate $J_2\leq\est{3}$ and integrating by parts in $J_3$ we obtain
\begin{align*}
J_3=-6L\e{4}\int_{S\e{1}}(\tilde{\kappa}(\eta))\e{2}\da{}\tilde{\kappa}(\eta)\da{}u_0(\tilde{\Gamma})\cdot\tilde{n}(\eta)d\eta+2L\e{4}\int_{S\e{1}}(\tilde{\kappa}(\eta))\e{4}\da{}u_0(\tilde{\Gamma})\cdot\tilde{\tau}(\eta)d\eta
\end{align*}
Using \eqref{cotapsi} we can estimate $J_{10}\leq\est{3}$. Using Lemma \ref{lemu0galdi} and Lemma \ref{lemderivu0} all $J_i$ for $j=2,3,5,6,7,8,9,10,11$ can be estimated by $\est{3}$.

Therefore, it then only remains $J_1$ and $J_4$. 

Let us start with $J_1$. Since $u_0(\tilde{\Gamma}(\eta),t)=F+G+w$ we have that $\da{3}u_0(\tilde{\Gamma}(\eta))=\da{3}F+\da{3}G+\da{3}w$ therefore
\begin{align}\label{J1}
\frac{J_1}{L(t)\e{2}}=\int_{S\e{1}}\da{}\tilde{\kappa}(\eta)\da{3}F\cdot \tilde{n}(\eta)d\eta+\int_{S\e{1}}\da{}\tilde{\kappa}(\eta)\da{3}G\cdot \tilde{n}(\eta)d\eta+\int_{S\e{1}}\da{}\tilde{\kappa}(\eta)\da{3}w\cdot \tilde{n}(\eta)d\eta\equiv J_{1}\e{1}+J_{1}\e{2}+J_{1}\e{3}
\end{align}
Using \eqref{cotaw} we have
\begin{displaymath}
\norm{w}_{H\e{3}(S\e{1})}\leq C\norm{U-\bar{u}_0}_{H\e{\frac{1}{2}}(\partial D)}
\end{displaymath}
Therefore, 
\begin{equation}\label{J13}
J_{1}\e{3}\leq C\norm{\Ftil}_{L\e{\infty}(S\e{1})}\norm{\tilde{\Gamma}}_{H\e{3}(S\e{1})}\leq\est{3}
\end{equation}

Now in order to deal with the term $J_{1}\e{1}$ we can write it as follows
\begin{align*}
J_{1}\e{1}&=-\frac{L}{2\pi\mu}\int_{S\e{1}}\int_{S\e{1}}\da{} \tilde{\kappa}(\eta)\da{2}\tilde{\kappa}(\eta-r)\frac{\Delta\tilde{\Gamma}\cdot\tilde{\tau}(\eta)}{\abs{\Delta\tilde{\Gamma}}\e{2}}\tilde{n}(\eta-r)\cdot \tilde{n}(\eta)dr d\eta + l.o.t.\\
&=-\frac{L}{2\pi\mu}\int_{S\e{1}}\int_{S\e{1}}\da{}\tilde{\kappa}(\alpha)\da{2}\tilde{\kappa}(\eta-r)\frac{\Delta\tilde{\Gamma}\cdot\tilde{\tau}(\eta)}{\abs{\Delta\tilde{\Gamma}}\e{2}}dr d\eta\\
&-\frac{L}{2\pi\mu}\int_{S\e{1}}\int_{S\e{1}}\da{}\tilde{\kappa}(\eta)\da{2}\tilde{\kappa}(\eta-r)\frac{\Delta\tilde{\Gamma}\cdot\tilde{\tau}(\eta)}{\abs{\Delta\tilde{\Gamma}}\e{2}}(\tilde{n}(\eta-r)-\tilde{n}(\eta))\cdot \tilde{n}(\eta)dr d\eta+l.o.t.\\
&\equiv J_{1}\e{11}+J_{1}\e{12}+l.o.t.
\end{align*}
The terms $l.o.t.$ are the terms obtained differentiating $\da{3}F$ that yield at most one derivative of $\tilde{\kappa}$. Their contribution can be estimated as $l.o.t.\leq\est{3}$ using elementary inequalities. We will skip the details. The term $J_{1}\e{11}$ can now be estimated using Lemma \ref{lemcotadificil} for $f=\da{}\tilde{\kappa}$ and using that $\norm{\da{}\tilde{\kappa}}_{L\e{2}}\e{2}\leq C\frac{1}{L\e{4}}\norm{\tilde{\Gamma}}\e{2}_{H\e{3}(S\e{1})}$, as follows
\begin{align*}
&J_{1}\e{11}\leq -\frac{1}{2\pi\mu}\norm{\Lambda\e{\frac{1}{2}}(\da{}\tilde{\kappa})}\e{2}_{L\e{2}}+C\norm{\Ftil}_{L\e{\infty}(S\e{1})}\norm{\tilde{\Gamma}}_{H\e{3}(S\e{1})}\e{2}+C\norm{\Ftil}_{L\e{\infty}(S\e{1})}\e{\frac{3}{2}}\norm{\tilde{\Gamma}}_{H\e{3}(S\e{1})}\e{3}+C\norm{\Ftil}_{L\e{\infty}(S\e{1})}\e{2}\norm{\tilde{\Gamma}}_{H\e{3}(S\e{1})}\e{4}\\
&\leq-\frac{1}{2\pi\mu}\norm{\Lambda\e{\frac{1}{2}}(\da{}\tilde{\kappa})}\e{2}_{L\e{2}}+\est{3}
\end{align*}

In order to estimate $J_{1}\e{12}$ we use again $\da{2}\tilde{\kappa}(\eta-r)=-\dr{}\da{}\tilde{\kappa}(\eta-r)$ as well as $\dr{}(\tilde{n}(\eta-r)-\tilde{n}(\eta))=L\tilde{\kappa}(\eta-r)\tilde{\tau}(\eta-r)$
\begin{align}\label{J1122}\nonumber
&J_{1}\e{12}=-\frac{L\e{2}}{2\pi\mu}\int_{S\e{1}}\int_{S\e{1}}\da{}\tilde{\kappa}(\eta)\da{}\tilde{\kappa}(\eta-r)\dx{}(\frac{\Delta\tilde{\Gamma}\cdot\tilde{\tau}(\eta)}{\abs{\Delta\tilde{\Gamma}}\e{2}})r\int_{S\e{1}}\tilde{\kappa}(\alpha)\tilde{\tau}(\alpha)ds\cdot \tilde{n}(\eta)dr d\eta\\
&-\frac{L\e{2}}{2\pi\mu}\int_{S\e{1}}\int_{S\e{1}}\da{}\tilde{\kappa}(\eta)\da{}\tilde{\kappa}(\eta-r)\frac{\Delta\tilde{\Gamma}\cdot\tilde{\tau}(\eta)}{\abs{\Delta\tilde{\Gamma}}\e{2}}\tilde{\kappa}(\eta-r)\tilde{\tau}(\eta-r)\cdot\tilde{n}(\eta) dr d\eta\equiv J_{1}\e{121}+J_{1}\e{122}
\end{align}

Using \eqref{Mi} and $\tilde{\tau}(\alpha)-\tilde{\tau}(\eta)=rL\int_{S\e{1}}\tilde{\kappa}(\phi)\tilde{n}(\phi)$ we can write $J_{1}\e{121}$ as
\begin{align*}
&J_{1}\e{121}=\sum_{i=1}\e{7}-\frac{L\e{2}}{2\pi\mu}\int_{S\e{1}}\int_{S\e{1}}\da{}\tilde{\kappa}(\eta)\da{}\tilde{\kappa}(\eta-r)M_i(\eta,r)r\int_{S\e{1}}\tilde{\kappa}(\alpha)\tilde{\tau}(\alpha)ds\cdot \tilde{n}(\eta)dr d\eta\\
&-\frac{L\e{3}}{2\pi\mu}\int_{S\e{1}}\int_{S\e{1}}\da{}\tilde{\kappa}(\eta)\da{}\tilde{\kappa}(\eta-r)\int_{S\e{1}}\int_{S\e{1}}\tilde{\kappa}(\alpha)\tilde{\kappa}(\phi)\tilde{n}(\phi)ds\cdot \tilde{n}(\eta)dr d\eta
\end{align*}
Therefore using \eqref{cotaMi} we derive the estimate $\abs{r M_i(\eta,r)}\leq\est{3}$ and thus
\begin{displaymath}
J_{1}\e{121}\leq\est{3}
\end{displaymath}
Using $\tilde{\tau}(\eta-r)-\tilde{\tau}(\eta)=rL\int_{S\e{1}}\tilde{\kappa}(\alpha)\tilde{n}(\alpha)ds$ we can estimate $J_{1}\e{122}$ as follows
\begin{displaymath}
J_{1}\e{122}\leq C\frac{1}{L\e{5}}\norm{\Ftil}_{L\e{\infty}(S\e{1})}\norm{\tilde{\Gamma}}_{H\e{3}(S\e{1})}\e{4}\leq\est{3}
\end{displaymath}
Therefore, we have
\begin{align}\label{J11}
L\e{2}J_{1}\e{1}\leq-\frac{L\e{2}}{2\pi\mu}\norm{\Lambda\e{\frac{1}{2}}\da{}\kappa}\e{2}_{L\e{2}  }+ \est{3}
\end{align}
The fact that the first term of the right hand side of this estimate is negative is the crucial step in the proof of well posedness of \eqref{reformGFBP1}-\eqref{reformGFBP5} that we derive in this paper.

In order to deal with $J_{1}\e{2}$ we use \eqref{cotaderiv3G} in Lemma \ref{lema2der}  and we can rewrite
\begin{align*}
&J_{1}\e{2}=-8L\e{5}\int_{S\e{1}}\da{}\tilde{\kappa}(\eta)(\tilde{\kappa}(\eta))\e{2}H(\tilde{\kappa})(\eta)d\eta+L\int_{S\e{1}}\int_{S\e{1}}\da{}\tilde{\kappa}(\eta)\tilde{\kappa}(r)Q(\eta,r)drd\eta
\end{align*}
where $H$ is the periodic Hilbert transform defined in \eqref{hilbertop}.
Using the $L\e{2}$-estimates for the Hilbert transform, Cauchy-Schwartz as well as the estimate of $Q$ in Lemma \ref{lema2der}, we obtain 
\begin{align}\nonumber
&J_{1}\e{2}\leq C\norm{\da{}\tilde{\kappa}}_{L\e{2}(S\e{1})}\norm{\tilde{\kappa}}\e{2}_{L\e{\infty}(S\e{1})}\norm{H(\tilde{\kappa})}_{L\e{2}(S\e{1})}+\frac{C}{L\e{3}}\norm{\da{}\tilde{\kappa}}_{L\e{2}(S\e{1})}\norm{\tilde{\kappa}}_{L\e{\infty}(S\e{1})}\norm{Q}_{L\e{2}(S\e{1})}\\\label{J12}
&\leq\est{3}
\end{align}

Combining \eqref{J13}, \eqref{J11} and \eqref{J12} we arrive at
\begin{align*}
J_1\leq\est{3}
\end{align*}
The most singular term in $J_4$ is
\begin{align*}
&J_4=-3L\int_{S\e{1}}(\da{}\tilde{\kappa}(\eta))\e{2}\da{}F\cdot\tilde{\tau}(\eta) d\eta-3L\int_{S\e{1}}(\da{}\tilde{\kappa}(\eta))\e{2}\da{}G\cdot\tilde{\tau}(\eta)d\eta+l.o.t.\equiv J_4\e{1}+J_4\e{2}+l.o.t.
\end{align*}
where $l.o.t.$ contains the contribution of $w$ in \eqref{wsolucionu0}. Therefore this term can be estimated easily by $\est{3}$.
Using \eqref{derivF} we obtain
\begin{align*}
&J_4\e{1}=\frac{3L\e{2}}{2\pi\mu}\int_{S\e{1}}\int_{S\e{1}}(\da{}\tilde{\kappa}(\eta))\e{2}\frac{\tilde{\kappa}(r)}{\eta-r}\tilde{n}(r)\cdot\tilde{\tau}(\eta)dr d\eta+\frac{3L\e{2}}{2\pi\mu}\int_{S\e{1}}\int_{S\e{1}}(\da{}\tilde{\kappa}(\eta))\e{2}\tilde{\kappa}(r)\textbf{A}_1(\eta,r)\cdot\da{}\tilde{\Gamma}(\eta)\tilde{n}(r)\cdot\tilde{\tau}(\eta)drd\eta\\
&\equiv J_4\e{11}+J_4\e{12}
\end{align*}
Using the fact that $\tilde{n}(r)\cdot\tilde{\tau}(\eta)=\tilde{n}(r)\cdot(\tilde{\tau}(\eta)-\tilde{\tau}(r))=\mathcal{O}(\abs{\eta-r})$ we can estimate $J_4\e{11}$ by $\est{3}$.
The term $J_4\e{12}$ can be estimated by $\est{3}$ using \eqref{cotaA1}.

Using \eqref{derivG} of Lemma \ref{lemderivG} in $J_4\e{2}$ we can deduce that
\begin{displaymath}
J_4\e{2}\leq C\norm{\Ftil}_{L\e{\infty}(S\e{1})}\e{2}\norm{\tilde{\Gamma}}_{H\e{3}(S\e{1})}\e{3}+C\norm{\Ftil}_{L\e{\infty}(S\e{1})}\e{\frac{5}{2}}\norm{\tilde{\Gamma}}_{H\e{3}(S\e{1})}\e{3}\leq\est{3}
\end{displaymath}

Therefore, 
\begin{equation}\label{J4}
J_4\leq\est{3}
\end{equation}
Collecting all the estimates for the terms $J_i$ for $i=1,\cdots,11$, we obtain 
\begin{equation}\label{cotaevolH3}
\frac{d}{dt}\norm{\da{3}\tilde{\Gamma}}_{L\e{2}(S\e{1})  }\leq\est{3}
\end{equation}

Finally, we need to derive that 
\begin{equation}\label{cotasenF}
\frac{d}{dt}\norm{\Ftil}_{L\e{\infty}(S\e{1})  }\leq\est{3}.
\end{equation}
 To do that, we take $p>1$ and we have
\begin{align*}
\frac{d}{dt}\norm{\Ftil}_{L\e{p}  }\e{p}(t)=\frac{d}{dt}\int_\R\int_\R(\frac{r\e{2}}{\abs{\Delta\tilde{\Gamma}}\e{2}})\e{p}dr d\eta=-2p\int_\R\int_\R\frac{r\e{2p+1}}{\abs{\Delta\tilde{\Gamma}}\e{2p+2}}\int_{S\e{1}}\da{}\tilde{\Gamma}(\eta)ds\cdot\Delta\tilde{\Gamma}_t dr d\eta
\end{align*}
Since we can write 
\begin{align*}
&\Delta\tilde{\Gamma}_t=(u_0(\tilde{\Gamma}(\eta))\cdot \tilde{n}(\eta)-u_0(\tilde{\Gamma}(\eta-r))\cdot \tilde{n}(\eta-r))\tilde{n}(\eta-r)+u_0(\tilde{\Gamma}(\eta))\cdot \tilde{n}(\eta-r)(\tilde{n}(\eta)-\tilde{n}(\eta-r))\\
&+(\tilde{\psi}(\eta)-\tilde{\psi}(\eta-r))\tilde{\tau}(\eta-r)+\tilde{\psi}(\eta)(\tilde{\tau}(\eta)-\tilde{\tau}(\eta-r)))+\partial_t L(t)(\eta(\tilde{\tau}(\eta)-\tilde{\tau}(\eta-r))+r\tilde{\tau}(\eta-r))\\
&\equiv I_1+I_2+I_3+I_4+I_5
\end{align*}
we have
\begin{align*}
\frac{d}{dt}\norm{\Ftil}_{L\e{p}  }\e{p}(t)=\sum_{i=1}\e{5}-2p\int_\R\int_\R\frac{r\e{2p+2}}{\abs{\Delta\tilde{\Gamma}}\e{2p+2}}\int_{S\e{1}}\da{}\tilde{\Gamma}(\eta)ds\cdot \frac{I_i}{r} dr d\eta
\end{align*}
Using the fact that $f(\eta)-f(\eta-r)=r\int_{0}\e{1}\da{}f(\alpha)ds$ and the same techniques as for the norm $\norm{\tilde{\Gamma}}_{H\e{3}(S\e{1})  }$ we can see that 
\begin{displaymath}
\norm{\frac{I_{i}}{r}}_{L\e{\infty}(S\e{1})  }\leq\est{3}
\end{displaymath}
Therefore, 
\begin{displaymath}
\abs{\frac{d}{dt}\norm{\Ftil}_{L\e{p}(S\e{1}) }\e{p}(t)}\leq p\est{3}\norm{\Ftil}_{L\e{\infty}(S\e{1})  }\norm{\Ftil}_{L\e{p}(S\e{1})}\e{p}
\end{displaymath}
and 
\begin{displaymath}
\abs{\frac{d}{dt}\norm{\Ftil}_{L\e{p}(S\e{1})}(t)}\leq \est{3}\norm{\Ftil}_{L\e{p}(S\e{1})}
\end{displaymath}
Integrating on the interval $[t,t+h]$, $h>0$ and taking the limit as $p\to\infty$,
\begin{align}\label{cotaevolF}
&\abs{\norm{\Ftil}_{L\e{\infty}(S\e{1})  }(t+h)-\norm{\Ftil}_{L\e{\infty}(S\e{1})}(t)}\leq\int_t\e{t+h}\est{3}(s)\norm{\Ftil}_{L\e{\infty}(S\e{1})}(s)ds
\end{align}
Therefore, $\norm{\Ftil}_{L\e{\infty}(S\e{1})}(\cdot)\in W\e{1,\infty}([0,T))$. Then, $\frac{d}{dt}\norm{\Ftil}_{L\e{\infty}(S\e{1})}(t)$ is differentiable for a.e. $t\in[0,T]$. Moreover, using \eqref{cotaevolF} we have
\begin{align*}
&\frac{d}{dt}\norm{\Ftil}_{L\e{\infty}(S\e{1})  }(t)=\lim_{h\to 0}\frac{\norm{\Ftil}_{L\e{\infty}(S\e{1})  }(t+h)-\norm{\Ftil}_{L\e{\infty}(S\e{1})  }(t)}{h}\\
&\leq\lim_{h\to 0}\frac{\int_t\e{t+h}\est{3}(s)\norm{\Ftil}_{L\e{\infty}(S\e{1})}(s)ds}{h}\\
&\leq \est{3}\norm{\Ftil}_{L\e{\infty}(S\e{1})  }
\end{align*}
for a.e. $t\in[0,T]$.

Using that $\est{3}\leq C\exp(E[\tilde{\Gamma}](t))$ and combining \eqref{cotasenF}, \eqref{cotaevolL2} and \eqref{cotaevolH3} we obtain
\begin{equation}\label{cotaenergia}
\frac{d}{dt}E[\tilde{\Gamma}](t)\leq C\exp(E[\tilde{\Gamma}](t))\quad\textit{a.e.}\quad t\in[0,T]
\end{equation}

Therefore,
\begin{displaymath}
E[\tilde{\Gamma}](t)\leq -\log(\exp(-CE[\tilde{\Gamma}](0))-C\e{2}t))\quad\textit{a.e.}\quad t\in[0,T]
\end{displaymath}
as long as $T<\frac{\exp(-CE[\tilde{\Gamma}](0))}{C\e{2}}$.

To finish the proof we use a classical regularization procedure. To this end we replace \eqref{reformGFBP1}-\eqref{reformGFBP5} by the following problem
\begin{align}\label{regulgam}
&\partial_t\tilde{\Gamma}\e{\varepsilon}(\eta,t)=\mathcal{J}_{\varepsilon}(u\e{\varepsilon}_0(\tilde{\Gamma}(\eta,t),t)\cdot \mathcal{J}_{\varepsilon}\tilde{n}\e{\varepsilon}(\eta,t))\mathcal{J}_{\varepsilon}\tilde{n}\e{\varepsilon}(\eta,t)]+\mathcal{J}_{\varepsilon}[\tilde{\psi}\e{\varepsilon}(\eta,t)\mathcal{J}_{\varepsilon}\tilde{\tau}\e{\varepsilon}(\eta,t)]+\mathcal{J}_{\varepsilon}[(\partial_t L\e{\varepsilon}(t))\eta\mathcal{J}_{\varepsilon}\tilde{\tau}\e{\varepsilon}(\eta,t)]\\\label{regulL}
&\partial_t L\e{\varepsilon}(t)=-L\e{\varepsilon}(t)\int_{S\e{1}}\tilde{\kappa}\e{\varepsilon}(\eta,t)u\e{\varepsilon}_0(\tilde{\Gamma}\e{\varepsilon}(\eta,t),t)\cdot\tilde{n}\e{\varepsilon}(\eta,t)d\eta
\end{align}
where
\begin{align}
&\tilde{\psi}\e{\varepsilon}(\eta,t)=\int_0\e{\eta}\tilde{\kappa}\e{\varepsilon}(r,t)u\e{\varepsilon}_0(\tilde{\Gamma}(r,t),t)\cdot\tilde{n}\e{\varepsilon}(r,t)dr\\\nonumber
&u\e{\varepsilon}_0(\tilde{\Gamma}\e{\varepsilon
})(\eta,t)=\bar{u}\e{\varepsilon}_0(\tilde{\Gamma}\e{\varepsilon})(\eta,t)+w\e{\varepsilon}(\eta,t)=-\frac{L\e{\varepsilon}(t)}{2\pi\mu}\int_{S\e{1}}\mathcal{J}_{\varepsilon}\tilde{\kappa}\e{\varepsilon}(r,t)\log(\abs{\tilde{\Gamma}\e{\varepsilon}(\eta,t)-\tilde{\Gamma}\e{\varepsilon}(r,t)}\e{2}+\varepsilon\e{2}(L\e{\varepsilon}(t))\e{2})\tilde{n}\e{\varepsilon}(r,t)dr\\\label{u0epsilon}
&+\frac{L\e{\varepsilon}(t)}{2\pi\mu}\int_{S\e{1}}\tilde{\kappa}\e{\varepsilon}(r,t)\frac{\Delta\tilde{\Gamma}\e{\varepsilon}\cdot \tilde{n}\e{\varepsilon}(r,t)}{\abs{\Delta\tilde{\Gamma}\e{\varepsilon}}\e{2}+\varepsilon\e{2}(L\e{\varepsilon}(t))\e{2}}\Delta\tilde{\Gamma}\e{\varepsilon} dr+w\e{\varepsilon}(\eta,t)\equiv F_\varepsilon+G_\varepsilon+w\e{\varepsilon}
\end{align} 
with initial values
\begin{equation}\label{regulini}
\tilde{\Gamma}\e{\varepsilon}(\eta,0)=\mathcal{J}_{\varepsilon}\tilde{\Gamma}_{0},\quad L\e{\varepsilon}(0)=L_0
\end{equation} 
where $\mathcal{J}_{\varepsilon}(f)=\rho_\varepsilon *f$ with $\rho_\varepsilon(\eta)=\frac{1}{\varepsilon}\rho(\frac{\eta}{\varepsilon})$, $\rho(\abs{\eta})\in C_c\e{\infty}(\R)$, $\rho(\eta)\geq 0$ and $\int_{\R}\rho(\eta)d\eta=1$ is a standard mollifier.

The well posedness of \eqref{regulgam}-\eqref{regulini} for each $\varepsilon>0$ is standard, due to the smoothness of the right hand side of these equations, and it can be proved applying Picard’s Theorem after reformulating the problem as a fixed point problem for a contractive operator. We can now obtain uniform estimates for $E(t)$ defined in \eqref{energia} mimicking the arguments made before for $\varepsilon=0$. These arguments can now be made in a rigorous manner due to the smoothness of the functions involved. In particular, the estimate \eqref{cotaenergia} holds, with $C$ independent on $\varepsilon$, as long as the solution  of \eqref{regulgam}-\eqref{regulini} is defined. The fact that $C$ is independent on $\varepsilon$ requires to estimate carefully the terms $F_\varepsilon$ and $G_\varepsilon$ in \eqref{u0epsilon}. Notice that the regularization of the problem \eqref{regulgam}-\eqref{regulini} has been made in a self-adjoint manner, something that it is convenient in order to estimate the quadratic functionals $\norm{\tilde{\Gamma}_\varepsilon}\e{2}_{L\e{2}(S\e{1})}$ and $\norm{\da{3}\tilde{\Gamma}_\varepsilon}\e{2}_{L\e{2}(S\e{1})}$. More precisely, computing $\frac{d}{dt}(\norm{\tilde{\Gamma}_\varepsilon}\e{2}_{L\e{2}(S\e{1})}+\norm{\da{3}\tilde{\Gamma}_\varepsilon}\e{2}_{L\e{2}(S\e{1})})$ we obtain
\begin{equation}\label{cotaepsilonenergia}
\frac{d}{dt}(\norm{\tilde{\Gamma}_\varepsilon}\e{2}_{L\e{2}(S\e{1})}+\norm{\da{3}\tilde{\Gamma}_\varepsilon}\e{2}_{L\e{2}(S\e{1})})\leq-C\norm{\Lambda_\varepsilon\e{\frac{1}{2}}(\mathcal{J}_\varepsilon\tilde{\Gamma}_\varepsilon)}\e{2}_{L\e{2}(S\e{1})}+C\exp(\norm{\mathcal{F}(\tilde{\Gamma}_\varepsilon)}_{L\e{\infty}(S\e{1})}\e{2}+b\norm{\mathcal{J}_\varepsilon\tilde{\Gamma}_\varepsilon}_{H\e{3}(S\e{1})}\e{2})
\end{equation}
where $C>0$ is independent on $\varepsilon$ and $\Lambda_\varepsilon$ is the regularized fractional Laplacian given by
\begin{displaymath}
\Lambda_\varepsilon(f)=\pi PV\int_{S\e{1}}\frac{f(x)-f(x-y)}{\sin\e{2}(\pi y)+\varepsilon\e{2}}
\end{displaymath} 
 and where we define $\Lambda_\varepsilon\e{\frac{1}{2}}$ as the unique operator that $(\Lambda_\varepsilon\e{\frac{1}{2}})\e{2}=\Lambda_\varepsilon$. The constant $b>0$ can be chosen arbitrary small, although the constant $C$ in front of the exponential depends on $b$. The first term in the right hand side on \eqref{cotaepsilonenergia} is part of the contribution due to the term $F_\varepsilon$ in \eqref{u0epsilon}. More precisely, we need to estimate the term $\int_{S\e{1}}\mathcal{J_\varepsilon}\da{3}\tilde{\Gamma}_\varepsilon(\eta)\da{3}F_\varepsilon(\eta)d\eta$. This term can be estimated as the term $J_1\e{1}$ in \eqref{J1}. To this end we need to prove a result analogous to Lemma \ref{lemcotadificil} where $\frac{\Delta\tilde{\Gamma}\cdot\tilde{\tau}}{\abs{\Delta\tilde{\Gamma}}\e{2}}$ must be replaced by $\frac{\Delta\tilde{\Gamma}_\varepsilon\cdot\tilde{\tau}_\varepsilon}{\abs{\Delta\tilde{\Gamma}
_\varepsilon}\e{2}+\varepsilon\e{2}L_\varepsilon\e{2}}$. This term appears differentiating $\log(\abs{\tilde{\Gamma}\e{\varepsilon}(\eta)-\tilde{\Gamma}\e{\varepsilon}(r)}\e{2}+\varepsilon\e{2}(L\e{\varepsilon}(t))\e{2})$ in $F_\varepsilon$. It turns out that a modified version of Lemma \ref{lemcotadificil} holds with $\Lambda$ replaced by $\Lambda_\varepsilon$. The remaining terms that arise in $\frac{d}{dt}(\norm{\tilde{\Gamma}_\varepsilon}\e{2}_{L\e{2}(S\e{1})}+\norm{\da{3}\tilde{\Gamma}_\varepsilon}\e{2}_{L\e{2}(S\e{1})})$ can be estimated with tedious, but straightforward adaptations, in the arguments that yield \eqref{cotaevolL2} and \eqref{cotaevolH3}. We just remark that in this arguments we need to use, instead of the the usual Hilbert transform in \eqref{hilbertop}, the regularized Hilbert transform given by 
 
\begin{equation*}
H_\varepsilon(f)=PV\int_S\e{1}\frac{f(r)\sin(\pi(\eta-r))}{\tan\e{2}(\pi(\eta-r))+\varepsilon\e{2}}dr
\end{equation*}

Using the well-known fact that $\norm{\mathcal{J}_\varepsilon\tilde{\Gamma}_\varepsilon
}_{H\e{3}(S\e{1})}\leq C\norm{\tilde{\Gamma}_\varepsilon}_{H\e{3}(S\e{1})}$ (c.f. \cite{majda}) it then follows that 
\begin{equation}\label{cotaevolGamregu}
\frac{d}{dt}(\norm{\tilde{\Gamma}_\varepsilon}\e{2}_{L\e{2}(S\e{1})}+\norm{\da{3}\tilde{\Gamma}_\varepsilon}\e{2}_{L\e{2}(S\e{1})})\leq C\exp(\norm{\mathcal{F}(\tilde{\Gamma}_\varepsilon)}_{L\e{\infty}(S\e{1})}\e{2}+\norm{\tilde{\Gamma}_\varepsilon}\e{2}_{L\e{2}(S\e{1})}+\norm{\da{3}\tilde{\Gamma}_\varepsilon}\e{2}_{L\e{2}(S\e{1})})
\end{equation}

On the other hand arguing exactly as in the proof of \eqref{cotaevolF} we obtain
\begin{equation}\label{cotaevolFregu}
\frac{d}{dt}\norm{\mathcal{F}(\tilde{\Gamma}_\varepsilon)}_{L\e{\infty}(S\e{1})}\e{2}\leq C\exp(\norm{\mathcal{F}(\tilde{\Gamma}_\varepsilon)}_{L\e{\infty}(S\e{1})}\e{2}+\norm{\tilde{\Gamma}_\varepsilon}\e{2}_{L\e{2}(S\e{1})}+\norm{\da{3}\tilde{\Gamma}_\varepsilon}\e{2}_{L\e{2}(S\e{1})}).
\end{equation}
Adding \eqref{cotaevolGamregu} and \eqref{cotaevolFregu} and applying a Grönwall argument  

\begin{equation}\label{cotagronwall}
\norm{\mathcal{F}(\tilde{\Gamma}_\varepsilon)}_{L\e{\infty}(S\e{1})}\e{2}+\norm{\tilde{\Gamma}_\varepsilon}\e{2}_{L\e{2}(S\e{1})}+\norm{\da{3}\tilde{\Gamma}_\varepsilon}\e{2}_{L\e{2}(S\e{1})}\leq C(T)
\end{equation}
in $t\in[0,T]$ with $C(T)>0$ and $T>0$ independent of $\varepsilon$. We now claim that
\begin{equation}\label{esttiempo}
\sup_{0\leq t\leq T}(\norm{\partial_t\tilde{\Gamma}_\varepsilon(\cdot,t)}_{H\e{2}(S\e{1})}+\abs{\partial_tL_\varepsilon}(t))\leq C(T)
\end{equation}
with $C(T)$ independent of $\varepsilon$. The estimate of the first term in \eqref{esttiempo} can be obtained differentiating twice \eqref{regulgam}. The more singular term in $\da{2}u_0(\tilde{\Gamma}_\varepsilon)$ is $\da{2}F_\varepsilon$. Using  \eqref{u0epsilon} and doing an integration by parts we obtain
\begin{displaymath}
\da{2}F_\varepsilon=-\frac{L_\varepsilon(t)}{2\pi\mu}\int_{S\e{1}}\da{}\mathcal{J}_\varepsilon\tilde{\kappa}_\varepsilon(\eta-r)\frac{\Delta\tilde{\Gamma}_\varepsilon\cdot\tilde{\tau}_\varepsilon(\eta)}{\abs{\Delta\tilde{\Gamma}_\varepsilon}\e{2}+\varepsilon\e{2}}\tilde{n}_\varepsilon(\eta-r)dr+l.o.t.
\end{displaymath}
Approximating the integral operator as a Hilbert transform as in the previous arguments we obtain
\begin{displaymath}
\norm{\da{2}u_0(\tilde{\Gamma}_\varepsilon)}_{L\e{2}(S\e{1})}\leq C\norm{\tilde{\Gamma}_\varepsilon}_{H\e{3}(S\e{1})}
\end{displaymath}
The the contribution of the term $G_\varepsilon$ in  \eqref{u0epsilon} can be estimated in a similar manner. The term $w_\varepsilon$ in \eqref{u0epsilon} can be estimated using that $U\in W\e{1,\infty}([0,T],H\e{\frac{1}{2}}(\partial D))$. The remaining terms in \eqref{regulgam} can be bounded in analogous way. This yields the estimate for $\norm{\partial_t\tilde{\Gamma}_\varepsilon(\cdot,t)}_{H\e{2}(S\e{1})}$ in \eqref{esttiempo}. The estimate for $\abs{\partial_tL_\varepsilon}(t)$ follows from \eqref{regulL} and \eqref{cotagronwall}.

A standard compactness argument shows that there exist a sequence $\{\varepsilon_n\}_{n\in\N}$ with $\varepsilon_n\to 0$,\newline $\tilde{\Gamma}\in W\e{1,\infty}([0,T],H\e{2}(S\e{1}))\cap L\e{\infty}([0,T],H\e{3}(S\e{1}))$ and $L\in W\e{1,\infty}([0,T])$ such that
\begin{align*}
&\tilde{\Gamma}_{\varepsilon_n}\to \tilde{\Gamma}\quad\textit{in}\quad L\e{2}([0,T],H\e{2}(S\e{1})),\\
&\tilde{\Gamma}_{\varepsilon_n}\rightharpoonup  \tilde{\Gamma}\quad\textit{in}\quad W\e{1,2}([0,T],H\e{2}(S\e{1})),\\
&\tilde{\Gamma}_{\varepsilon_n}(\eta,t)\to\tilde{\Gamma}(\eta,t)\quad\textit{for a.e.}\quad(\eta,t)\in \tilde{\mathcal{G}}_T,\\
&L_{\varepsilon_n}\to L\quad\textit{in}\quad L\e{2}([0,T]),\\
&L_{\varepsilon_n}(t)\to L(t)\quad\textit{for a.e.}\quad t\in[0,T]
\end{align*}

Taking the limit in \eqref{regulgam}-\eqref{regulini} we obtain that $(\tilde{\Gamma},L)$ satisfy \eqref{reformGFBP1}-\eqref{reformGFBP5} a.e. $(\eta,t)\in\tilde{\mathcal{G}}_T$. Then the existence part of the Theorem follows.

\vspace{0.2cm}

$(ii)$\textit{Uniqueness}
 
\vspace{0.2cm}

In order to show uniqueness we suppose that we have two solutions of \eqref{reformGFBP1}-\eqref{reformGFBP5}, $(\tilde{\Gamma}_1,L_1(t)),(\tilde{\Gamma}_2,L_2(t))\in W\e{1,\infty}([0,T],H\e{2}(S\e{1}))\cap L\e{\infty}([0,T],H\e{3}(S\e{1}))\times W\e{1,\infty}([0,T])$. We define $f=\tilde{\Gamma}_1-\tilde{\Gamma}_2$. Our goal is to prove the following inequality
\begin{equation}\label{estimunic}
\frac{d}{dt}(\norm{f}_{L\e{2}(S\e{1})}\e{2}+\norm{\da{2}f}\e{2}_{L\e{2}(S\e{1})}+\abs{L_1-L_2}\e{2})\leq C(\norm{f}_{L\e{2}(S\e{1})}\e{2}+\norm{\da{2}f}\e{2}_{L\e{2}(S\e{1})}+\abs{L_1-L_2}\e{2})
\end{equation}
where $C$ is a constant that only depends on the energy defined in \eqref{energia} for $\tilde{\Gamma}_1$ and $\tilde{\Gamma}_2$, that is, $E[\tilde{\Gamma}_1](t)$ and $E[\tilde{\Gamma}_2](t)$.

In order to prove \eqref{estimunic} we decompose the evolution of $\norm{\da{2}f}_{L\e{2}(S\e{1})}\e{2}$, using \eqref{reformGFBP1}, as follows
\begin{align}\label{normh2f}\nonumber
&\frac{1}{2}\frac{d}{dt}\norm{\da{2}f}_{L\e{2}(S\e{1})}\e{2}=\int_{S\e{1}}\da{2}f\cdot\da{2}(f_t)d\eta=\int_{S\e{1}}\da{2}f\cdot\da{2}\big(((u_0(\tilde{\Gamma}_1)-u_0(\tilde{\Gamma}_2))\cdot \tilde{n}_1)\tilde{n}_1\big)d\eta\\\nonumber
&+\int_{S\e{1}}\da{2}f\cdot\da{2}\big((u_0(\tilde{\Gamma}_2)\cdot(\tilde{n}_1-\tilde{n}_2))\tilde{n}_1\big)d\eta+\int_{S\e{1}}\da{2}f\cdot\da{2}\big((u_0(\tilde{\Gamma}_2)\cdot \tilde{n}_2)(\tilde{n}_1-\tilde{n}_2)\big)d\eta+\int_{S\e{1}}\da{2}f\cdot\da{2}\big((\psi_1-\psi_2)\tilde{\tau}_1\big)d\eta\\\nonumber
&+\int_{S\e{1}}\da{2}f\cdot\da{2}\big(\psi_2(\tilde{\tau}_1-\tilde{\tau}_2)\big)d\eta+\partial_t(L_1-L_2)\int_{S\e{1}}\da{2}f\cdot\da{2}\big(\eta\tilde{\tau}_1\big)d\eta+\partial_t L_2\int_{S\e{1}}\da{2}f\cdot\da{2}\big(\eta(\tilde{\tau}_1-\tilde{\tau}_2)\big)d\eta\\
&=\int_{S\e{1}}(\da{2}f\cdot \tilde{n}_1)\big(\da{2}(u_0(\tilde{\Gamma}_1)-u_0(\tilde{\Gamma}_2))\cdot \tilde{n}_1\big)d\eta +l.o.t.\equiv I+l.o.t.
\end{align}
where we have denoted 
\begin{equation}\label{tauini}
\tilde{\tau}_i=\frac{1}{L_i(t)}\da{}\tilde{\Gamma}_i,\quad \tilde{n}_i=\frac{1}{L_i(t)}(\da{}\tilde{\Gamma}_i)\e{\perp}\quad\textit{for}\quad i=1,2
\end{equation}
and $\psi_i$ is the function defined by means of \eqref{reformGFBP2} for $\tilde{\Gamma}_i$ for $i=1,2$. We notice that since $\tilde{\Gamma}_1(\cdot,t),\tilde{\Gamma}_2(\cdot,t)\in H\e{3}(S\e{1})$ for $t\in[0,T]$, all terms in \eqref{normh2f} are well defined. In particular, due to \eqref{u0gamma} we have $u_0(\tilde{\Gamma}_i(\cdot,t))\in H\e{2}(S\e{1})$ for $t\in[0,T]$ and $i=1,2$. Indeed the term $F$ on the right hand side of \eqref{u0gamma} is in $H\e{2}(S\e{1})$ due to the classical theory of singular integrals (c.f. \cite{stein}) as well as the fact that $\tilde{\kappa}\in H\e{1}(S\e{1})$. The term $G$ in \eqref{u0gamma} can be differentiated twice due to \eqref{derivG} in  Lemma \ref{lemderivG} and $\tilde{\kappa}\in H\e{1}(S\e{1})$.
The remaining terms in \eqref{normh2f} contain at most two derivatives of $\psi_i$, $\tilde{\tau}_i$ and $\tilde{n}_i$ that can be estimated in terms of $\tilde{\Gamma}_i(\cdot,t)\in H\e{3}(S\e{1})$ due to \eqref{reformGFBP2} and \eqref{tauini}.

We will give now details on how to estimate the term $I$ which is the most singular one in \eqref{normh2f}, the rest of the terms can be estimated following the same ideas as in the previous energy estimates. 

Using \eqref{u0gamma} we can obtain that
\begin{displaymath}
\da{2}u_0(\tilde{\Gamma})(\eta)=-\frac{L}{2\pi\mu}\int_{S\e{1}}\da{}\tilde{\kappa}(\eta-r)\frac{\Delta\tilde{\Gamma}\cdot\tilde{\tau}(\eta)}{\abs{\Delta\tilde{\Gamma}}\e{2}}n(\eta-r)dr  + l.o.t.
\end{displaymath}
therefore
\begin{align*}
&I=-\frac{(L_1-L_2)}{2\pi\mu}\int_{S\e{1}}\int_{S\e{1}}\da{2}f(\eta)\cdot \tilde{n}_1(\eta)\da{}\tilde{\kappa}_1(\eta-r)\frac{\Delta\tilde{\Gamma}_1\cdot\tilde{\tau}_1(\eta)}{\abs{\Delta\tilde{\Gamma}_1}\e{2}}\tilde{n}_1(\eta-r)\cdot \tilde{n}_1(\eta)drd\eta\\
&-\frac{L_2}{2\pi\mu}\int_{S\e{1}}\int_{S\e{1}}\da{2}f(\eta)\cdot \tilde{n}_1(\eta)\da{}(\tilde{\kappa}_1-\tilde{\kappa}_2)(\eta-r)\frac{\Delta\tilde{\Gamma}_1\cdot\tilde{\tau}_1(\eta)}{\abs{\Delta\tilde{\Gamma}_1}\e{2}}\tilde{n}_1(\eta-r)\cdot \tilde{n}_1(\eta)drd\eta\\
&-\frac{L_2}{2\pi\mu}\int_{S\e{1}}\int_{S\e{1}}\da{2}f(\eta)\cdot \tilde{n}_1(\eta)\da{}\tilde{\kappa}_2(\eta-r)\frac{\Delta f\cdot\tilde{\tau}_1(\eta)}{\abs{\Delta\tilde{\Gamma}_1}\e{2}}\tilde{n}_1(\eta-r)\cdot \tilde{n}_1(\eta)drd\eta\\
&-\frac{L_2}{2\pi\mu}\int_{S\e{1}}\int_{S\e{1}}\da{2}f(\eta)\cdot \tilde{n}_1(\eta)\da{}\tilde{\kappa}_2(\eta-r)\frac{\Delta \tilde{\Gamma}_2\cdot(\tilde{\tau}_1-\tilde{\tau}_2)(\eta)}{\abs{\Delta\tilde{\Gamma}_1}\e{2}}\tilde{n}_1(\eta-r)\cdot \tilde{n}_1(\eta)drd\eta\\
&-\frac{L_2}{2\pi\mu}\int_{S\e{1}}\int_{S\e{1}}\da{2}f(\eta)\cdot \tilde{n}_1(\eta)\da{}\tilde{\kappa}_2(\eta-r)\frac{\Delta\tilde{\Gamma}_2\cdot\tilde{\tau}_2(\eta)}{\abs{\Delta\tilde{\Gamma}_1}\e{2}}(\tilde{n}_1-\tilde{n}_2)(\eta-r)\cdot \tilde{n}_1(\eta)drd\eta\\
&-\frac{L_2}{2\pi\mu}\int_{S\e{1}}\int_{S\e{1}}\da{2}f(\eta)\cdot \tilde{n}_1(\eta)\da{}\tilde{\kappa}_2(\eta-r)\Delta\tilde{\Gamma}_2\cdot\tilde{\tau}_2(\eta)\tilde{n}_2(\eta-r)\cdot \tilde{n}_1(\eta)(\frac{1}{\abs{\Delta\tilde{\Gamma}_1}\e{2}}-\frac{1}{\abs{\Delta\tilde{\Gamma}_2}\e{2}})drd\eta +l.o.t.\equiv \sum_{i=1}\e{6}J_i+l.o.t.
\end{align*}
where $\tilde{\kappa}_i(\eta,t)=\frac{1}{L_i(t)\e{2}}\da{2}\tilde{\Gamma}_i(\eta,t)\cdot \tilde{n}_1(\eta,t)$ for $i=1,2$.

We estimate $J_1$ approximating $\frac{\Delta\tilde{\Gamma}_1\cdot\tilde{\tau}_1(\eta)}{\abs{\Delta\tilde{\Gamma}_1}\e{2}}$ as the sum of $\frac{1}{r L_1}$ and more regular terms. Then,
\begin{align*}
J_1=-\frac{L_1-L_2}{2\pi\mu}\int_{S\e{1}}\da{2}f(\eta)\cdot \tilde{n}_1(\eta)H(\da{}\tilde{\kappa}_1)(\eta)d\eta+J_{11}
\end{align*}
where $J_{11}\leq C\abs{L_1-L_2}\norm{f}_{H\e{2}}$. Using that $\tilde{\Gamma}_1\in H\e{3}(S\e{1})$, standard estimates for the Hilbert transform and Cauchy-Schwarz and Young's inequalities, we obtain
\begin{displaymath}
J_1\leq C\abs{L_1-L_2}\norm{\da{2}f}_{L\e{2}(S\e{1})}\norm{\da{}\tilde{\kappa}_1}_{L\e{2}(S\e{1})}+C\abs{L_1-L_2}\norm{f}_{H\e{2}}\leq C(\abs{L_1-L_2}\e{2}+\norm{f}_{H\e{2}(S\e{1})}\e{2})
\end{displaymath}
Using $M_2(\tilde{\Gamma}_1)$ in \eqref{M2} for $\tilde{\Gamma}_1$ and $\Delta f=r\int_0\e{1}\da{}f(\alpha)ds$ where $\alpha=\eta-r+sr$, we can estimate
\begin{align*}
J_3\leq-\frac{L_2}{2L_1\e{2}\pi\mu}\int_{S\e{1}}\da{2}f(\eta)\cdot \tilde{n}_1(\eta)H(\da{}\tilde{\kappa}_2)(\eta)\da{}f(\eta)\cdot\tilde{\tau}_1(\eta)d\eta + C\norm{f}\e{2}_{H\e{2}(S\e{1})}\norm{\da{}\tilde{\kappa}_2}_{L\e{2}}\norm{M_2(\tilde{\Gamma}_1)}_{L\e{\infty}(S\e{1})}\leq C\norm{f}\e{2}_{H\e{2}(S\e{1})}
\end{align*} 
The terms $J_4$ and $J_5$ can be estimated in an  analogous way. To deal with $J_6$ we just need to compute 
\begin{displaymath}
(\frac{1}{\abs{\Delta\tilde{\Gamma}_1}\e{2}}-\frac{1}{\abs{\Delta\tilde{\Gamma}_2}\e{2}})=-\frac{\mathcal{F}(\tilde{\Gamma}_2)}{\abs{\Delta\tilde{\Gamma}_1}\e{2}}\int_0\e{1}\da{}f(\alpha)\cdot(\da{}\tilde{\Gamma}_1(\alpha)+\da{}\tilde{\Gamma}_2(\alpha))ds.
\end{displaymath} 
Using again $M_2(\tilde{\Gamma}_1)$, we obtain that $J_6\leq C\norm{f}_{H\e{2}(S\e{1})}\e{2}$.

Finally, we estimate $J_2$ which is the most singular term. Using that $\tilde{n}_1(\eta-r)-\tilde{n}_1(\eta)=rL_1\int_{0}\e{1}\tilde{\kappa}_1(\eta-rs)\tilde{\tau}_1(\eta-rs)ds$ and $\Delta\tilde{\Gamma}_1=rL_1\int_0\e{1}\tilde{\tau}_1(\alpha)ds$ we can rewrite $J_2=J_{21}+J_{22}$ where
\begin{align*}
&J_{21}=-\frac{L_2}{2\pi\mu}\int_{S\e{1}}\int_{S\e{1}}\da{2}f(\eta)\cdot \tilde{n}_1(\eta)\da{}(\tilde{\kappa}_1-\tilde{\kappa}_2)(\eta-r)\frac{\Delta\tilde{\Gamma}_1\cdot\tilde{\tau}_1(\eta)}{\abs{\Delta\tilde{\Gamma}_1}\e{2}}drd\eta\\
&J_{22}=-\frac{L_2L_1\e{2}}{2\pi\mu}\int_{S\e{1}}\int_{S\e{1}}\da{2}f(\eta)\cdot \tilde{n}_1(\eta)\da{}(\tilde{\kappa}_1-\tilde{\kappa}_2)(\eta-r)\mathcal{F}(\tilde{\Gamma}_1)\int_{0}\e{1}\tilde{\kappa}_1(\alpha)\tilde{\tau}_1(\alpha)\cdot \tilde{n}_1(\eta)drd\eta
\end{align*}
In order to estimate $J_{22}$ we do integration by parts using that $\da{}\tilde{\kappa}_1(\eta-r)=-\partial_r\tilde{\kappa}_1(\eta-r)$ and we can proceed as in the term $J_1\e{122}$ in \eqref{J1122} in the energy estimates. Therefore, $J_{22}\leq C\norm{\da{2}f}_{L\e{2}(S\e{1})}\e{2}$. Finally, we can write $J_{21}$ as follows
\begin{displaymath}
J_{21}=-\frac{L_2}{2L_1\pi\mu}\int_{S\e{1}}\da{2}f(\eta)\cdot \tilde{n}_1(\eta)H(\da{}(\tilde{\kappa}_1-\tilde{\kappa}_2))(\eta)d\eta-\frac{L_2}{2\pi\mu}\int_{S\e{1}}\int_{S\e{1}}\da{2}f(\eta)\cdot \tilde{n}_1(\eta)\da{}(\tilde{\kappa}_1-\tilde{\kappa}_2)(\eta-r)A(\tilde{\Gamma}_1)drd\eta\equiv J_{21}\e{1}+J_{21}\e{2}
\end{displaymath}
where $A(\tilde{\Gamma}_1)=\frac{\Delta\tilde{\Gamma}_1\cdot\tilde{\tau}_1(\eta)}{\abs{\Delta\tilde{\Gamma}_1}\e{2}}-\frac{\pi\da{}\tilde{\Gamma}_1(\eta)\cdot\tilde{\tau}_1(\eta)}{\tan(\pi r)\abs{\da{}\tilde{\Gamma}_1(\eta)}\e{2}}$.

The term $J_{21}\e{2}$ can be estimated as the term $J_{22}$ using $\da{}(\tilde{\kappa_1}-\tilde{\kappa}_2)(\eta-r)=-\partial_r(\tilde{\kappa}_1-\tilde{\kappa}_2)(\eta-r)$ and integration by parts. Then using \eqref{operadorA} with $\tilde{\Gamma}_1$ instead of $\tilde{\Gamma}$  we obtain $\abs{\partial_r A(\tilde{\Gamma}_1)}\leq C E[\tilde{\Gamma}_1]$. Therefore
\begin{align*}
J_{21}\e{2}=-\frac{L_2}{2L_1\pi\mu}\int_{S\e{1}}\int_{S\e{1}}\da{2}f(\eta)\cdot \tilde{n}_1(\eta)(\tilde{\kappa}_1-\tilde{\kappa}_2)(\eta-r)\partial_r A(\tilde{\Gamma}_1)drd\eta\leq C\norm{f}_{H\e{2}(S\e{1})}\e{2}
\end{align*}
Using that $\Lambda=\da{}H$ where $H$ is the periodic Hilbert transform \eqref{hilbertop}, the fact that $\da{2}\tilde{\Gamma}_i\cdot\tilde{\tau}_i=0$ for $i=1,2$ and $$\tilde{\kappa}_1-\tilde{\kappa}_2=\frac{(L_2-L_1)(L_2\e{2}+L_1L_2+L_1\e{2})}{L_2\e{3}L_1\e{2}}\da{2}\tilde{\Gamma}_1\cdot\tilde{n}_1+\frac{L_1}{L_2\e{3}}\da{2}f\cdot \tilde{n}_1+\frac{1}{L_2\e{3}}\da{2}\tilde{\Gamma}_2\cdot(\da{}f)\e{\perp},$$ we can write $J_{21}\e{1}$ as
\begin{align*}
&J_{21}\e{1}=-\frac{L_2}{2L_1\pi\mu}\int_{S\e{1}}\da{2}f(\eta)\cdot \tilde{n}_1(\eta)\Lambda(\tilde{\kappa}_1-\tilde{\kappa}_2)(\eta)d\eta=-(L_2-L_1)\frac{L_2\e{2}+L_1L_2+L_1\e{2}}{2L_1\e{3}L\e{2}\pi\mu}\int_{S\e{1}}\da{2}f(\eta)\cdot\tilde{n}_1(\eta)H(\da{3}\tilde{\Gamma}_1\cdot\tilde{n}_1)(\eta)d\eta\\
&-\frac{1}{2L_2\e{3}\pi\mu}\int_{S\e{1}}\da{2}f(\eta)\cdot \tilde{n}_1(\eta)\Lambda(\da{2}f\cdot \tilde{n}_1)(\eta)d\eta-\frac{1}{2L_1L_2\e{2}\pi\mu}\int_{S\e{1}}\da{2}f(\eta)\cdot \tilde{n}_1(\eta)H(\da{3}\tilde{\Gamma}_2\cdot(\da{}f)\e{\perp})(\eta)d\eta\\
&-\frac{1}{2L_1L_2\e{2}\pi\mu}\int_{S\e{1}}\da{2}f(\eta)\cdot \tilde{n}_1(\eta)H(\da{2}\tilde{\Gamma}_2\cdot(\da{2}f)\e{\perp})(\eta)d\eta\\
&\leq-\frac{L_2}{2L_1\pi\mu}\int_{S\e{1}}(\Lambda\e{\frac{1}{2}}(\da{2}f\cdot \tilde{n}_1)(\eta))\e{2}d\eta+C\abs{L_1-L_2}\norm{\da{2}f}_{L\e{2}(S\e{1})}\norm{\da{3}\tilde{\Gamma}_1}_{L\e{2}(S\e{1})}+C\norm{\da{2}f}_{L\e{2}(S\e{1})}\norm{\da{}f}_{L\e{\infty}(S\e{1})}\norm{\da{3}\tilde{\Gamma}_2}_{L\e{2}(S\e{1})}\\
&+C\norm{\da{2}f}\e{2}_{L\e{2}(S\e{1})}\norm{\da{2}\tilde{\Gamma}_2}_{L\e{\infty}(S\e{1})}\leq -C\norm{\Lambda\e{\frac{1}{2}}(\da{2}f\cdot \tilde{n}_1)}_{L\e{2}(S\e{1})}\e{2}+C\abs{L_1-L_2}\e{2}+C\norm{f}\e{2}_{H\e{2}(S\e{1})}
\end{align*}

Thus, we can conclude that 
\begin{equation}\label{estimf}
\frac{d}{dt}(\norm{f}_{L\e{2}(S\e{1})}\e{2}+\norm{\da{2}f}\e{2}_{L\e{2}(S\e{1})})\leq C(\norm{f}_{L\e{2}(S\e{1})}\e{2}+\norm{\da{2}f}\e{2}_{L\e{2}(S\e{1})}+\abs{L_1-L_2}\e{2})
\end{equation}
On the other hand, we can estimate
 $\frac{d}{dt}\abs{L_1-L_2}\e{2}$ using \eqref{reformGFBP6} and similar arguments. Then 
\begin{equation}\label{estimL1L2}
\frac{d}{dt}\abs{L_1-L_2}\e{2}\leq C(\norm{f}_{L\e{2}(S\e{1})}\e{2}+\norm{\da{2}f}\e{2}_{L\e{2}(S\e{1})}+\abs{L_1-L_2}\e{2})
\end{equation}

Adding \eqref{estimf} and \eqref{estimL1L2} we obtain \eqref{estimunic}. Using that $\norm{f}\e{2}_{H\e{2}(S\e{1})}+\abs{L_1-L_2}\e{2}=0$ at $t=0$, we then deduce that $\tilde{\Gamma}_1=\tilde{\Gamma}_2$ and $L_1=L_2$ hence the sought-for uniqueness result follows.
\end{proof}

\begin{nota}
Theorem \ref{thexistreform} and Proposition \ref{reformulacion} imply Theorem \ref{thexist}.
\end{nota}

\subsection{Positivity of the thinness function $h$}

In order to conclude the proof of the local well posedness of the Geometric Free Boundary Problem we show that $h$ is positive and bounded as long as the curve parametrized by means of $\Gamma$ is well defined.

\begin{thm}\label{ththickness}
Let $h_0\in H_p\e{k}([0,L_0])$, $\Gamma_{0}(\alpha)\in H^{k}([0,L_0])$ for some $k\ge 3$, $L_0=L(0)>0$, $\mathcal{F}(\Gamma_{0})(\alpha,\beta)\in L^{\infty}$. Let us assume also that $h_0(\xi)>0$ for all $\xi\in[0,L_0]$. Let $(\Gamma,L)$ be the unique solution of the problem \eqref{GFBP1}, \eqref{GFBP2} and \eqref{GFBP4} in the domain $\mathcal{G}_T$ defined in \eqref{GT} with $T=T(\Gamma_{0},L_0)>0$ and the regularity propierties stated in Theorem \ref{thexist}. Then, there exists a solution of \eqref{GFBP3} with $h(\cdot,0)=h_0$, $h\in H\e{1}(\mathcal{G}_T)$ and $h(\cdot,t)\in H\e{k}_p([0,L(t)])$ of $t\in[0,T]$. Moreover, $h(\xi,t)>0$ for all $(\xi,t)\in\mathcal{G}_T$.
\end{thm}
\begin{nota}
Notice that the thickness function $h$ is smooth and positive as long as the solution of the problem \eqref{GFBP1}, \eqref{GFBP2} and \eqref{GFBP4} is well defined. Therefore, the arguments made to approximate the original free boundary problem for the Stokes equation (c.f. \eqref{stokesepsilon}-\eqref{cond7} and \eqref{ecevol3}, \eqref{ecevol1})  in Section \ref{S3} by means of the Geometrical Free Boundary problem are self consistent.
\end{nota}
\begin{proof}
We are going to solve \eqref{GFBP3} using the characteristic method for a first order PDE. 
To this end we define $\alpha=\alpha(t,\gamma)$ for $\gamma\in[0,L_0]$ and $t\in[0,T]$ by means of the family of differential equations

\begin{align}
\label{eqcar2}
&\frac{\partial\alpha}{\partial t}(t,\gamma)=(u_0(\Gamma)\cdot\tau)(\alpha(t,\gamma),t)-\psi_0(\alpha(t,\gamma),t),\quad\alpha(0,\gamma)=\gamma\quad\textit{for}\quad\gamma\in[0,L_0]
\end{align}
where $u_0$ is as in Theorem \ref{thexist}. The function $\alpha$ is well defined by means of \eqref{eqcar2} due to the regularity properties for $u_0$ in Theorem \ref{thexist} (ii) as well as the fact that the function $\psi_0$ in \eqref{GFBP2} is Lipschitz continuous in the variable $\xi$. 

It is well know that if a solution of \eqref{GFBP3} exists we have 

\begin{displaymath}
\frac{dh}{dt}(\alpha(t,\gamma),t)=-(n\cdot\nabla u_0(\Gamma) n)(\alpha(t,\gamma),t)h(\alpha(t,\gamma),t)
\end{displaymath}

Then $h$ would be given by

\begin{equation}\label{eqhexplicit}
h(\alpha(t,\gamma),t)=h_0(\gamma)\exp(-\int_0\e{t}(n\cdot\nabla u_0(\Gamma) n)(\alpha(r,\gamma),r)dr)
\end{equation}

It turns out that the function \eqref{eqhexplicit} allows to define a solution of \eqref{GFBP3} if the mapping $\gamma\to\alpha(t,\gamma)$ is bijective from $[0,L_0]$ to an interval $[\alpha(t,0),\alpha(t,L_0)]$ with 
\begin{equation}\label{restaalfa}
\alpha(t,L_0)-\alpha(t,0)=L(t)
\end{equation}
 for $t\in[0,T]$. The mapping $\gamma\to\alpha(t,\gamma)$ is strictly monotone for $\gamma\in[0,L_0]$ and $t\in[0,T]$ due to the fact that \eqref{eqcar2} is a first order ordinary differential equation in the $t$ variable in which the right hand side is Lipschitz in $\alpha$.

Thus, it only remains to prove \eqref{restaalfa}.
In order to do that we define a function $\phi(t)$ by means of
\begin{equation*}
\phi(t)=\alpha(t,L_0)-\alpha(t,0)
\end{equation*}

Therefore, using \eqref{eqcar2} we can see that the function $\phi$ solves the following first order differential equation
\begin{equation*}
\frac{\partial\phi}{\partial t}(t)=(u_0(\Gamma)\cdot\tau)(\alpha(t,0)+\phi(t),t)-(u_0(\Gamma)\cdot\tau)(\alpha(t,0),t)-(\psi_0(\alpha(t,0)+\phi(t),t)-\psi_0(\alpha(t,0),t))
\end{equation*}
On the other hand, since $\Gamma$ is periodic with period $L(t)$ we can rewrite \eqref{reformGFBP6} by
\begin{equation*}
\frac{\partial L}{\partial t}(t)=(u_0(\Gamma)\cdot\tau)(\alpha(t,0)+L(t),t)-(u_0(\Gamma)\cdot\tau)(\alpha(t,0),t)-(\psi_0(\alpha(t,0)+L(t),t)-\psi_0(\alpha(t,0),t))
\end{equation*}
where we use that $(u_0(\Gamma)\cdot\tau)(\alpha(t,0)+L(t),t)-(u_0(\Gamma)\cdot\tau)=0$ due to the periodicity of $u_0\circ\Gamma$ and $\tau$.
Therefore, $\phi$ and $L$ solve the same ODE with the same initial data $\phi(0)=L_0$. Then it follows from standard uniqueness result that $\phi(t)=L(t)$ and the Theorem follows.

\end{proof}
\section*{Acknowledgement}
The authors acknowledge support through the CRC 1060 (The Mathematics of Emergent
Effects) that is funded through the German Science Foundation (DFG), and the Deutsche Forschungsgemeinschaft (DFG, German Research Foundation) under Germany 's Excellence Strategy – EXC-2047/1 – 390685813.

\bibliography{referencias}
\bibliographystyle{plain}
\end{document}